%% file: main.tex
\pdfoutput=1
\documentclass[toc]{mathprint} %
\usepackage{macros}

\title[Visibility problem]{Visibility problem in the plane}

\author
  {Tuomas Orponen}
  {Department of Mathematics and Statistics, University of Jyväskylä, P.O.\ Box 35 (MaD), FI-40014 University of Jyväskylä, Finland}
  {tuomas.t.orponen@jyu.fi}

\author
  {Alex Rutar}
  {Department of Mathematics and Statistics, University of Jyväskylä, P.O.\ Box 35 (MaD), FI-40014 University of Jyväskylä, Finland}
  {alex@rutar.org}

\begin{document}
\begin{abstract}
    We disprove the visibility conjecture in the plane and prove the sharp upper bound for the almost-sure dimension of visible parts.

    Precisely, let $K \subset \mathbb{R}^{2}$ be a compact set.
    For $\sigma \in S^{1}$, let $\mathrm{Vis}_{\sigma}(K) \subset K$ be the visible part of $K$ in direction $\sigma$.
    We prove that $\operatorname{dim}_{\mathrm{H}} \mathrm{Vis}_{\sigma}(K) \leq \tfrac{3}{2}$ for $\mathcal{H}^{1}$ almost every $\sigma \in S^{1}$.
    This is sharp: we construct a compact set $K \subset \R^{2}$ such that $\Hd \vis_{\sigma}(K)\geq \tfrac{3}{2}$ for all $\sigma \in S^1$.
\end{abstract}

\section{Introduction}
\subsection{Visible parts and their dimensions}
We start by introducing basic notation.
\begin{notation}\label{not1}
    For $z \in \R^{2}$ and $\sigma \in [-1,1]$, let $\ell^{+}_{z,\sigma} \coloneqq \{z + (r,\sigma r) : r \geq 0\}$ be the \emph{ray with slope $\sigma$ starting from $z$}.
    Similarly, let $\ell_{z,\sigma} \coloneqq \{z + (r,\sigma r) : r \in \R\}$ be the \emph{line with slope $\sigma$ passing through $z$}.
    For $\omega = (z,\sigma) \in \R^{2} \times [-1,1]$, we also write $\ell^{+}_{\omega} \coloneqq \ell^{+}_{z,\sigma}$ and $\ell_{\omega} = \ell_{z,\sigma}$.
\end{notation}

\begin{definition}[$\vis_{\sigma}(K)$]\label{def:visiblePart}
    Let $K \subset \R^{2}$ be a compact set, and let $\sigma \in [-1,1]$.
    The \emph{visible part of $K$ in direction $\sigma$} is the set
    \begin{equation*}
        \vis_{\sigma}(K) \coloneqq \{z \in K : K \cap \ell^{+}_{z,\sigma} = \{z\}\}.
    \end{equation*}
    The \emph{bi-visible part of $K$ in direction $\sigma$} is the set
    \begin{equation*}
        \vis_{\sigma}^{2}(K) \coloneqq \{z \in K : K \cap \ell_{z,\sigma} = \{z\}\}.
    \end{equation*}
\end{definition}
\begin{remark}
    Evidently $\vis^{2}_{\sigma}(K) \subset \vis_{\sigma}(K)$ for all $\sigma \in [0,1)$.
    This inclusion can be strict for all $\sigma \in [0,1)$.
    For example $\vis_{\sigma}(S^{1})$ is an arc for all $\sigma \in [0,1)$ whereas $\#\vis_{\sigma}^{2}(S^{1}) \equiv 2$.
\end{remark}
The \emph{visibility problem} in fractal geometry is the quest for finding upper bounds on the Hausdorff dimension of $\vis_{\sigma}(K)$, typically when $K$ is compact.
The first positive results were published by Järvenpää, Järvenpää, MacManus, and O'Neil \cite{zbl:1026.28002} in 2003.
They showed that if $K \subset \R^{2}$ is (i) a quasi-circle or (ii) a connected self-similar set without rotations, then $\Hd \vis_{\sigma}(K) \leq 1$ for all $\sigma \in [0,1)$.
They also showed that the visible parts of graphs $\Gamma \subset \R^{2}$ of continuous functions $[0,1] \to [0,1]$ are at most $1$-dimensional in all directions, except possibly one: the (bi-)visible part of $\Gamma$ in the vertical direction is evidently $\Hd \Gamma$-dimensional, and $\Hd \Gamma$ can take any value in $[1,2]$.

While no ``visibility conjecture'' was formulated in \cite{zbl:1026.28002}, it soon turned into common belief that if $K \subset \R^{2}$ is compact, then $\Hd \vis_{\sigma}(K) \leq 1$ for almost all $\sigma \in [0,1)$.
This has been explicitly suggested (at least) in \cite[Problem~11]{zbl:1049.28007}, \cite[Conjecture~1.3]{zbl:1267.28006}, \cite[Section 10]{zbl:1338.28007}, and \cite[Section~7]{arxiv:2602.22002}.
The hypothesis has been confirmed for certain special families of sets such as quasi-circles, fractal percolation, and some self-similar and self-affine sets (see \cite{zbl:1251.28007,zbl:1267.28006,zbl:1026.28002,zbmath:7541839,zbl:1479.28008}).
There is also a strong partial result for planar continua due to O'Neil \cite{zbl:1152.28318}.
In \cite{zbl:1119.28003}, it is shown that if $s \in (1,2]$, and $K \subset \R^{2}$ is a compact set with $\mathcal{H}^{s}(K) < \infty$, then $\mathcal{H}^{s}(\vis_{\sigma}(K)) = 0$ for a.e.\ $\sigma \in [0,1)$.

For general compact sets, the first partial result appeared in \cite{zbl:1561.28083}, and this first a.e.\ upper bound ($1.99$) was subsequently lowered by Matheus and D{\k a}browski \cite{zbmath:8092718,zbl:1484.28010}.

In this paper, we prove the sharp $\mathcal{H}^{1}$ almost sure upper bound for general compact sets. In fact, our method yields an $\mathcal{H}^{\alpha}$ almost sure upper bound for all $\alpha \in [0,1]$, but we do not know if that bound is sharp for $\alpha \in (0,1)$, see \cref{q1}.
\begin{itheorem}\label{main}
    Let $K \subset \R^{2}$ be compact and $\alpha\in[0,1]$.
    Then $\Hd\vis_\sigma^2(K) \leq\Hd\vis_{\sigma}(K) \leq 2 - \tfrac{\alpha}{2}$ for $\mathcal{H}^{\alpha}$-a.e.\ $\sigma \in [-1,1]$.
    On the other hand, there exists a compact set $K \subset \R^{2}$ such that $\Hd \vis_{\sigma}(K) \geq\Hd\vis_\sigma^2(K) \geq \tfrac{3}{2}$ for all $\sigma \in [-1,1]$.
\end{itheorem}
We will prove the upper bound in \cref{main} by reducing it to a $\delta$-discretised problem concerning lower bounds on ``ray incidences'', discussed in \cref{ss:lb-ray}.
The set $K$ witnessing the $\tfrac{3}{2}$-bound is constructed by an iterative process, whose main ingredient is the ``basic building block'' described in \cref{p:bbb}.
The proof can be found in \cref{s:lower-bound}.
To pass from angles in $[-1,1]$ to the full circle (as claimed in the abstract), see \cref{r:s1} for the short argument suggested to us by E.\ J\"arvenp\"a\"a.

\begin{remark}
    While the conclusions about $\vis_{\sigma}(K)$ are the main content of \cref{main}, proving the lower bound $\Hd \vis_{\sigma}^{2}(K) \geq \tfrac{3}{2}$ has no additional cost, and it sheds light on a problem proposed in \cite[Section 6.4]{MR3617376}: for Lebesgue typical $\sigma \in [-1,1]$, how much of a compact set $K \subset \R^{2}$ can be covered by ``$s$-light'' lines $\ell \subset \R^{2}$ with slope $\sigma$, satisfying $\Hd (K \cap \ell) \leq s$, where $s < \Hd K - 1$?
    \cref{main} shows that a $\tfrac{3}{2}$-dimensional set can be covered by such ``$s$-light'' lines, for every $\sigma \in [-1,1]$, and every $s > 0$.
\end{remark} 
While \cref{main} resolves the visibility problem as commonly stated in the literature, the problem remains open in the connected and Ahlfors regular cases:
\begin{question}\label{problem2}
    Assume that $K \subset \R^{2}$ is compact and either (a) connected, or (b) Ahlfors $s$-regular with $s \in [0,2]$.
    Is it true that $\Hd \vis_{\sigma}(K) \leq \min\{1,s\}$ for a.e.\ $\sigma \in [-1,1]$?
\end{question}
\begin{remark}
    The set constructed for the lower bound in \cref{main} is far from Ahlfors regular (let alone connected).
    There is partial progress towards \cref{problem2}: D{\k a}browski has proved that if $K \subset \R^{2}$ is compact and Ahlfors $s$-regular with $s \in [1,2]$, then 
    \begin{displaymath} \Hd \vis_{\sigma}(K) \leq s - a(s - 1) \end{displaymath}
    for a.e.\ $\sigma \in [-1,1]$, where $a > 0.183$ is absolute.
    In particular, $\mathcal{H}^{1}$ a.e.\ visible parts of Ahlfors $\tfrac{3}{2}$-regular sets are strictly less than $\tfrac{3}{2}$-dimensional.

    Remarkably, \cref{problem2} is open for general self-similar sets, even without rotations.
\end{remark}
We also explicitly record the question about the sharpness of Theorem~ref{main} for $\alpha \in (0,1)$:
\begin{question}\label{q1}
    Let $\alpha\in(0,1)$.
    Does there exists a compact set $K\subset\R^2$ along with a set of directions $\Sigma \subset[-1,1]$ satisfying $\mathcal{H}^\alpha(\Sigma) > 0$ such that $\Hd\vis_\sigma^2(K) \geq 2 - \tfrac{\alpha}{2}$  for all $\sigma\in\Sigma$?
\end{question}
The following ``endpoint'' problem is also still open:
\begin{question}
    Does there exist a compact set $K \subset \R^{2}$ such that $\mathcal{H}^{3/2}(\vis_{\sigma}(K)) > 0$ for all $\sigma \in [-1,1]$?
    By the results in \cite{zbl:1119.28003}, this is only conceivable if $\mathcal{H}^{3/2}(K) = \infty$.
\end{question}

\subsection{Higher dimensions?} Visible parts can readily be defined and studied in higher dimensions, and this is indeed done in many of the papers in previous literature, cited above. We thank T. Keleti for pointing out that if $K \subset \R^{2}$ is the set produced by \cref{main} with $\Hd \vis_{\sigma}^{2}(K) \geq \tfrac{3}{2}$ for all $\sigma \in [-1,1]$, then the product $K \times [0,1]^{d - 2}$ gives an example of a compact set in $\R^{d}$ with at least $(d - \tfrac{1}{2})$-dimensional bi-visible parts in an open set of directions in $\R^{d}$. We give the details below.

First, we extend \cref{not1} to all dimensions $d \geq 2$ by setting 
\begin{displaymath} \ell_{\mathbf{z},\bm{\sigma}} \coloneqq \{\mathbf{z} + (r,\bm{\sigma} r) : r \in \R\}, \qquad \mathbf{z} \in \R^{d}, \, \bm{\sigma} \in [-1,1]^{d - 1}. \end{displaymath}
For $\bm{\sigma} \in [-1,1]^{d - 1}$ and $K \subset \R^{d}$, we define the bi-visible part $\vis_{\bm{\sigma}}^{2}(K)$ exactly as in \cref{def:visiblePart}. Rays and visible parts are defined similarly. 

Now, let $K \subset \R^{2}$ be the compact set produced by \cref{main}, and write 
\begin{displaymath} K^{d} \coloneqq K \times [0,1]^{d - 2}, \qquad d \geq 2. \end{displaymath}
For $\bm{\sigma} = (\sigma,\sigma_{2},\ldots,\sigma_{d - 1}) \in [-1,1]^{d - 1}$, write $\pi(\bm{\sigma}) \coloneqq \sigma \in [-1,1]$. We claim that 
\begin{equation}\label{higherDimensions} \vis_{\pi(\bm{\sigma})}^{2}(K) \times [0,1]^{d - 2} \subset \vis_{\bm{\sigma}}^{2}(K^{d}), \qquad \bm{\sigma} \in [-1,1]^{d - 1}.  \end{equation} 
This will show that $\Hd \vis_{\bm{\sigma}}^{2}(K^{d}) \geq \tfrac{3}{2} + (d - 2) = d - \tfrac{1}{2}$ for all $\bm{\sigma} \in [-1,1]^{d - 1}$.

To prove \cref{higherDimensions}, fix $\bm{\sigma} = (\sigma,\sigma_{2},\dots,\sigma_{d - 1}) \in [-1,1]^{d - 1}$ and $\mathbf{z} = (z,z_{3},\ldots,z_{d}) \in \vis_{\sigma}^{2}(K) \times [0,1]^{d - 2}$, where $\sigma = \pi(\bm{\sigma})$ and $z \in \vis_{\sigma}^{2}(K)$.

To prove that $\mathbf{z} \in \vis_{\bm{\sigma}}^{2}(K^{d})$, we need to fix $\mathbf{w} \in K^{d} \cap \ell_{\mathbf{z},\bm{\sigma}}$, and show that $\mathbf{w} = \mathbf{z}$. To this end, first use the definition of $\mathbf{w} \in \ell_{\mathbf{z},\bm{\sigma}}$ to express
\begin{displaymath} \mathbf{w} = \mathbf{z} + (r,\bm{\sigma}r) = (z + (r,\sigma r),z_{3} + r\sigma_{2},\ldots,z_{d} + r\sigma_{d - 1}). \end{displaymath}
Next, recalling that $\mathbf{w} \in K^{d} = K \times [0,1]^{d - 2}$, we have $z + (r,\sigma r) \in K$. Since evidently $z + (r,\sigma r) \in \ell_{z,\sigma}$, it follows from $z \in \vis^{2}_{\sigma}(K)$ that $r = 0$. Thus $\mathbf{w} = \mathbf{z}$, as claimed.

The method of proving the $\tfrac{3}{2}$-upper bound in \cref{main} does not immediately extend to higher dimensions. The following question seems natural:
\begin{question} Let $K \subset \R^{d}$ be compact. Is it true that $\Hd \vis_{\bm{\sigma}}(K) \leq d - \tfrac{1}{2}$ for $\mathcal{H}^{d - 1}$ almost every $\bm{\sigma} \in [-1,1]^{d - 1}$? \end{question}

\subsection{Lower bounds on ray incidences}\label{ss:lb-ray}
In this section we discuss the technical tools behind the upper bound $\Hd \vis_{\sigma}(K) \leq 2 - \tfrac{\alpha}{2}$ in \cref{main}.
In \cite{zbmath:8038252}, Cohen, Pohoata and Zakharov proved a beautiful result guaranteeing many $\delta$-discretised point-line incidences.
The $\delta$-discretised result behind \cref{main} is a variant of Cohen, Pohoata and Zakharov's theorem.
We now proceed to state both the original result, and the variant we will prove in this paper.
\begin{definition}[Phase space and Frostman $(\delta,\alpha,\beta,C)$-sets]\label{def:PhaseSpace}
    Let $\Omega \coloneqq [-1,1]^{3}$ be the \emph{phase space}.
    For $u,v,w \in (0,1]$ and $(a,b,\sigma) \in \Omega$ we define the \emph{phase space rectangle}
    \begin{equation*}
        \mathbf{R}_{u \times v \times w}(a,b,\sigma) \coloneqq \{(a + r_{1},b + \sigma r_{1} + r_{2},\sigma + r_{3}) : (r_{1},r_{2},r_{3}) \in [-u,u] \times [-v,v] \times [-w,w]\}.
    \end{equation*}
    Let $\alpha,\beta \geq 0$, $\delta \in (0,1]$, and $C > 0$.
    A finite set $\mathbf{X} \subset \Omega$ is called a \emph{Frostman $(\delta,\alpha,\beta,C)$-set} if $\mathbf{X}$ is non-empty, and
    \begin{equation}\label{frostmanCondition}
        |\mathbf{X} \cap \mathbf{R}_{u \times uw \times w}(a,b,\sigma)| \leq Cu^{\alpha}w^{\beta}|\mathbf{X}|, \qquad (a,b,\sigma) \in \R^{3}, \, u,w \in (0,1], \, uw \geq \delta.
    \end{equation}
\end{definition}

\begin{notation}
    For $\omega = (a,b,c) \in \Omega$, we write $z_{\omega} \coloneqq (a,b)$.
    We also remind the reader of the lines $\ell_{\omega}$ and rays $\ell_\omega^+$ defined in \cref{not1}.
    We will also write $P[\mathbf{X}] \coloneqq \{z_{\omega} : \omega \in \mathbf{X}\}$ for the ``base point projection'' of $\mathbf{X}$.
\end{notation}

\begin{definition}[Line and ray incidences]\label{def:incidences}
    Let $\mathbf{X} \subset \Omega$ and $\delta \in (0,1]$.
    For $\mathbf{A},\mathbf{B} \subset \mathbf{X}$, we define the \emph{line incidences}
    \begin{equation*}
        \mathcal{I}_{\delta}(\mathbf{A},\mathbf{B}) \coloneqq |\{(\omega_{1},\omega_{2}) \in \mathbf{A} \times \mathbf{B} : \dist(z_{\omega_{2}},\ell_{\omega_{1}}) \leq \delta\}|.
    \end{equation*}
    We also define the \emph{ray incidences} by replacing $\ell_{\omega_{1}}$ by $\ell_{\omega_{1}}^{+}$:
    \begin{equation*}
        \mathcal{I}_{\delta}^{+}(\mathbf{A},\mathbf{B}) \coloneqq |\{(\omega_{1},\omega_{2}) \in \mathbf{A} \times \mathbf{B} : \dist(z_{\omega_{2}},\ell_{\omega_{1}}^{+}) \leq \delta\}|.
    \end{equation*}
    We abbreviate $\mathcal{I}_{\delta}(\mathbf{A},\mathbf{A}) \eqqcolon \mathcal{I}_{\delta}(\mathbf{A})$  and $\mathcal{I}_{\delta}^{+}(\mathbf{A},\mathbf{A}) \eqqcolon \mathcal{I}_{\delta}^{+}(\mathbf{A})$.
\end{definition}
Here is Cohen, Pohoata and Zakharov's theorem \cite[Theorem 1.9]{zbmath:8038252} in its original form:
\begin{theorem}[\cite{zbmath:8038252}]\label{thm:CPZ}
    Let $\alpha,\beta \in (1,2]$ with $\alpha + \beta > 3$, and let $\varepsilon > 0$.
    Then, there exist $\eta = \eta(\alpha,\beta,\varepsilon) > 0$ and $\delta_{0} = \delta_{0}(\alpha,\beta,\varepsilon) > 0$ such that the following holds for all $\delta \in 2^{-\N} \cap (0,\delta_{0}]$:

    Let $\mathbf{X} \subset \Omega$ be a Frostman $(\delta,\alpha,\beta,\delta^{-\eta})$-set.
    Then,
    \begin{equation}\label{form35}
        \mathcal{I}_{\delta}(\mathbf{X}) \geq \delta^{1 + \varepsilon}|\mathbf{X}|^{2}.
    \end{equation}
\end{theorem}
Roughly speaking, \cref{thm:CPZ} says that a $(\delta,\alpha,\beta)$-Frostman set $\mathbf{X} \subset \Omega$ spans as many $\delta$-discretised point-line incidences as it would if the points and lines were placed ``at random''.
Starting from a counter assumption to our main result, \cref{main}, with $s > 2 - \tfrac{\alpha}{2}$, it is relatively straightforward to construct an appropriate Frostman $(\delta,s,s + \alpha - 1,C)$-set $\mathbf{X} \subset \Omega$ to which one is tempted to apply \cref{thm:CPZ}.
The construction of $\mathbf{X}$ is accomplished in \cref{s3}.
However, a direct application of \cref{thm:CPZ} fails to prove \cref{main} for two distinct reasons:
\begin{enumerate}[nl]
    \item We need a lower bound for $\mathcal{I}_{\delta}^{+}(\mathbf{X})$ in place of $\mathcal{I}_{\delta}(\mathbf{X})$.
    \item The $\delta^{\varepsilon}$-factor in \cref{form35} is unacceptable for us.
        Roughly speaking, this lower bound could be used to prove the existence of two $\delta^{\varepsilon}$-separated incidences.
        For \cref{main}, we need to upgrade $\delta^{\varepsilon}$ to a constant independent of $\delta$.
        However, this conclusion may fail for Frostman $(\delta,\alpha,\beta,\delta^{-\eta})$-sets, as in \cref{thm:CPZ}, where the Frostman constant depends on $\delta$.

\end{enumerate}
In the application arising from \cref{main}, fortunately, the Frostman constant of $\mathbf{X}$ is independent of $\delta$.
So, we need to establish a variant of \cref{thm:CPZ} where both the hypothesis and conclusion are stronger.
This is \cref{thm:ray-inc-endp} below.
\begin{itheorem}\label{thm:ray-inc-endp}
    Let $\alpha,\beta \in (1,2]$ with $\alpha + \beta > 3$.
    Then, there exists $\kappa = \kappa(\alpha, \beta) \geq 1$ such that the following holds for all $C \geq 1$ and all $\delta \in 2^{-\N} \cap (0, \kappa^{-1}C^{-\kappa}]$:

    Let $\mathbf{X} \subset \Omega$ be a Frostman $(\delta,\alpha,\beta,C)$-set.
    Then there exist $\mathbf{A},\mathbf{B} \subset \mathbf{X}$ such that
    \begin{equation*}
        \dist(P[\mathbf{A}],P[\mathbf{B}]) \gtrsim_{\alpha,\beta} C^{-\kappa} \quad \text{and} \quad \mathcal{I}_{\delta}^+(\mathbf{A},\mathbf{B}) \gtrsim_{\alpha,\beta} C^{-\kappa} \cdot \delta |\mathbf{X}|^{2}.
    \end{equation*}
    In particular, there exists a pair $(\omega_{1},\omega_{2}) \in \mathbf{A} \times \mathbf{B}$ such that
    \begin{equation}\label{form38}
        |z_{\omega_{1}} - z_{\omega_{2}}| \gtrsim_{\alpha,\beta} C^{-\kappa} \quad \text{and} \quad \dist(z_{\omega_{2}},\ell_{\omega_{1}}^+) \leq \delta.
    \end{equation}
\end{itheorem}

\begin{remark}
    In particular, \cref{thm:ray-inc-endp} implies that $\mathcal{I}_{\delta}^+(\mathbf{X}) \gtrsim_{\alpha,\beta,C} \delta|\mathbf{X}|^{2}$.
    It is conceivable (but non-trivial) that this alone guarantees $c$-separated incidences (with $c \sim_{\alpha,\beta,C} 1$).
    In any case, proving the stated ``bi-partite'' version of \cref{thm:ray-inc-endp} poses no additional difficulties, and gives the $\rho$-separation of the incidences for free.
\end{remark}

Perhaps surprisingly, we are able to formally deduce \cref{thm:ray-inc-endp} from the following incidence lower bound for rays which still has the $\delta^{\varepsilon}$-factor in the conclusion.
The statement is the same as \cite[Theorem 1.9]{zbmath:8038252}, except with rays in place of lines.
\begin{itheorem}\label{thm:CPZ2}
    Let $\alpha,\beta \in (1,2]$ with $\alpha + \beta > 3$, and let $\varepsilon > 0$.
    Then, there exist $\eta = \eta(\alpha,\beta,\varepsilon) > 0$ and $\delta_{0} = \delta_{0}(\alpha,\beta,\varepsilon) > 0$ such that the following holds for all $\delta \in 2^{-\N} \cap (0,\delta_{0}]$:

    Let $\mathbf{X} \subset \Omega$ be a Frostman $(\delta,\alpha,\beta,\delta^{-\eta})$-set.
    Then,
    \begin{equation*}
        \mathcal{I}^{+}_{\delta}(\mathbf{X}) \geq \delta^{1 + \varepsilon}|\mathbf{X}|^2.
    \end{equation*}
\end{itheorem}

\subsection{Outline of the paper}
In the proof of \cref{main}, we begin with the upper bound and then proceed with the construction for the lower bound.

As mentioned earlier, the main upper bound on visible parts (in \cref{main}) is deduced from \cref{thm:ray-inc-endp}.
The first step is to find an appropriate $\delta$-discretised version of \cref{main}.
This is accomplished in \cref{s1}.
The proof of the $\delta$-discretised counterpart is then completed in \cref{s3} by identifying an appropriate Frostman set $\mathbf{X} \subset \Omega$, and applying \cref{thm:ray-inc-endp}.

Most of the proof of the upper bound is occupied by the proof of \cref{thm:ray-inc-endp}.
There are two distinct steps.
The first one is to obtain the (strictly stronger) version of \cref{thm:CPZ} where $\mathcal{I}_{\delta}(\mathbf{X})$ is simply replaced by $\mathcal{I}_{\delta}^{+}(\mathbf{X})$, as stated in \cref{thm:CPZ2}.
For those familiar with Cohen, Pohoata, and Zakharov's proof of \cref{thm:CPZ}, we mention that the least trivial step in the proof of our extension is to prove a version of the ``high-low lemma'' for rays in place of lines; see \cref{thm:highLowRays}.
In the corresponding proof of the high-low lemma for lines (\cite[Theorem 1.7]{zbmath:8038252}), a certain orthogonality property for pairs of lines is used, which is no longer true for (all) pairs of rays.
The solution is to examine the pairs of rays where the orthogonality relation may fail, and prove that there are not too many of them.

\cref{thm:ray-inc-endp} is formally deduced from \cref{thm:CPZ2}.
The idea is to first apply \cref{thm:CPZ2} at a suitable intermediate scale $\Delta = \Delta(\alpha,\beta,C) > 0$.
Then, one applies the high-low inequality (for rays) to show that many $\Delta$-separated $\Delta$-discretised point-ray incidences also guarantee many $\Delta$-separated $\delta$-discretised point-ray incidences with $\delta \ll \Delta$.

To summarise, the logical structure of the proof of the upper bound is as follows:
\begin{equation*}
    \text{\cref{thm:CPZ2}}\!\quad\!\Longrightarrow\!\quad\!\text{\cref{thm:ray-inc-endp}}\!\quad\!\Longrightarrow\!\quad\!\text{\cref{p:vis-disc}}\!\quad\!\Longrightarrow\!\quad\!\text{\cref{main}}.
\end{equation*}
The sections follow the implications in reverse order.
We prove \cref{thm:CPZ2} in \cref{s:non-endp}; the first implication is established in \cref{s:endp}; and the final two implications are established in \cref{s1}.

The proof of the lower bound construction in \cref{main} can be found in \cref{s:lower}.
\cref{s2} contains $\delta$-discretised statements (the main result being \cref{thm1}) required for the construction, and the actual construction takes place in \cref{s:lower-bound}.
\cref{appA} contains the construction of an auxiliary ``basic building block''.
The details of this object are explained in \cref{s2}.

\subsection*{Notation}
We use the convention that $\N = \{0,1,2,\ldots,\}$.
For $r \in 2^{-\N}$, and $K \subset \R^{d}$ (including $K \subset \Omega = [-1,1]^{3}$), the notation $\mathcal{D}_{r}(K)$ stands for the dyadic cubes of side-length $r$ which intersect $K$.
We also write $\mathcal{D} \coloneqq \bigcup_{r\in 2^{-\N}} \mathcal{D}_{r}(\R^{2})$ (note that we only include in $\mathcal{D}$ dyadic cubes of side-length $\leq 1$).
The side-length of $Q \in \mathcal{D}$ is denoted $\ell(Q)$.
We abbreviate $\mathcal{D}_{r} \coloneqq \mathcal{D}_{r}([0,1)^{2})$.

For $K \subset \R^{d}$ and $r \in 2^{-\N}$, we use the notation
\begin{displaymath} |K|_{r} \coloneqq |\mathcal{D}_{r}(K)|. \end{displaymath}
The notation $|K|$ stands for the cardinality of $K$ (or $+\infty$ if $K$ is infinite).

For $K \subset \R^{d}$, and $r > 0$, the notation $[K]_{r}$ stands for the open $r$-neighbourhood of $K$. For $z_{1},z_{2} \in \R^{d}$, the notation $|z_{1} - z_{2}|$ means the Euclidean distance of $z_{1},z_{2}$.

For $A,B \geq 0$, the notation $A \lesssim_{p} B$ means that $A \leq CB$, where $C > 0$ is a constant depending only on $p$.
If there is no ``$p$'' visible, then $C$ is absolute.
For $A,B,\delta > 0$, the notation $A \lessapprox_{\delta} B$ means the same as $A \leq C(\log \tfrac{1}{\delta})^{C}B$, where $C > 0$ is absolute.

The notation $\mathcal{H}^{t}_{\infty}$, $t \geq 0$, will refer to the \emph{dyadic $t$-dimensional Hausdorff content}:
\begin{equation}\label{def:content}
    \mathcal{H}^{t}_{\infty}(K) = \inf \left\{ \sum_{i} \ell(Q_{i})^{t} : K \subset \bigcup_{i} Q_{i} \right\}, \qquad K \subset \R^{d},
\end{equation}
where the infimum runs over all families $\{Q_{i}\} \subset \mathcal{D}$.
We abbreviate $\mathcal{H}^{t}_{\infty}(\mathcal{P}) \coloneqq \mathcal{H}^{t}_{\infty}(\cup \mathcal{P})$ for $\mathcal{P} \subset \mathcal{D}$.
\begin{remark}
    The dyadic Hausdorff content of subsets of $[-1,1]^{d}$, or any other fixed bounded set, is comparable to the standard Hausdorff content (where covers by dyadic cubes are replaced by arbitrary covers), up to a constant depending only on $d$.
    With the exception of \cref{s:lower} and \cref{appA}, the reader does not need to keep in mind that we are using the dyadic variant.
    However, in \cref{s:lower} we will encounter an iterative argument where tracking the size of the Hausdorff content is important.
    This is more convenient to do for the dyadic content.
\end{remark}

If $\mathcal{P}$ consists of dyadic cubes of side-length $\geq \delta \in 2^{-\N}$ and $t \in [0,d]$, it is easy to check (and will be implicitly used) that the optimal cover in the definition of $\mathcal{H}^{t}_{\infty}(\mathcal{P})$ also consists of dyadic cubes side-length $\geq \delta$.
To wit, if $\{\mathcal{Q}_{j}\} \subset \mathcal{D}$ is an arbitrary cover of $Q \in \mathcal{D}_{\delta}$ by dyadic cubes, then $\sum_{j} \ell(Q_{j})^{t} \geq \ell(Q)^{t}$.

We will also use the following notational convention: if $\mathcal{Q} \subset \mathcal{D}_{\Delta}$ and $\mathcal{P} \subset \mathcal{D}_{\delta}$ with $\delta \leq \Delta$, we will write 
\begin{equation*}
    \mathcal{Q} \cap \mathcal{P} \coloneqq \{p \in \mathcal{P} : p \subset Q \text{ for some } Q \in \mathcal{Q}\}.
\end{equation*} 

For $\delta > 0$, a \emph{$\delta$-tube} refers to the closed $(\delta/2)$-neighbourhood of a line in $\R^{2}$.
The \emph{slope} of a non-vertical line $\ell = \{(x,y) : y = ax + b\} \subset \R^{2}$ is $\sigma(\ell) \coloneqq a$.

\section{Reduction of the visibility conjecture to an incidence lower bound}\label{s1}
In this section we reduce \cref{main} to the incidence lower bound in \cref{thm:ray-inc-endp}.
\subsection{Initial discretisation}
We first reduce \cref{main} to the $\delta$-discretised counterpart in \cref{p:vis-disc} below. We recall the (standard) notion of \emph{Frostman $(\delta,t,C)$-sets} in $\R^{d}$. These should not be confused with Frostman $(\delta,\alpha,\beta,C)$-sets from \cref{def:PhaseSpace}.

\begin{definition} Let $\delta \in 2^{-\N}$ and $C,t \geq 0$. A set $P \subset \R^{d}$ is called a \emph{Frostman $(\delta,t,C)$-set} if 
\begin{displaymath} |P \cap B(x,r)|_{\delta} \leq Cr^{t}|P|_{\delta}, \qquad x \in \R^{d}, \, \delta \leq r \leq 1. \end{displaymath} 
\end{definition} 

\begin{proposition}\label{p:vis-disc}
    For every $\alpha\in (0,1]$ and $s > 2 - \tfrac{\alpha}{2}$ and $C > 0$ there exists $\eta = \eta(C,s, \alpha) > 0$ and $\delta_{0} = \delta_{0}(C,s, \alpha) > 0$ such that the following holds for all $\delta \in 2^{-\N} \cap (0,\delta_{0}]$:

    Let $\Sigma \subset \delta \cdot \Z \cap [-1,1]$ be a non-empty Frostman $(\delta, \alpha, C)$-set.
    For each $\sigma \in \Sigma$, let $P_{\sigma} \subset (\delta \cdot \Z)^{2} \cap [-1,1]^{2}$ be a non-empty Frostman $(\delta,s,C)$-set.
    Write $P \coloneqq \bigcup_{\sigma \in \Sigma} P_{\sigma}$.
    Then, there exist $\sigma \in \Sigma$ and $z \in P_{\sigma}$ such that
    \begin{equation}\label{form7}
        \diam(P \cap [\ell^+_{z,\sigma}]_{\delta}) \geq \eta.
    \end{equation}
\end{proposition}

We now deduce \cref{main} from \cref{p:vis-disc}.
The argument is nearly the same as the proof of the implication between \cite[Theorem 3.2]{arxiv:2407.00306} and \cite[Theorem 3.1]{arxiv:2407.00306}, but we give all the details for the reader's convenience.
\begin{proof}[of \cref{main} assuming \cref{p:vis-disc}]
    The case $\alpha = 0$ of \cref{main} is trivial, so we may assume that $\alpha \in (0,1]$.
    We make a counter assumption: there exists a compact set $K \subset [0,1]^{2}$, $s \in (2 - \tfrac{\alpha}{2},2]$, and $\kappa > 0$, such that
    \begin{equation}\label{form1}
        \mathcal{H}_\infty^{\alpha}(\{\sigma \in [-1,1] : \mathcal{H}^{s}_{\infty}(\vis_{\sigma}(K)) > \kappa\}) > \kappa.
    \end{equation}
    Let $C \geq 1$ be an absolute constant, and let $\eta \coloneqq \eta(C\kappa^{-1},s,\alpha) > 0$ be the constant provided by \cref{p:vis-disc}.
    In fact, $C$ is the constant from the following version of Frostman's lemma: if $\delta \in 2^{-\N}$, $s \in (0,2]$, and $H \subset [0,1]^{2}$, then $\mathcal{D}_{\delta}(H)$ contains a non-empty Frostman $(\delta,s,C/\mathcal{H}^{s}_{\infty}(H))$-set.
    For a proof of this statement, see \cite[Proposition 3.9]{zbmath:08099050}.

    For $\sigma \in [-1,1]$, let
    \begin{equation*}
        \vis_{\sigma}^{\eta}(K) \coloneqq \{x \in K : K \cap \ell^+_{x,\sigma} \subset B(x,\eta/8)\}.
    \end{equation*}
    Recall that $B(x,\eta/8)$ is an open disc, so if $x \in \vis_{\sigma}^{\eta}(K)$ and $y \in K$ with $y \in \ell^+_{x,\sigma}$, then $|x - y| < \eta/8$.
    Clearly $\vis_{\sigma}(K) \subset \vis_{\sigma}^{\eta}(K)$, so \cref{form1} implies
    \begin{equation}\label{form2}
        \mathcal{H}_\infty^{\alpha}(\{\sigma \in [-1,1] : \mathcal{H}^{s}_{\infty}(\vis_{\sigma}^{\eta}(K)) > \kappa\}) > \kappa.
    \end{equation}
    For every $\varepsilon > 0$, we further define
    \begin{equation}\label{form36}
        \vis_{\sigma}^{\eta,\varepsilon}(K) \coloneqq \{x \in K : K \cap [\ell^+_{x,\sigma}]_{\varepsilon} \subset B(x,\eta/8)\},
    \end{equation}
    where $[\ell^+_{x,\sigma}]_{\varepsilon}$ stands for the open $\varepsilon$-neighbourhood of the ray $\ell^+_{x,\sigma}$.
    We claim that
    \begin{equation}\label{form3}
        \vis_{\sigma}^{\eta}(K) = \bigcup_{\varepsilon > 0} \vis_{\sigma}^{\eta,\varepsilon}(K), \qquad \sigma \in [-1,1].
    \end{equation}
    The inclusion $\vis_{\sigma}^{\eta,\varepsilon}(K) \subset \vis_{\sigma}^{\eta}(K)$ clearly holds for all $\varepsilon > 0$, so it suffices to prove the other inclusion.
    Fix $x \in \vis_{\sigma}^{\eta}(K)$, so $K \cap \ell^+_{x,\sigma} \subset B(x,\eta/8)$.
    If $x \notin \bigcup_{\varepsilon > 0} \vis_{\sigma}^{\eta,\varepsilon}(K)$, then for every $\varepsilon > 0$ there exists $y_{\varepsilon} \in K \cap [\ell^+_{x,\sigma}]_{\varepsilon} \, \setminus \, B(x,\eta/8)$.
    Since $K$ is compact, we may find a sequence $\{\varepsilon_{n}\}$ such that $y_{\varepsilon_{n}} \to y \in K \cap \ell^+_{x,\sigma}$.
    Since $|x - y_{\varepsilon_{n}}| \geq \eta/8$, we also have $|x - y| \geq \eta/8$.
    This contradicts the definition of $x \in \vis_{\sigma}^{\eta}(K)$ and proves \cref{form3}.

    As $\varepsilon \searrow 0$, the sets $\vis_{\sigma}^{\eta,\varepsilon}(K)$ increase to $\vis_{\sigma}^{\eta}(K)$.
    Therefore, if $\mathcal{H}^{s}_{\infty}(\vis_{\sigma}^{\eta}(K)) > \kappa$, Davies' increasing sets lemma \cite[Theorem~4]{zbl:0187.00903} implies that also $\mathcal{H}^{s}_{\infty}(\vis_{\sigma}^{\eta,\varepsilon}(K)) > \kappa$ for $\varepsilon > 0$ sufficiently small.
    This observation implies that also the sets $\{\sigma \in [-1,1] : \mathcal{H}^{s}_{\infty}(\vis_{\sigma}^{\eta,\varepsilon}(K)) > \kappa\}$ increase to $\{\sigma \in [-1,1] : \mathcal{H}^{s}_{\infty}(\vis_{\sigma}^{\eta}(K)) > \kappa\}$ as $\varepsilon \searrow 0$.
    By a second application of \cite[Theorem~4]{zbl:0187.00903}, we infer from \cref{form2} that
    \begin{equation}\label{form4}
        \mathcal{H}_\infty^{\alpha}(\{\sigma \in [-1,1] : \mathcal{H}^{s}_{\infty}(\vis_{\sigma}^{\eta,\varepsilon}(K)) > \kappa\}) > \kappa
    \end{equation}
    for $\varepsilon > 0$ sufficiently small.

    Let $\varepsilon > 0$ be so small that \cref{form4} holds, and let
    \begin{equation*}
        \delta \in 2^{-\N} \cap (0,c \min\{\varepsilon,\eta,\delta_{0}\}],
    \end{equation*}
    where $c \in (0,\tfrac{1}{16}]$ is an absolute constant to be determined, and $\delta_{0} = \delta_{0}(C\kappa^{-1},s,\alpha) > 0$ is the threshold provided by \cref{p:vis-disc}.
    Applying Frostman's lemma (or more precisely again \cite[Proposition 3.9]{zbmath:08099050}), let
    \begin{equation*}
        \mathcal{S} \coloneqq \mathcal{S}_{\delta} \subset \mathcal{D}_{\delta}(\{\sigma \in [-1,1] : \mathcal{H}^{s}_{\infty}(\vis_{\sigma}^{\eta,\varepsilon}(K)) > \kappa\})
    \end{equation*}
    be a non-empty $(\delta, \alpha, C\kappa^{-1})$-set.
    Let $\Sigma$ denote the left endpoints of dyadic intervals in $\mathcal{S}$.

    For each $\sigma \in \Sigma$, let $\bar{\sigma}$ denote an element in the dyadic interval containing $\sigma$ with $\mathcal{H}^{s}_{\infty}(\vis_{\bar{\sigma}}^{\eta,\varepsilon}(K)) > \kappa$.
    Thus, Frostman's lemma yields a Frostman $(\delta,s,C\kappa^{-1})$-set
    \begin{equation*}
        \mathcal{P}_{\sigma} \subset \mathcal{D}_{\delta}(\vis_{\bar{\sigma}}^{\eta,\varepsilon}(K)).
    \end{equation*}
    We let $P_{\sigma} \subset \cup \mathcal{P}_{\sigma}$ consist of the lower left corners of the squares in $\mathcal{P}_{\sigma}$.
    Then $P_{\sigma}$ is also a Frostman $(\delta,s,C\kappa^{-1})$-set.

    Write $P \coloneqq \bigcup_{\sigma \in \Sigma} P_\sigma$, as in \cref{p:vis-disc}.
    We now claim a contradiction to \cref{p:vis-disc}.
    Namely,  if $\sigma \in \Sigma$ and $z \in P_{\sigma}$, we claim that
    \begin{equation}\label{form5}
        \diam(P \cap [\ell^+_{z,\sigma}]_{\delta}) \leq \eta/2.
    \end{equation}
    This contradiction will complete the proof of \cref{main}.

    To prove \cref{form5}, fix $\sigma \in \Sigma$ and $z \in P_{\sigma}$.
    By definition, this means that there exist $\bar{\sigma} \in [-1,1]$ with $|\sigma - \bar{\sigma}| \leq \delta$, and a point $x \in \vis_{\bar{\sigma}}^{\eta,\varepsilon}(K)$ which shares the dyadic $\delta$-square with $z$.
    Unwrapping the definitions further (see \cref{form36}),
    \begin{equation*}
        K \cap [\ell^+_{x,\bar{\sigma}}]_{\varepsilon} \subset B(x,\eta/8).
    \end{equation*}
    We claim that this, and $\delta \leq \tfrac{1}{16} \min\{\varepsilon,\eta\}$, implies $[K]_{2\delta} \cap [\ell^+_{x,\bar{\sigma}}]_{\varepsilon/2} \subset B(x,\eta/4)$.
    Indeed, if $y \in [K]_{2\delta} \cap [\ell^+_{x,\bar{\sigma}}]_{\varepsilon/2}$, then by the triangle inequality exists a point $\bar{y} \in K \cap [\ell^+_{x,\bar{\sigma}}]_{\varepsilon} \subset B(x,\eta/8)$ with $|\bar{y} - y| \leq 2\delta$, and therefore $y \in B(x,\eta/4)$.

    Next, we claim that $P \cap [\ell^+_{z,\sigma}]_{\delta} \subset [K]_{2\delta} \cap [\ell^+_{x,\bar{\sigma}}]_{\varepsilon/2}$, which implies \cref{form5}.
    Clearly $P \subset [K]_{2\delta}$.
    Second, the inclusion $[-2,2]^{2} \cap [\ell^+_{z,\sigma}]_{\delta} \subset [\ell^+_{x,\bar{\sigma}}]_{\varepsilon/2}$ holds if $\delta \leq c \varepsilon$ for a sufficiently small absolute constant $c > 0$, using the facts that $|z - x| \leq 2\delta$ and $|\sigma - \bar{\sigma}| \leq \delta$.

    We have now proved \cref{form5}, and consequently \cref{main}.
\end{proof}

\subsection{From the discretised visibility conjecture to incidence lower bounds}\label{s3}
In this section we deduce \cref{p:vis-disc} from \cref{thm:ray-inc-endp}.
Since in \cref{s1} we saw that \cref{p:vis-disc} implies \cref{main}, this completes the proof of \cref{main} (modulo proving \cref{thm:ray-inc-endp}).

\begin{proof}[of \cref{p:vis-disc} assuming \cref{thm:ray-inc-endp}]
    Using the data $\{P_{\sigma}\}_{\sigma \in \Sigma}$ provided in the statement of \cref{p:vis-disc}, we construct a $(\delta,s,s + \alpha - 1,O(C^{2}))$-set $\mathbf{X} \subset \Omega$, and apply \cref{thm:ray-inc-endp} to derive \cref{form7}.
    Note that $s > 2 - \tfrac{\alpha}{2}$ is equivalent to $s + (s + \alpha - 1) > 3$.

    First of all, to prove \cref{p:vis-disc}, we may assume that $\Sigma$ is a Katz--Tao $(\delta,\alpha, 1)$-set with $|\Sigma|\gtrsim \delta^{-\alpha}/C$ and that each $P_{\sigma} \subset (\delta \Z)^{2} \cap [-1,1]^{2}$ is a Katz--Tao $(\delta,s,1)$-set with $|P_{\sigma}| \gtrsim \delta^{-s}/C$.
    This is because $\Sigma$ and each $P_{\sigma}$ contains such a subset, and it suffices to prove \cref{p:vis-disc} with these subsets in place of the original sets.

    Define
    \begin{equation*}
        \mathbf{X} \coloneqq \{(a,b,\sigma) \in (\delta \Z)^{3} \cap \Omega : (a,b) \in P_{\sigma}\}.
    \end{equation*}
    Then, $\mathbf{X}$ is $\delta$-separated, and
    \begin{equation}\label{form8}
        C^{-2}\delta^{-\alpha - s} \lesssim \sum_{\sigma \in \Sigma} |P_{\sigma}| = |\mathbf{X}| \lesssim \delta^{-\alpha - s}.
    \end{equation}
    We claim that $\mathbf{X}$ is a Frostman $(\delta,s,s + \alpha - 1,O(C^{2}))$-set in the sense of \cref{def:PhaseSpace}.
    To see this, fix $u,w \in (0,1]$ with $uw \geq \delta$ and $(a,b,\sigma) \in \Omega$.
    It suffices to show that
    \begin{equation}\label{form9}
        |\mathbf{X} \cap \mathbf{R}_{uw \times uw \times w}(a,b,\sigma)| \lesssim w^\alpha(uw)^{s} \cdot \delta^{-\alpha - s} \stackrel{\mathclap{\cref{form8}}}{\lesssim} C^{2}w^\alpha(uw)^{s}|\mathbf{X}|.
    \end{equation}
    This is because the rectangle $\mathbf{R}_{u \times uw \times w}(a,b,\sigma)$ appearing in \cref{def:PhaseSpace} can be covered by $\lesssim w^{-1}$ rectangles of the form $\mathbf{R}_{uw \times uw \times w}(a',b,\sigma)$.

    To prove \cref{form9}, note that if $(a',b',\sigma') \in \mathbf{X}$ lies in the rectangle
    \begin{equation*}
        \mathbf{R}_{uw \times uw \times w}(a,b,\sigma) = \{(a + r_{1},b + \sigma r_{1} + r_{2},\sigma + r_{3}) : (r_{1},r_{2},r_{3}) \in [-uw,uw]^{2} \times [-w,w]\},
    \end{equation*}
    then $(a',b') \in B((a,b),2uw) \eqqcolon B_{uw} \subset \R^{2}$ and $\sigma' \in B(\sigma,w) \cap [-1,1] \eqqcolon I$.
    Therefore, recalling the definition of $\mathbf{X}$,
    \begin{align*} |\mathbf{X} \cap \mathbf{R}_{uw \times uw \times w}(a,b,\sigma)| & \leq |\{(a',b',\sigma') \in B_{uw} \times I : (a',b') \in P_{\sigma'}\}|\\
& \leq \sum_{\sigma' \in \Sigma \cap I} |P_{\sigma'} \cap B_{uw}| \lesssim (w/\delta)^\alpha \cdot (uw/\delta)^{s} = w^\alpha(uw)^{s} \cdot \delta^{-\alpha - s}.
    \end{align*}
    This completes the proof of \cref{form9}, and shows that $\mathbf{X}$ is a Frostman $(\delta,s,s + \alpha - 1,C')$-set with $C' \lesssim C^{2}$.

    To complete the proof of \cref{p:vis-disc}, apply \cref{thm:ray-inc-endp} with parameters $\alpha,s + \alpha - 1$ and $C'$ to obtain a constant $\kappa = \kappa(\alpha,s,C') \geq 1$.
    Then, for $\rho \sim_{\alpha,s} (C')^{-\kappa}$, \cref{thm:ray-inc-endp} gives a pair $(z_{1},\sigma_{1}),(z_{2},\sigma_{2}) \in \mathbf{X}$ such that $|z_{1} - z_{2}| \geq \rho$ and $\dist(z_{2},\ell^+_{z_{1},\sigma_{1}}) \leq \delta$.
    
    Finally, note that $\dist(z_{2},\ell^+_{z_{1},\sigma_{1}}) \leq \delta$ implies $\{z_{1},z_{2}\} \subset P \cap [\ell^+_{z_{1},\sigma_{1}}]_{\delta}$.
    Therefore $\diam(P \cap [\ell^+_{z_{1},\sigma_{1}}]_{\delta}) \geq |z_{1} - z_{2}| \geq \rho$.
    This proves \cref{form7} with $\eta = \rho$.
\end{proof}

\section{Incidence lower bound for rays with polynomial constant dependence}\label{s:endp}
We now embark on proving \cref{thm:ray-inc-endp}.
In this section, we prove \cref{thm:ray-inc-endp} conditional on \cref{thm:CPZ2}.
We defer the proof of \cref{thm:CPZ2} until \cref{s:non-endp}.

\subsection{Bi-partisation at an intermediate scale}
The plan is as follows: we will apply \cref{thm:CPZ2} at a large intermediate scale $\Delta$, and then use \cref{thm:highLowRays} ("the high-low estimate") to guarantee many incidences at scale $\delta$.

We start by recording a well-known lemma on finding large bi-partite subgraphs.
The lemma is due to Erd{\H o}s \cite{zbl:0134.43403}, but the original formulation concerns undirected graphs, while we need the result for directed graphs.
This variant is also well-known, but we repeat the short proof and, with no extra effort, record a weighted generalisation.
The proof is from \cite[Theorem~2.2.1]{zbmath:6566409}, where the result is also stated for undirected and unweighted graphs, with constant ``$\tfrac{1}{2}$'' instead of ``$\tfrac{1}{4}$''.
\begin{lemma}\label{lemma:erdos}
    Let $G = (V,E)$ be a directed graph without loops, and let $w \colon E \to (0,\infty)$ be a weight function.
    Then there exists a bi-partite subgraph $(V_{1} \cup V_{2},\bar{E})$ with $\sum_{e \in \bar{E}} w(e) \geq \tfrac{1}{4}\sum_{e \in E} w(e)$, where $V_{1},V_{2} \subset V$ are disjoint, and
    \begin{equation}\label{form53}
        \bar{E} = \{(v_{1},v_{2}) \in E : v_{1} \in V_{1}, \, v_{2} \in V_{2}\}.
    \end{equation}
\end{lemma}
\begin{proof}
    Include a vertex $v \in V$ to $V_{1}$ with probability $\tfrac{1}{2}$, and let $V_{2} \coloneqq V \, \setminus \, V_{1}$.
    For $e = (v_{1},v_{2}) \in E$, let $X_{e}$ be the random variable which takes value $w(e)$ if and only if $v_{1} \in V_{1}$ and $v_{2} \in V_{2}$.
    Evidently $\mathbb{E}[X_{e}] = w(e)/4$, and therefore
    \begin{equation*}
        \mathbb{E} \Big[ \sum_{e \in E} X_{e} \Big] = \tfrac{1}{4}\sum_{e \in E}w(e).
    \end{equation*}
    It follows that there exists a random set $V_{1}$ such that $\sum_{e \in E} X_{e} \geq \tfrac{1}{4}\sum_{e \in E}w(e)$.
    Now, with $\bar{E}$ as in \cref{form53}, it holds $\sum_{e \in \bar{E}} w(e) \geq \tfrac{1}{4}\sum_{e \in E} w(e)$.
\end{proof}

\begin{definition}
    We say that sets $\mathbf{A},\mathbf{B} \subset \Omega$ are \emph{spatially $r$-separated} if $\dist(P[\mathbf{A}],P[\mathbf{B}]) \geq r$.
\end{definition}
The following lemma is the key ingredient in the proof of \cref{thm:ray-inc-endp}.
Observe that the sets $\mathbf{A}$, $\mathbf{B}$, and $\mathbf{X}$ are $\delta$-separated so the net effect is that the incidences at scale $3\Delta$ are counted with multiplicity.
\begin{lemma}\label{lemma:bipartite}
    For every $\alpha,\beta \in (1,2]$ with $\alpha + \beta > 3$ and $\varepsilon > 0$, there exist $\eta = \eta(\alpha,\beta,\varepsilon) > 0$ and $\Delta_{0} = \Delta_{0}(\alpha,\beta,\varepsilon) > 0$ such that the following holds for all $\delta,\Delta \in 2^{-\N} \cap (0,\Delta_{0}]$ with $\delta \leq \Delta$:

    Let $\mathbf{X} \subset \Omega$ be a Frostman $(\delta,\alpha,\beta,\Delta^{-\eta})$-set.
    Then, there exist spatially $\Delta$-separated subsets $\mathbf{A},\mathbf{B} \subset \mathbf{X}$ such that
    \begin{equation}\label{form11}
        \mathcal{I}^+_{\Delta}(\mathbf{A},\mathbf{B}) \geq \Delta^{1 + \varepsilon}|\mathbf{X}|^{2}.
    \end{equation}
\end{lemma}
\begin{proof}
    We start by fixing parameters.
    We may assume that $\varepsilon \in (0,(\alpha - 1)/2]$, since the statement gets stronger when $\varepsilon$ decreases.
    Next, let $\eta > 0$ be so small that
    \begin{equation}\label{form75}
        2\eta < (\alpha - 1)/2 \leq \alpha - 1 - \varepsilon.
    \end{equation}
    Additionally, assume that $2\eta \leq \eta_{0}(\alpha,\beta,\varepsilon/2)$, where $\eta_{0}(\alpha,\beta,\varepsilon/2) > 0$ is the constant given by \cref{thm:CPZ2} applied with parameters $\alpha,\beta,\varepsilon/2$.
    Let $\Delta_{0} = \Delta_{0}(\alpha,\beta,\varepsilon/2) > 0$ be the scale threshold given by \cref{thm:CPZ2}, and assume that $\delta,\Delta \in 2^{-\N} \cap (0,\Delta_{0}]$ with $\delta \leq \Delta$.
    Finally, we will assume that $\Delta_{0}$ is small in terms of $\varepsilon,\eta$, so that various absolute constants appearing in the proof are bounded from above by $\Delta_{0}^{-\eta}$ and $\Delta_{0}^{-\varepsilon/2}$.

    We start by replacing $\mathbf{X}$ by a subset with spatial separation.
    Write $\mathcal{Q} \coloneqq \mathcal{D}_{\Delta}([-1,1)^{2})$.
    Split $\mathcal{Q}$ into $4$ sub-families $\mathcal{Q}_{1},\ldots,\mathcal{Q}_{4}$ such that $\dist(Q,Q') \geq \Delta$ for all $i=1,\ldots, 4$ and $Q, Q'\in \mathcal{Q}_{i}$ with $Q \neq Q'$.
    Let $\mathbf{X}_{j} \coloneqq \{\omega \in \mathbf{X} : z_{\omega} \in \cup \mathcal{Q}_{j}\}$.
    Then there exists $j \in \{1,\ldots,4\}$ such that $|\mathbf{X}_{j}| \geq \tfrac{1}{4}|\mathbf{X}|$.
    Now $\mathbf{X}_{j}$ satisfies the same hypotheses as does $\mathbf{X}$ with constant ``$2\eta$'', and additionally: \emph{the family $\mathcal{Q}_{j} \coloneqq \mathcal{D}_{\Delta}(P[\mathbf{X}_{j}])$ is $\Delta$-separated}.
    To simplify notation, we now assume that $\mathbf{X}$ originally has this property, and we write $\mathcal{Q} \coloneqq \mathcal{D}_{\Delta}(P[\mathbf{X}])$.

    Recall that $2\eta \leq \eta_{0}(\alpha,\beta,\varepsilon/2)$, $\Delta \leq \Delta_{0}(\alpha,\beta,\varepsilon/2)$, and note that $\mathbf{X}$ is a $(\Delta,\alpha,\beta,\Delta^{-2\eta})$-set (simply because $\delta \leq \Delta$).
    \cref{thm:CPZ2} applied at scale $\Delta$ now gives
     \begin{equation}\label{form10}
        \mathcal{I}^+_{\Delta}(\mathbf{X}) \geq \Delta^{1 + \varepsilon/2}|\mathbf{X}|^{2}.
    \end{equation}

    Now we proceed to find the sets $\mathbf{A},\mathbf{B} \subset \mathbf{X}$ in \cref{form11}.
    For each pair $(Q_{1},Q_{2}) \in \mathcal{Q} \times \mathcal{Q}$ (including diagonal pairs $(Q,Q)$), define
    \begin{equation}\label{def:mQ}
        m_{Q_{1},Q_{2}} \coloneqq |\{(\omega_{1},\omega_{2}) \in \mathbf{X}^{2} : z_{\omega_{1}} \in Q_{1}, \, z_{\omega_{2}} \in Q_{2} \text{ and } \dist(z_{\omega_{2}},\ell^+_{\omega_{1}}) \leq \Delta\}|.
    \end{equation}
    Let us show that the diagonal sum $\sum_{Q} m_{Q,Q}$ is negligible.
    First, estimate
    \begin{equation*}
        \sum_{Q \in \mathcal{Q}} m_{Q,Q} \leq \sum_{Q} |\{(\omega_{1},\omega_{2}) \in \mathbf{X}^{2} : z_{\omega_{1}},z_{\omega_{2}} \in Q\}| \leq |\mathbf{X}| \cdot \max_{Q \in \mathcal{Q}} |\{\omega \in \mathbf{X} : z_{\omega} \in Q\}|.
    \end{equation*}
    Writing $Q = [a_{0},a_{0} + \Delta) \times [b_{0},b_{0} + \Delta) \in \mathcal{D}_{\Delta}([-1,1)^{2})$, it is easy to check that $\{\omega \in \mathbf{X} : z_{\omega} \in Q\} \subset \mathbf{R}_{\Delta \times \Delta \times 1}(a_{0},b_{0},0)$.
    Since $\mathbf{X}$ was assumed to be a Frostman $(\delta,\alpha,\beta,\Delta^{-\eta})$-set, we may infer that $|\{\omega \in \mathbf{X} : z_{\omega} \in Q\}| \leq \Delta^{\alpha - \eta}|\mathbf{X}|$.
    In particular, since we assumed at \cref{form75} that $\alpha - \eta > 1 + \varepsilon > 1 + \varepsilon/2$,
    \begin{equation*}
        \sum_{Q \in \mathcal{Q}} m_{Q,Q} \leq \tfrac{1}{10} \Delta^{1 + \varepsilon/2}|\mathbf{X}|^{2},
    \end{equation*}
    provided $\Delta > 0$ is sufficiently small.
    As a consequence of this and \cref{form10},
    \begin{equation*}
        \sum_{Q_{1} \neq Q_{2}} m_{Q_{1},Q_{2}} \gtrsim \Delta^{1 + \varepsilon/2}|\mathbf{X}|^{2}.
    \end{equation*}

    We now apply \cref{lemma:erdos} to the weighted directed graph $G = (\mathcal{Q},E,\{w(e)\}_{e \in E})$, where $E = \{(Q_{1},Q_{2}) : Q_{1} \neq Q_{2}\}$ and $w(Q_{1},Q_{2}) = m_{Q_{1},Q_{2}}$.
    The result is a bi-partite sub-graph $(\mathcal{A} \cup \mathcal{B},\overline{\mathcal{E}})$, where $\mathcal{A},\mathcal{B} \subset \mathcal{Q}$ are disjoint, $\overline{\mathcal{E}} = \{(Q_{1},Q_{2}) \in \mathcal{E} : Q_{1} \in \mathcal{A} \text{ and } Q_{2} \in \mathcal{B}\}$, and
    \begin{equation*}
        \sum_{(Q_{1},Q_{2}) \in \overline{\mathcal{E}}} m_{Q_{1},Q_{2}} \geq \tfrac{1}{4} \sum_{Q_{1} \neq Q_{2}} m_{Q_{1},Q_{2}} \gtrsim \Delta^{1 + \varepsilon/2}|\mathbf{X}|^{2}.
    \end{equation*}
    Define $\mathbf{A} \coloneqq \{\omega \in \mathbf{X} : z_{\omega} \in \cup \mathcal{A}\}$ and $\mathbf{B} \coloneqq \{\omega \in \mathbf{X} : z_{\omega} \in \cup \mathcal{B}\}$.
    The sets $\mathbf{A},\mathbf{B}$ are spatially $\Delta$-separated, since $\mathcal{A},\mathcal{B}$ are disjoint subsets of the $\Delta$-separated family $\mathcal{Q}$.
    Finally, recalling the definition of $m_{Q_{1},Q_{2}}$ from \cref{def:mQ}, and using $\overline{\mathcal{E}} \subset \mathcal{E}$,
    \begin{align*}
        \mathcal{I}^{+}_{\Delta}&(\mathbf{A},\mathbf{B}) \stackrel{\mathrm{def.}}{=} |\{(\omega_{1},\omega_{2}) \in \mathbf{A} \times \mathbf{B} : \dist(z_{\omega_{2}},\ell^+_{\omega_{1}}) \leq \Delta\}|\\
                             &= \sum_{(Q_{1},Q_{2}) \in \mathcal{A} \times \mathcal{B}} |\{(\omega_{1},\omega_{2}) \in \mathbf{X}^{2} : z_{\omega_{1}} \in Q_{1}, \, z_{\omega_{2}} \in Q_{2} \text{ and } \dist(z_{\omega_{2}},\ell^+_{\omega_{1}}) \leq \Delta\}|\\
                             & \stackrel{\mathclap{\mathrm{def.}}}{=} \sum_{(Q_{1},Q_{2}) \in \overline{\mathcal{E}}} m_{Q_{1},Q_{2}} \gtrsim \Delta^{1 + \varepsilon/2}|\mathbf{X}|^{2}.
    \end{align*}
    This proves \cref{form11} for $\Delta > 0$, depending on $\varepsilon$.
\end{proof}

\subsection{Proof of \texorpdfstring{\cref{thm:ray-inc-endp}}{Theorem~\ref{thm:ray-inc-endp}}}
Now we use \cref{lemma:bipartite} along with the \emph{high-low inequality} (for rays) to prove \cref{thm:ray-inc-endp}.
If one wished to prove \cref{thm:ray-inc-endp} with lines in place of rays, the version of the high-low inequality in \cite[Theorem 1.7]{zbmath:8038252} stated in \cref{form54} would suffice.
We give the statement of an appropriate version for rays in \cref{thm:highLowRays} below, but postpone the proof to \cref{s:highLowAppendix}.

Let $\mathbf{A},\mathbf{B} \subset \Omega$ be finite sets.
The reader may think of $\mathbf{A}$ as parametrising a family of rays, and $\mathbf{B}$ as parametrising a family of points.
Accordingly, define
\begin{equation*}
    g \coloneqq \sum_{\omega \in \mathbf{B}} \delta_{z_{\omega}} \quad \text{and} \quad f_{w} \coloneqq w^{-1} \sum_{\omega \in \mathbf{A}} \mathbf{1}_{[\ell_{\omega}^{+}]_{w/2}}, \, w \in (0,1],
\end{equation*}
where $[\ell_{\omega}^{+}]_{w/2}$ is the open $(w/2)$-neighbourhood of $\ell_{\omega}^{+}$.
Let $\chi \in C_{c}^{\infty}(\R^{2})$ be a fixed non-negative function with $\spt \chi \subset B(1)$.
For $w > 0$, write $\chi_{w} \coloneqq w^{-2}\chi(\cdot/w)$, and define
\begin{equation}\label{form40}
    \mathcal{J}_{w}^{+}(\mathbf{A},\mathbf{B}) \coloneqq w \cdot \langle \chi_{w} \ast \chi_{w/2} \ast f_{w},g\rangle \coloneqq w \cdot \int \chi_{w} \ast \chi_{w/2} \ast f_{w} \, dg.
\end{equation}
The definition of $\mathcal{J}^+_{w}$ depends on $\chi$, but we suppress this from the notation.
The quantity $\mathcal{J}_{w}^{+}$ should be viewed as a variant of the incidence count $\mathcal{I}_{w}^{+}$ which is more amenable to $L^{2}$-orthogonality arguments; we will see this benefit in the proof of \cref{thm:highLowRays}.
For now, we only need to know that $\mathcal{J}_{w}^{+}$ is comparable to $\mathcal{I}_{w}^{+}$ in the following sense:
 \begin{lemma}\label{lemma:incidenceComparison}
     If $\chi \in C_{c}^{\infty}(\R^{2})$ is arbitrary with $\spt \chi \subset B(1)$, then $\mathcal{J}_{w}^{+}(\mathbf{A},\mathbf{B}) \lesssim_{\chi} \mathcal{I}_{2w}^{+}(\mathbf{A},\mathbf{B})$.
 Moreover, if $\chi \geq \mathbf{1}_{B(1/2)}$, then $\mathcal{J}_{w}^{+}(\mathbf{A},\mathbf{B}) \gtrsim \mathcal{I}^{+}_{w}(\mathbf{A},\mathbf{B})$.
\end{lemma}
\begin{proof}
    The upper bound follows from $\chi_{w} \ast \chi_{w/2} \ast \mathbf{1}_{[\ell^+_{\omega}]_{w/2}} \lesssim_{\chi} \mathbf{1}_{[\ell^+_{\omega}]_{2w}}$, which yields
    \begin{equation*}
        \mathcal{J}_{w}^{+}(\mathbf{A},\mathbf{B}) \lesssim_{\chi} \sum_{\omega_{1} \in \mathbf{A}} \sum_{\omega_{2} \in \mathbf{B}} \mathbf{1}_{[\ell_{\omega_{1}}^{+}]_{2w}}(z_{\omega_2}) = \mathcal{I}_{2w}^{+}(\mathbf{A},\mathbf{B}),
    \end{equation*}
    as desired.
    If $\chi \geq \mathbf{1}_{B(1/2)}$, then we also have a nearly matching lower bound $\chi_{w} \ast \chi_{w/2} \ast \mathbf{1}_{[\ell^{+}_{\omega}]_{w/2}} \gtrsim \mathbf{1}_{[\ell^{+}_{\omega}]_{w/2}}$, which yields the stated lower bound.
\end{proof}
Before stating the high-low theorem, we need one more piece of notation: for $\mathbf{X} \subset \Omega$ and $w > 0$, define (following the notation in \cite{zbmath:8038252})
\begin{equation}\label{form29}
    M_{w \times w}(\mathbf{X}) \coloneqq \sup_{z_{0} \in \R^{2}} |\{\omega \in \mathbf{X} : z_{\omega} \in B(z_{0},w)\}|
\end{equation}
and
\begin{equation}\label{form29a}
    M_{1 \times w}(\mathbf{X}) \coloneqq \sup_{\ell_{0} \in \mathcal{A}(2,1)} |\{\omega \in \mathbf{X} : \ell_{\omega} \in B_{d_{\mathcal{A}}}(\ell_{0},w)\}|.
\end{equation}
The ball $B_{d_{\mathcal{A}}}(\ell_{0},w) \subset \mathcal{A}(2,1)$ is defined relative to the following metric $d_{\mathcal{A}}$ on $\mathcal{A}(2,1)$ (the same metric as in \cite[(3)]{zbmath:8038252}).
For $\ell_{1},\ell_{2} \in \mathcal{A}(2,1)$ with slopes $\sigma_{1},\sigma_{2} \in [-1,1]$, we set
\begin{equation*}
    d_{\mathcal{A}}(\ell_{1},\ell_{2}) \coloneqq |\dist(\ell_{1},0) - \dist(\ell_{2},0)| + |\sigma_{1} - \sigma_{2}|.
\end{equation*}
We are then prepared to state the high-low inequality for rays:

\begin{theorem}\label{thm:highLowRays} Let $\mathbf{A},\mathbf{B} \subset \Omega$ be finite sets.
    Then,
    \begin{equation*}
        \Big| \frac{\mathcal{J}_{w}^{+}(\mathbf{A},\mathbf{B})}{w|\mathbf{A}||\mathbf{B}|} - \frac{\mathcal{J}^{+}_{w/2}(\mathbf{A},\mathbf{B})}{(w/2)|\mathbf{A}||\mathbf{B}|} \Big| \leq A_{\chi}\Big(\frac{M_{1 \times w}(\mathbf{A})M_{w \times w}(\mathbf{B})}{|\mathbf{A}||\mathbf{B}|} \cdot w^{-3}\log w^{-1} \Big)^{1/2}, \quad w \in (0,\tfrac{1}{2}],
    \end{equation*}
    where the constant $A_{\chi} > 0$ depends only on $\chi$.
\end{theorem}
\begin{remark}
    Let us compare \cref{thm:highLowRays} to \cite[Theorem 1.7]{zbmath:8038252}.
    First of all, one could also define $\mathcal{J}_{w}(\mathbf{A},\mathbf{B})$ as above, but with the ray $\ell_{\omega}^{+}$ inside the definition of $f_{w}$ replaced by the line $\ell_{\omega}$.
    Then the following holds:
    \begin{equation}\label{form54}
        \Big| \frac{\mathcal{J}_{w}(\mathbf{A},\mathbf{B})}{w|\mathbf{A}||\mathbf{B}|} - \frac{\mathcal{J}_{w/2}(\mathbf{A},\mathbf{B})}{(w/2)|\mathbf{A}||\mathbf{B}|} \Big| \leq A_{\chi}\Big(\frac{M_{1 \times w}(\mathbf{A})M_{w \times w}(\mathbf{B})}{|\mathbf{A}||\mathbf{B}|} \cdot w^{-3} \Big)^{1/2}, \quad w \in (0,1].
    \end{equation}
    This is proved in \cite[Appendix A.1]{zbmath:8038252} (and it implies the high-low inequality as stated in \cite[Theorem 1.7]{zbmath:8038252}).
    The inequality \cref{form54} is a ``$\log w^{-1}$'' factor sharper than the one stated in \cref{thm:highLowRays}.
\end{remark}
We are then equipped to prove \cref{thm:ray-inc-endp}.
\begin{proofref}{thm:ray-inc-endp}
    Let $\chi \in C^{\infty}_{c}(\R^{2})$ be a non-negative function satisfying $\mathbf{1}_{B(1/2)} \leq \chi \leq \mathbf{1}_{B(1)}$, and let $A \geq 1$ be an absolute constant to be determined, which in particular is larger than the constant $A_{\chi}$ from \cref{thm:highLowRays} associated to $\chi$.
    Let $\mathbf{X} \subset \Omega$ be a Frostman $(\delta,\alpha,\beta,C)$-set.
    Let $\eta = \eta(\alpha,\beta,\varepsilon) > 0$ and $\Delta_{0} = \Delta_{0}(\alpha,\beta,\varepsilon) > 0$ be the constants given by \cref{lemma:bipartite} with
    \begin{equation*}
        \varepsilon \coloneqq \varepsilon(\alpha,\beta) \coloneqq \frac{\alpha + \beta - 3}{6} \in (0,1].
    \end{equation*}
    (Note that $\eta,\Delta_{0}$ actually only depend on $\alpha,\beta$.) Finally, choose
    \begin{equation*}
        \kappa \coloneqq \kappa(\alpha,\beta) \coloneqq \frac{1}{\min\{\eta,\varepsilon\}}.
    \end{equation*}
    We claim that \cref{thm:ray-inc-endp} holds with this choice of $\kappa$.

    We begin by choosing an intermediate scale $\Delta$.
    The first requirement is that $\Delta \leq \min\{\Delta_0, C^{-1/\eta}\}$.
    This guarantees that \cref{lemma:bipartite} can be applied at scale $\Delta$ (since $\mathbf{X}$ is a Frostman $(\delta,\alpha,\beta,\Delta^{-\eta})$-set).
    The second requirement is that $\Delta$ is sufficiently small so that the following inequality holds:
    \begin{equation}\label{form13}
        \Delta^{\varepsilon} - A \sum_{w \in 2^{-\N} \cap (0,\Delta]} (C^{2}w^{\alpha + \beta - 3}\log w^{-1})^{1/2} \geq \tfrac{1}{2}\Delta^{\varepsilon}.
    \end{equation}
    This can be arranged: if $\Delta = 2^{-N}$ is sufficiently small depending only on $\varepsilon$, it holds $k^{1/2} \leq 2^{-\varepsilon k}$ for all $k \geq N$.
    Thus there is a constant $c_\varepsilon > 0$ so that if $\Delta \leq c_\varepsilon C^{-1/\varepsilon}$,
    \begin{equation*}
        \sum_{k=N}^\infty k^{1/2}2^{-3\varepsilon k} \leq \sum_{k=N}^\infty 2^{-2\varepsilon k} \leq \frac{\Delta^{2\varepsilon}}{1-2^{-2\varepsilon}} \leq \frac{\Delta^{\varepsilon}}{2AC}.
    \end{equation*}

    Let $\Delta \in 2^{-\N}$ be chosen so that
    \begin{equation}\label{e:Dch}
        \Delta \leq \min\{C^{-1/\eta}, c_\varepsilon C^{-1/\varepsilon}, \Delta_0\} < 2\Delta
    \end{equation}
    Let $\mathbf{A},\mathbf{B} \subset \mathbf{X}$ be the spatially $\Delta$-separated subsets obtained by \cref{lemma:bipartite} applied at scale $\Delta$.

    First, note that the conclusion \cref{form11} combined with \cref{lemma:incidenceComparison} implies $\mathcal{J}^+_{\Delta}(\mathbf{A},\mathbf{B}) \gtrsim \mathcal{I}^+_{\Delta}(\mathbf{A},\mathbf{B}) \geq \Delta^{1 + \varepsilon}|\mathbf{X}|^{2}$.
    It follows from the Frostman $(\delta,\alpha,\beta,C)$-set property of $\mathbf{X}$ that
    \begin{equation}\label{form77}
        M_{w \times w}(\mathbf{B}) \lesssim Cw^{\alpha}|\mathbf{X}| \quad \text{and} \quad M_{1 \times w}(\mathbf{A}) \lesssim Cw^{\beta}|\mathbf{X}|, \quad w \in 2^{-\N}.
    \end{equation}
    (To check this detail, note that sets of the form $\{\omega : z_{\omega} \in B(z_{0},w)\}$ and $\{\omega : \ell_{\omega} \in B_{d_{\mathcal{A}}}(\ell_{0},w)\}$ appearing in the definitions of $M_{w \times w}(\mathbf{B})$ and $M_{1 \times w}(\mathbf{A})$ are contained, respectively, in phase space rectangles of the form $\mathbf{R}_{Aw \times Aw \times 1}(\omega_{0})$ and $\mathbf{R}_{1 \times Aw \times Aw}(\omega_{0})$, where $A \geq 1$ is absolute.
    This is discussed in \cite[(17)]{zbmath:8038252}.) 
    
    Fix $w \in [\delta/2,\Delta]$, and apply the high-low inequality (\cref{thm:highLowRays}) at scale $w$ as follows:
    \begin{align*}
        \Big| \frac{\mathcal{J}^+_{w}(\mathbf{A},\mathbf{B})}{w|\mathbf{X}|^{2}} - \frac{\mathcal{J}^+_{w/2}(\mathbf{A},\mathbf{B})}{(w/2)|\mathbf{X}|^{2}} \Big|
        & = \frac{|\mathbf{A}||\mathbf{B}|}{|\mathbf{X}|^{2}} \Big| \frac{\mathcal{J}^+_{w}(\mathbf{A},\mathbf{B})}{w|\mathbf{A}||\mathbf{B}|} - \frac{\mathcal{J}^+_{w/2}(\mathbf{A},\mathbf{B})}{(w/2)|\mathbf{A}||\mathbf{B}|} \Big|\\
        & \leq A_{\chi} \left(\frac{|\mathbf{A}||\mathbf{B}|}{|\mathbf{X}|^{2}} \right)^{1/2}\left(\frac{M_{1 \times w}(\mathbf{A})M_{w \times w}(\mathbf{B})}{|\mathbf{X}|^{2}} \cdot w^{-3} \log w^{-1} \right)^{1/2}\\
        & \stackrel{\mathclap{\cref{form77}}}{\leq}\,\, A(C^{2}w^{\alpha + \beta - 3} \log w^{-1})^{1/2}.
    \end{align*}
    Finally, summing over $w \in [\delta/2,\Delta]$ and using \cref{lemma:incidenceComparison},
    \begin{align*}
        \frac{\mathcal{I}^+_{\delta}(\mathbf{A},\mathbf{B})}{\delta |\mathbf{X}|^{2}} \gtrsim \frac{\mathcal{J}^+_{\delta/2}(\mathbf{A},\mathbf{B})}{\delta|\mathbf{X}|^{2}}
        &\geq \Delta^{\varepsilon} - A \sum_{w \in 2^{-\N} \cap (0,\Delta]} (C^{2}w^{\alpha + \beta - 3} \log w^{-1})^{1/2}\\
        &\stackrel{\mathclap{\cref{form13}}}{\geq}\,\,\tfrac{1}{2}\Delta^{\varepsilon}\stackrel{\textup{\cref{e:Dch}}}{\gtrsim_{\alpha,\beta}} C^{-\max\{\varepsilon/\eta, 1\}}.
    \end{align*}
    Since $\mathbf{A},\mathbf{B}$ are spatially $\Delta$-separated, the proof of \cref{thm:ray-inc-endp} is complete.
    \end{proofref}

\section{Incidence lower bound for rays with \texorpdfstring{$\delta^{\varepsilon}$}{δ\^ε}-loss}\label{s:non-endp}
We now begin the proof of \cref{thm:CPZ2}.
The argument is virtually the same as the proof of \cite[Theorem 1.9]{zbmath:8038252} (or \cref{thm:CPZ}), apart from a few technical difficulties caused by considering ray incidences instead of line incidences.
Fortunately, we are able to use as black boxes some of the most complicated parts in the proof of \cite[Theorem 1.9]{zbmath:8038252} which deal with the ``combinatorics of Lipschitz functions''.

\subsection{An initial estimate for rays}\label{ss:initial} The purpose of this section is to prove \cref{lemma:initialEstimate}.
This is a counterpart of the ``initial estimate'' in \cite[Proposition 4.2]{zbmath:8038252}, with $\mathcal{I}^{+}_{\eta}$ in place of $\mathcal{I}_{\eta}$.

\begin{definition}[Uniformity]
    Let $\mathcal{S} = \{u_{j} \times v_{j} \times w_{j}\}_{j \in J}$ be a family of scale triples satisfying $v_{j} \geq u_{j}w_{j}$, and let $K \geq 1$.
    A finite set $\mathbf{X} \subset \Omega$ is called \emph{$K$-uniform on $\mathcal{S}$} if for every $j \in J$, $\omega \in \mathbf{X}$,
    \begin{equation*}
        |\mathbf{X} \cap \mathbf{R}_{u_{j} \times v_{j} \times w_{j}}(\omega)| \geq \frac{M_{u_{j} \times v_{j} \times w_{j}}(\mathbf{X})}{K}.
    \end{equation*}
    Here $M_{u \times v \times w}(\mathbf{X}) \coloneqq \sup_{\omega \in \Omega} |\mathbf{X} \cap \mathbf{R}_{u \times v \times w}(\omega)|$.

    We will also say that \emph{$\mathbf{X}$ is uniform on a sequence $\{\Delta_{j}\}_{j = 0}^{m}$}.
    This means that $\mathbf{X}$ is uniform on $\{\Delta_{i} \times \Delta_{j} \times \Delta_{k} : i,j,k \in \{0,\ldots,m\} \text{ and } \Delta_{j} \geq \Delta_{i}\Delta_{k}\}$.
\end{definition}

\begin{remark}
    In \cite{zbmath:8038252} the number $M_{u \times v \times w}(\mathbf{X})$ is defined as a maximum over a fixed boundedly overlapping cover of $\R^{3}$ by $(u \times v \times w)$-rectangles, see above \cite[Lemma 3.2]{zbmath:8038252}.
    The two definitions of $M_{u \times v \times w}(\mathbf{X})$ are comparable up to absolute constants.
\end{remark}

\begin{definition}[Rectangular covering numbers]
    For $u,v,w \in (0,1]$ with $v \geq uw$, and $\mathbf{X} \subset \Omega$, we write $|\mathbf{X}|_{u \times v \times w}$ for the minimal number of rectangles of the form $\mathbf{R}_{u \times v \times w}(\omega)$, $\omega \in \Omega$, required to cover $\mathbf{X}$.
\end{definition}

For $K$-uniform sets, there is a useful relation between $M_{u \times v \times w}(\mathbf{X})$ and $|\mathbf{X}|_{u \times v \times w}$:
\begin{lemma}\label{lemma5}
    Let $\mathbf{X} \subset \Omega$ be a finite set which is $K$-uniform on $\{u \times v \times w\}$.
    Then,
    \begin{equation*}
        M_{u \times v \times w}(\mathbf{X}) \lesssim K|\mathbf{X}|/|\mathbf{X}|_{u \times v \times w}.
    \end{equation*}
\end{lemma}
\begin{proof}
    This is \cite[(27)]{zbmath:8038252}.
\end{proof}

\begin{lemma}\label{lemma:initialEstimate}
    Let $\eta \in 2^{-\N} \cap (0,\tfrac{1}{10}]$ and $K \geq 1$.
    Let $\mathbf{X} \subset \Omega$ be a finite set which is $K$-uniform on $\{(1 \times \eta \times \eta),(\eta \times \eta \times 1),(\eta \times \eta \times \eta)\}$.
    Then,
    \begin{equation}\label{form26a}
        \frac{\mathcal{I}^{+}_{10\eta}(\mathbf{X})}{\eta |\mathbf{X}|^{2}} \gtrsim K^{-5} \frac{|\mathbf{X}|_{\eta \times \eta \times \eta}}{\eta |\mathbf{X}|_{\eta \times \eta \times 1}|\mathbf{X}|_{1 \times \eta \times \eta}}.
    \end{equation}
\end{lemma}
\begin{proof}
    We start by disposing of a special case where $|\mathbf{X}|_{\eta \times \eta \times \eta} \leq AK^{2}|\mathbf{X}|_{1 \times \eta \times \eta}$.
    Here $A \geq 1$ is an absolute constant to be determined later (above \cref{form59}).
    In this case, note that
    \begin{equation}\label{form76}
        \mathbf{R}_{\eta \times \eta \times 1}(\omega) \subset \{\omega' \in \Omega : |z_{\omega} - z_{\omega'}| \leq 3\eta\}, \qquad \omega \in \Omega, \, \eta \in (0,1].
    \end{equation}
    In particular, since always $z_{\omega} \in \ell_{\omega}^{+}$, we find
    \begin{equation*}
        \mathcal{I}^{+}_{10\eta}(\mathbf{X}) \geq \sum_{\omega \in \mathbf{X}} |\{\omega' \in \mathbf{X} : |z_{\omega} - z_{\omega'}| \leq 3\eta\}| \geq \sum_{\omega \in \mathbf{X}} |\mathbf{X} \cap \mathbf{R}_{\eta \times \eta \times 1}(\omega)| \geq K^{-1}|\mathbf{X}|M_{\eta \times \eta \times 1}(\mathbf{X}),
    \end{equation*}
    using $K$-uniformity in the final estimate.
    Since further $M_{\eta \times \eta \times 1}(\mathbf{X}) \geq |\mathbf{X}|/|\mathbf{X}|_{\eta \times \eta \times 1}$, we see that $\mathcal{I}^{+}_{10\eta}(\mathbf{X}) \geq K^{-1}|\mathbf{X}|^{2}/|\mathbf{X}|_{\eta \times \eta \times 1}$.
    Since $|\mathbf{X}|_{\eta \times \eta \times \eta} \leq AK^{2}|\mathbf{X}|_{1 \times \eta \times \eta}$ by assumption, this yields
    \begin{equation*}
        \frac{\mathcal{I}^{+}_{10\eta}(\mathbf{X})}{\eta|\mathbf{X}|^{2}} \geq A^{-1}K^{-3}\frac{|\mathbf{X}|_{\eta \times \eta \times \eta}}{\eta|\mathbf{X}|_{\eta \times \eta \times 1}|\mathbf{X}|_{1 \times \eta \times \eta}},
    \end{equation*}
    which is better than \cref{form26a}.
    So, in the sequel we may assume that
    \begin{equation}\label{form19a}
        |\mathbf{X}|_{\eta \times \eta \times \eta} \geq AK^{2}|\mathbf{X}|_{1 \times \eta \times \eta}.
    \end{equation}

    We now construct a boundedly overlapping cover of $\mathbf{X}$ by rectangles of the form $\mathbf{R}_{2 \times \eta \times \eta}(\omega)$ with $\omega \in \mathbf{X}$.
    Consider the map $\Pi \colon \Omega \to \R^{2}$ defined by $\Pi(a,b,\sigma) \coloneqq (b - a\sigma,\sigma)$.
    Let $\{(a_{j},b_{j},\sigma_{j})\}_{j \in J} \subset \mathbf{X}$ be a maximal set with the property that the points $\Pi(a_{j},b_{j},\sigma_{j}) \in \R^{2}$ are $(\eta/2)$-separated.
    Define $\mathcal{R} \coloneqq \{\mathbf{R}_{2 \times \eta \times \eta}(a_{j},b_{j},\sigma_{j}) : j \in J\}$.

    We claim that $\mathbf{X} \subset \cup \mathcal{R}$.
    Fix $(a,b,c) \in \mathbf{X}$.
    By definition, there exists $(a_{j},b_{j},\sigma_{j})$, $j \in J$, such that $|(b - a\sigma) - (b_{j} - a_{j}\sigma_{j})| \leq \eta/2$ and $|\sigma - \sigma_{j}| \leq \eta/2$.
    We claim that $(a,b,\sigma) \in \mathbf{R}_{2 \times \eta \times \eta}(a_{j},b_{j},\sigma_{j})$.
    To see this, first expand
    \begin{equation*}
        (a,b,\sigma) = (a_{j} + (a - a_{j}), b_j + (b - b_{j}),\sigma_{j} + (\sigma - \sigma_{j})) \eqqcolon (a_{j} + r_{1},b_{j} + (b - b_{j}),\sigma_{j} + r_{3}),
    \end{equation*}
    where $|r_{1}| = |a - a_{j}| \leq 2$ and $|r_{3}| \leq \eta/2$.
    Next, note that for some $\xi \in \R$ with $|\xi| \leq \eta/2$,
    \begin{equation*}
        b - b_{j} = a\sigma - a_{j}\sigma_{j} + \xi = \sigma_{j}(a - a_{j}) + a(\sigma - \sigma_{j}) + \xi = \sigma_{j}r_{1} + r_{2},
    \end{equation*}
    where $|r_{2}| = |a(\sigma - \sigma_{j}) + \xi| \leq \eta$.
    We have managed to write $(a,b,\sigma)$ in the form $(a,b,\sigma) = (a_{j} + r_{1},b_{j} + \sigma_{j}r_{1} + r_{2},\sigma_{j} + r_{3})$ for some $(r_{1},r_{2},r_{3}) \in [-2,2] \times [-\eta,\eta]^{2}$, and thus $(a,b,\sigma) \in \mathbf{R}_{2 \times \eta \times \eta}(a_{j},b_{j},\sigma_{j})$.

    \begin{claim}
        The rectangles $2\mathcal{R} \coloneqq \{\mathbf{R}_{2 \times 2\eta \times 2\eta}(a_{j},b_{j},\sigma_{j})\}$ have bounded overlap.
    \end{claim}
    \begin{proof}
        It suffices to show that if $(x,y,z) \in \mathbf{R}_{2 \times 2\eta \times 2\eta}(a,b,\sigma) \cap \mathbf{R}_{2 \times 2\eta \times 2\eta}(a',b',\sigma')$, then $|\Pi(a,b,\sigma) - \Pi(a',b',\sigma')| \lesssim \eta$.
        Write
        \begin{equation}\label{form57}
            (x,y,z) = (a + r_{1},b + r_{1}\sigma + r_{2},\sigma + r_{3}) = (a' + r_{1}',b' + r_{1}'\sigma' + r_{2}',\sigma' + r_{3}'),
        \end{equation}
        where $r_{1},r_{1}' \in [-2,2]$ and $r_{2},r_{2}',r_{3},r_{3}' \in [-2\eta,2\eta]$.
        Then $a + r_{1} = a' + r_{1}'$ and $|\sigma - \sigma'| \leq |r_{3}| + |r_{3}'| \leq 4\eta$.
        To simplify the calculation, let us pretend that $\sigma = \sigma'$; this is correct up to $O(\eta)$-terms.
        From \cref{form57} we see that
        \begin{equation*}
            b - b' = (r_{1}' - r_{1})\sigma + O(\eta) = (a - a')\sigma + O(\eta).
        \end{equation*}
        Therefore $|(b - a\sigma) - (b' - a'\sigma)| \lesssim \eta$, and finally $|\Pi(a,b,\sigma) - \Pi(a',b',\sigma')| \lesssim \eta$.
    \end{proof}
    Note that $|\mathcal{R}| \geq |\mathbf{X}|_{1 \times \eta \times \eta}$ by the definition of $|\cdot|_{1 \times \eta \times \eta}$, and $|\mathbf{X} \cap \mathbf{R}| \geq K^{-1}M_{1 \times \eta \times \eta}(\mathbf{X})$ by $K$-uniformity for all $\mathbf{R} \in \mathcal{R}$ (the rectangles in $\mathcal{R}$ are centred at $\mathbf{X}$).
    Consequently,
    \begin{equation*}
        M_{\eta \times \eta \times \eta}(\mathbf{X})|\mathbf{X} \cap \mathbf{R}|_{\eta \times \eta \times \eta} \geq |\mathbf{X} \cap \mathbf{R}| \geq K^{-1}M_{1 \times \eta \times \eta}(\mathbf{X}) \geq K^{-1}|\mathbf{X}|/|\mathbf{X}|_{1 \times \eta \times \eta}.
    \end{equation*}
    Rearranging, and using \cref{lemma5} to estimate
    \begin{equation}\label{form56}
        |\mathbf{X}| \gtrsim K^{-1}M_{\eta \times \eta \times \eta}(\mathbf{X})|\mathbf{X}|_{\eta \times \eta \times \eta},
    \end{equation}
    we find
    \begin{equation}\label{form54b}
        |\mathbf{X} \cap \mathbf{R}|_{\eta \times \eta \times \eta} \gtrsim K^{-2}|\mathbf{X}|_{\eta \times \eta \times \eta}/|\mathbf{X}|_{1 \times \eta \times \eta} \stackrel{\cref{form19a}}{\gtrsim} A, \qquad \mathbf{R} \in \mathcal{R}.
    \end{equation}
    Next, fix $\mathbf{R}= \mathbf{R}_{2 \times \eta \times \eta}(a_{0},b_{0},\sigma_{0}) \in \mathcal{R}$, thus $(a_{0},b_{0},\sigma_{0}) \in \mathbf{X}$, and
    \begin{equation*}
        \mathbf{R} = \{(a_{0} + r_{1},b_{0} + \sigma_{0}r_{1} + r_{2},\sigma_{0} + r_{3}) : (r_{1},r_{2},r_{3}) \in [-2,2] \times [-\eta,\eta]^{2}\}.
    \end{equation*}
    Let $P_{\mathbf{R}} \subset P[\mathbf{X} \cap \mathbf{R}] \subset [-1,1]^{2}$ be a maximal $\eta$-separated set.
    We claim that
    \begin{equation}\label{form55}
        |P_{\mathbf{R}}| \gtrsim |\mathbf{X} \cap \mathbf{R}|_{\eta \times \eta \times \eta} \gtrsim A.
    \end{equation}
    To see this, it suffices to show that $\{\omega_{2} \in \mathbf{R} : |z_{\omega_{2}} - z_{\omega_{1}}| \leq \eta\} \subset B(\omega_{1},3\eta)$ for all $\omega_{1} \in \mathbf{R}$.
    This is true, since every pair $\omega_{1},\omega_{2} \in \mathbf{R}$ differs in the $3^{rd}$ coordinate by $\leq 2\eta$.

    Write $\ell_{\mathbf{R}} \coloneqq \ell_{(a_{0},b_{0},\sigma_{0})} = (a_{0},b_{0}) + \R(1,\sigma_{0})$.
    From the definition of $\mathbf{R}$, one sees that $P[\mathbf{R}] \subset [\ell_{\mathbf{R}}]_{\eta}$ (the $\eta$-neighbourhood of $\ell_{\mathbf{R}}$).
    Summarising the previous observations, $P_{\mathbf{R}}$ is an $\eta$-separated subset of $[\ell_{\mathbf{R}}]_{\eta}$ with cardinality $|P_{\mathbf{R}}|\gtrsim |\mathbf{X} \cap \mathbf{R}|_{\eta\times\eta\times\eta} \gtrsim A$.
    Thus choosing $A \geq 1$ sufficiently large, $|P_{\mathbf{R}}| \geq 2$.
    Since $|\sigma_{0}| \leq 1$, we now extract disjoint subsets $A_{\mathbf{R}},B_{\mathbf{R}} \subset P_{\mathbf{R}}$ with cardinalities
    \begin{equation}\label{form59}
        |A_{\mathbf{R}}| \sim |P_{\mathbf{R}}| \sim |B_{\mathbf{R}}|
    \end{equation}
    such that $\max \{x : (x,y) \in A_{\mathbf{R}}\} \leq \min \{x : (x,y) \in B_{\mathbf{R}}\}$.
    It is then easy to check that the following holds:
    \begin{claim}\label{c6}
        Assume that $z_{1},z_{2} \in \R^{2}$ and $\sigma \in \R$ satisfy
        \begin{equation*}
            \dist(z_{1},A_{\mathbf{R}}) \leq 2\eta, \quad \dist(z_{2},B_{\mathbf{R}}) \leq 2\eta, \quad \text{and} \quad |\sigma - \sigma_{0}| \leq 2\eta.
        \end{equation*}
        Then $\dist(z_{2},\ell_{z_{1},\sigma}^{+}) \leq 10\eta$.
    \end{claim}
    We now define
    \begin{equation*}
        \mathbf{A}(\mathbf{R}) \coloneqq \{\omega \in \mathbf{X} \cap 2\mathbf{R} : \dist(z_{\omega},A_{\mathbf{R}}) \leq 2\eta\},
    \end{equation*}
    where $2\mathbf{R} \coloneqq \mathbf{R}_{2 \times 2\eta \times 2\eta}(a_{0},b_{0},\sigma_{0})$, and
    \begin{equation*}
        \mathbf{B}(\mathbf{R}) \coloneqq \{\omega \in \mathbf{X} : \dist(z_{\omega},B_{\mathbf{R}}) \leq 2\eta\}.
    \end{equation*}
    Note that if $\omega_{1} = (z_{1},\sigma_{1}) \in \mathbf{A}(\mathbf{R})$ and $\omega_{2} = (z_{2},\sigma_{2}) \in \mathbf{B}(\mathbf{R})$, then the points $z_{j}$ and the slope $\sigma_{1}$ satisfy the hypotheses of \cref{c6}.
    Thus,
    \begin{equation*}
        \mathbf{A}(\mathbf{R}) \times \mathbf{B}(\mathbf{R}) \subset \{(\omega_{1},\omega_{2}) \in \mathbf{X}^2 : \dist(z_{\omega_{2}},\ell_{\omega_{1}}^{+}) \leq 10\eta\}.
    \end{equation*}
    Since the rectangles $2\mathbf{R}$ have bounded overlap, this implies
    \begin{equation*}
        \mathcal{I}^{+}_{10\eta}(\mathbf{X}) = |\{(\omega_{1},\omega_{2}) \in \mathbf{X}^2 : \dist(z_{\omega_{2}},\ell_{\omega_{1}}^{+}) \leq 10\eta\}| \gtrsim \sum_{\mathbf{R} \in \mathcal{R}} |\mathbf{A}(\mathbf{R})||\mathbf{B}(\mathbf{R})|.
    \end{equation*}
    We finally claim that
    \begin{equation}\label{form58}
        |\mathbf{A}(\mathbf{R})| \gtrsim K^{-3}\frac{|\mathbf{X}|}{|\mathbf{X}|_{1 \times \eta \times \eta}} \quad \text{and} \quad |\mathbf{B}(\mathbf{R})| \gtrsim K^{-1}|P_{\mathbf{R}}| \cdot \frac{|\mathbf{X}|}{|\mathbf{X}|_{\eta \times \eta \times 1}}
    \end{equation}
    for all $\mathbf{R} \in \mathcal{R}$.
    Once these lower bound have been established, we may complete the proof as follows:
    \begin{align*}
        \sum_{\mathbf{R} \in \mathcal{R}} |\mathbf{A}(\mathbf{R})||\mathbf{B}(\mathbf{R})|
       & \gtrsim K^{-4}\frac{|\mathbf{X}|^{2}}{|\mathbf{X}|_{1 \times \eta \times \eta}|\mathbf{X}|_{\eta \times \eta \times 1}} \sum_{\mathbf{R} \in \mathcal{R}} |P_{\mathbf{R}}|\\
       & \stackrel{\mathclap{\cref{form55}}}{\gtrsim}  K^{-4} \frac{|\mathbf{X}|^{2}}{|\mathbf{X}|_{1 \times \eta \times \eta}|\mathbf{X}|_{\eta \times \eta \times 1}} \sum_{\mathbf{R} \in \mathcal{R}} |\mathbf{X} \cap \mathbf{R}|_{\eta \times \eta \times \eta}\\
       & \stackrel{\mathclap{(\ast)}}{\gtrsim} K^{-5}\frac{|\mathbf{X}|^{2}|\mathbf{X}|_{\eta \times \eta \times \eta}}{|\mathbf{X}|_{1 \times \eta \times \eta}|\mathbf{X}|_{\eta \times \eta \times 1}},
    \end{align*}
    where the estimate $(\ast)$ follows from $K$-uniformity:
    \begin{equation*}
        \sum_{\mathbf{R} \in \mathcal{R}} |\mathbf{X} \cap \mathbf{R}|_{\eta \times \eta \times \eta} \geq \sum_{\mathbf{R} \in \mathcal{R}} \frac{|\mathbf{X} \cap \mathbf{R}|}{M_{\eta \times \eta \times \eta}(\mathbf{X})} \geq \frac{|\mathbf{X}|}{M_{\eta \times \eta \times \eta}(\mathbf{X})} \stackrel{\cref{form56}}{\gtrsim} K^{-1}|\mathbf{X}|_{\eta \times \eta \times \eta}.
    \end{equation*}

    Let us finally prove \cref{form58}, starting with the estimate for $|\mathbf{A}(\mathbf{R})|$.
    First, write
    \begin{equation*}
        |\mathbf{A}(\mathbf{R})| \gtrsim \sum_{z \in A_{\mathbf{R}}} |\{\omega \in \mathbf{X} \cap 2\mathbf{R} : |z_{\omega} - z| \leq 2\eta\}|.
    \end{equation*}
    Fix $z \in A_{\mathbf{R}}$, and recall that $z = z_{\omega_{0}}$ for some $\omega_{0} \in \mathbf{X} \cap \mathbf{R}$.
    Now $\mathbf{X} \cap \mathbf{R}_{\eta \times \eta \times \eta}(\omega_{0}) \subset \{\omega \in \mathbf{X} \cap 2\mathbf{R} : |z_{\omega} - z| \leq 2\eta\}$, so
    \begin{equation*}
        |\{\omega \in \mathbf{X} \cap 2\mathbf{R} : |z_{\omega} - z| \leq 2\eta\}| \geq |\mathbf{X} \cap \mathbf{R}_{\eta \times \eta \times \eta}(\omega_{0})| \gtrsim K^{-1}\frac{|\mathbf{X}|}{|\mathbf{X}|_{\eta \times \eta \times \eta}}
    \end{equation*}
    by $K$-uniformity.
    Recalling from \cref{form59} that $|A_{\mathbf{R}}| \sim |P_{\mathbf{R}}| \gtrsim |\mathbf{X} \cap \mathbf{R}|_{\eta \times \eta \times \eta}$, we get
    \begin{equation*}
        |\mathbf{A}(\mathbf{R})| \gtrsim K^{-1} \frac{|\mathbf{X} \cap \mathbf{R}|_{\eta \times \eta \times \eta}|\mathbf{X}|}{|\mathbf{X}|_{\eta \times \eta \times \eta}}\stackrel{\textup{\cref{form54b}}}{\gtrsim} K^{-3} \frac{|\mathbf{X}|}{|\mathbf{X}|_{1 \times \eta \times \eta}},
    \end{equation*}
    as desired.
    The lower bound for $|\mathbf{B}(\mathbf{R})|$ is more straightforward, and is based on the following observation.
    Fix $z = z_{\omega_{0}} \in B_{\mathbf{R}}$ with $\omega_{0} \in \mathbf{X}$.
    Then,
    \begin{equation*}
        \mathbf{X} \cap \mathbf{R}_{\eta \times \eta \times 1}(\omega_{0}) \subset \{\omega \in \mathbf{X} : |z - z_{\omega}| \leq 2\eta\},
    \end{equation*}
    so $|\{\omega \in \mathbf{X} : |z - z_{\omega}| \leq 2\eta\}| \gtrsim K^{-1}|\mathbf{X}|/|\mathbf{X}|_{\eta \times \eta \times 1}$ by $K$-uniformity.
    Consequently,
     \begin{equation*}
        |\mathbf{B}(\mathbf{R})| \gtrsim \sum_{z \in B_{\mathbf{R}}} |\{\omega \in \mathbf{X}: |z - z_{\omega}| \leq 2\eta\}| \gtrsim K^{-1}|P_{\mathbf{R}}|\cdot \frac{|\mathbf{X}|}{|\mathbf{X}|_{\eta \times \eta \times 1}},
    \end{equation*}
    as desired.
    This completes the proof of \cref{form58}, and therefore the proof of \cref{lemma:initialEstimate}.
\end{proof}

\subsection{Large uniform subsets, branching functions, and effective triples}
In this section we recap concepts and basic results from \cite[Section 3]{zbmath:8038252}, starting with \cite[Lemma 3.6]{zbmath:8038252}:
\begin{lemma}\label{lemma1}
    Let $m,T \in \N$ and $\delta \coloneqq 2^{-mT}$.
    Write
    \begin{equation*}
        K = K(m,T) \coloneqq (Am^{3}\log(1/\delta))^{m^{3} + 2},
    \end{equation*}
    where $A \geq 1$ is a sufficiently large absolute constant.
    For every finite set $\mathbf{X} \subset \Omega$ there exists a subset $\mathbf{X}' \subset \mathbf{X}$ with $|\mathbf{X}'| \geq K^{-1}|\mathbf{X}|$ such that $\mathbf{X}'$ is $K$-uniform on the sequence $\{2^{-jT}\}_{j = 0}^{m}$.
\end{lemma}
\begin{definition}[$(m,T)$-uniformity]\label{def1}
    Let $m,T \in \N$.
    A set $\mathbf{X} \subset \Omega$ is called \emph{$(m,T)$-uniform} if $\mathbf{X}$ is $K$-uniform on $\{2^{-jT}\}_{j = 0}^{m}$ with constant $K = (Am^{3}\log(1/\delta))^{m^{3} + 2}$, where $\delta \coloneqq 2^{-mT}$, and $A \geq 1$ is the absolute constant from \cref{lemma1}.
\end{definition}
From now on, if $m,T \in \N$ are fixed, the symbol ``$K$'' will refer to the constant in \cref{def1}.
A crucial point is that if $m \in \N$ is fixed, and $\delta = 2^{-mT}$, then $K= \delta^{-o_{T \to \infty}(1)}$.

\begin{definition}[Branching function]\label{def:branchingFunction}
    Let $m,T \in \N$, and assume that $\mathbf{X} \subset \Omega$ is $K$-uniform on the sequence $\{2^{-jT}\}_{j = 0}^{m}$.
    Write
    \begin{equation*}
        \mathbf{D}_{m} \coloneqq \{(x,y,z) \in \tfrac{1}{m}\Z^{3} : x,y \in [0,1] \text{ and } z \in [0,\min\{1,x + y\}]\}.
    \end{equation*}
    We define the \emph{branching function $f_{\mathbf{X}} \colon \mathbf{D}_{m} \to [0,\infty)$ of $\mathbf{X}$} via the relation
    \begin{equation*}
        |\mathbf{X}|_{\delta^{x} \times \delta^{z} \times \delta^{y}} = \delta^{-f_{\mathbf{X}}(x,y,z)}, \qquad (x,y,z) \in \mathbf{D}_{m}.
    \end{equation*}
\end{definition}
\begin{remark}
    The well-posedness of $f_{\mathbf{X}}$ does not require the uniformity of $\mathbf{X}$.
    However, note that if $\mathbf{X}$ is uniform, and $(x,y,z) = (i/m,j/m,k/m) \in \mathbf{D}_{m}$, then $\delta^{x} \times \delta^{z} \times \delta^{y} = 2^{-iT} \times 2^{-kT} \times 2^{-jT}$ is one of the scales on which $\mathbf{X}$ is uniform.
\end{remark}

\begin{definition}[Renormalised branching function]
    Let $\mathbf{X},\mathbf{D}_{m},f_{\mathbf{X}}$ be given as in \cref{def:branchingFunction}.
    Fix $(x_{0},y_{0}) \in \tfrac{1}{m}\Z$ such that $(x_{0},y_{0},x_{0} + y_{0}) \in \mathbf{D}_{m}$.
    Define
    \begin{equation*}
        f_{\mathbf{X}}(x,y,z;x_{0},y_{0}) \coloneqq f_{\mathbf{X}}(x_{0} + x,y_{0} + y,x_{0} + y_{0} + z) - f_{\mathbf{X}}(x_{0},y_{0},x_{0} + y_{0}),
    \end{equation*}
    for all triples $(x,y,z) \in \tfrac{1}{m} \Z^{3}$ such that the right hand side is well-defined.
\end{definition}
As the name suggests, $f_{\mathbf{X}}(\cdots;x_{0},y_{0})$ is (roughly speaking) the branching function of the renormalised restriction of $\mathbf{X}$ to a rectangle of the form $\mathbf{R}_{\delta^{x_{0}} \times \delta^{x_{0} + y_{0}} \times \delta^{y_{0}}}(\omega)$.
\cref{lemma2} (which is \cite[Lemma 3.11]{zbmath:8038252}) will make this precise.

\begin{definition}[Rescaling map]\label{def:rescalingMap}
    Let $u,w \in (0,1]$ and $\omega_{0} = (a_{0},b_{0},\sigma_{0}) \in \Omega$.
    For $\mathbf{R} \coloneqq \mathbf{R}_{u \times uw \times w}(\omega_{0})$, we define the \emph{rescaling map}
    \begin{equation*}
        \psi_{\mathbf{R}}(\omega) \coloneqq \Big(\frac{a - a_{0}}{u},\frac{b - (b_{0} + \sigma_{0}(a - a_{0}))}{uw},\frac{\sigma - \sigma_{0}}{w}\Big), \qquad \omega = (a,b,\sigma) \in \Omega.
    \end{equation*}
\end{definition}

\begin{lemma}\label{lemma2}
    Let $m,T \in \N$, and let $\mathbf{X} \subset \Omega$ be $(m,T)$-uniform.
    Let $(x_{0},y_{0},z_{0}) \in \mathbf{D}_{m}$, $\omega_{0} \in \mathbf{X}$, and let $\mathbf{R} \coloneqq \mathbf{R}_{\delta^{x_{0}} \times \delta^{x_{0} + y_{0}} \times \delta^{y_{0}}}(\omega_{0})$.
    Finally, let $\mathbf{X}' \subset \psi_{\mathbf{R}}(\mathbf{X} \cap \mathbf{R})$ be an $(m',T)$-uniform subset with $|\mathbf{X}'| \geq K^{-1}|\mathbf{X} \cap \mathbf{R}|$, where $m' \coloneqq m(1 - x_{0} - y_{0})$.
    Then,
    \begin{equation*}
        |\mathbf{X}'|_{\delta^{x} \times \delta^{z} \times \delta^{y}} \sim_{K^{7}} \frac{|\mathbf{X}|_{\delta^{x_{0} + x} \times \delta^{x_{0} + y_{0} + z} \times \delta^{y_{0} + y}}}{|\mathbf{X}|_{\delta^{x_{0}} \times \delta^{x_{0} + y_{0}} \times \delta^{y_{0}}}} = \delta^{-f_{\mathbf{X}}(x,y,z;x_{0},y_{0})}, \qquad (x,y,z) \in \mathbf{D}_{m'}.
    \end{equation*}
\end{lemma}

\begin{definition}[Effective triple for $\mathbf{X}$]\label{def:effectiveTriple1}
    Let $m,T \in \N$, and let $\mathbf{X} \subset \Omega$ be $(m,T)$-uniform with branching function $f_{\mathbf{X}} \colon \mathbf{D}_{m} \to [0,\infty)$.
    Define the functions auxiliary functions
    \begin{equation}\label{def:b}
        b_{\mathbf{X}}(t;x,y) \coloneqq f_{\mathbf{X}}(t,t,t;x,y)- (f_{\mathbf{X}}(t,0,t;x;y) + f_{\mathbf{X}}(0,t,t;x,y) - t)
    \end{equation}
    and
    \begin{equation}\label{def:e}
        e_{\mathbf{X}}(s;x,y) \coloneqq \tfrac{1}{2}(f_{\mathbf{X}}(s,0,s;x,y) + f_{\mathbf{X}}(0,s,s;x,y) - 3s),
    \end{equation}
    for all parameters $x,y,s,t \in \tfrac{1}{m}\Z$ such that the right hand sides are well-defined.

    Let $(x,y) \in \tfrac{1}{m}\Z^{2}$ and $t \in \tfrac{1}{m}\Z$ with $0 \leq t \leq 1 - (x + y)$, and let $c_{1},c_{2} \geq 0$.
    A triple $(t;x,y)$ is called \emph{$(c_{1},c_{2})$-effective for $\mathbf{X}$} if $\max\{t,x,y\} \leq c_{2} \leq \tfrac{1}{3}$, and
    \begin{equation*}
        b_{\mathbf{X}}(t;x,y) + e_{\mathbf{X}}(s;x,y) \geq c_{1}, \qquad s \in [t,1 - (x + y)].
    \end{equation*}
\end{definition}

\begin{remark}\label{rem4}
    The meaning of the functions $b_{\mathbf{X}},e_{\mathbf{X}}$ is explained well in \cite[Section 4.2]{zbmath:8038252}, but let us say a few words.
    If $x = 0 = y$, then by definition
       \begin{equation*}
        \delta^{-b_{\mathbf{X}}(t;0,0)} = \frac{|\mathbf{X}|_{\delta^{t} \times \delta^{t} \times \delta^{t}}}{\delta^{t}|\mathbf{X}|_{\delta^{t} \times \delta^{t} \times 1}|\mathbf{X}|_{1 \times \delta^{t} \times \delta^{t}}}.
    \end{equation*}
    This is the quantity appearing on the right hand side of the ``initial estimate'', \cref{lemma:initialEstimate} with $\eta = \delta^{t}$.
    So, \cref{lemma:initialEstimate} tells us  roughly speaking that $\mathcal{I}^{+}_{\delta^{t}}(\mathbf{X})/(\delta^{t}|\mathbf{X}|^{2}) \geq \delta^{-\beta_{\mathbf{X}}(t;0,0)}$.

    On the other hand, the function $e_{\mathbf{X}}$ could be called the ``high-low error'', because
    \begin{equation*}
        \delta^{e_{\mathbf{X}}(s;0,0)} = \left(|\mathbf{X}|^{-1}_{\delta^{s} \times \delta^{s} \times 1}|\mathbf{X}|^{-1}_{1 \times \delta^{s} \times \delta^{s}} \delta^{-3s} \right)^{1/2}.
    \end{equation*}
    For uniform sets $\mathbf{X}$, the right hand side roughly matches the right hand side of the high-low inequality (\cref{thm:highLowRays}) with $\mathbf{A} = \mathbf{X} = \mathbf{B}$; we will systematically ignore the eventually harmless $\log w^{-1}$ factor (and various other constants) in this heuristic discussion.
    
     Now, assume that $(t;0,0)$ happens to be a $(c_{1},c_{2})$-effective triple.
     Then, combining the ``initial estimate'' (\cref{lemma:initialEstimate}) and the high-low inequality,
     \begin{align*}
         \frac{\mathcal{I}_{\delta}^{+}(\mathbf{X})}{\delta |\mathbf{X}|^{2}}
         & \geq \frac{\mathcal{I}^{+}_{\delta^{t}}(\mathbf{X})}{\delta^{t}|\mathbf{X}|^{2}} - \sum_{s \in [t,1]} \left(|\mathbf{X}|^{-1}_{\delta^{s} \times \delta^{s} \times 1}|\mathbf{X}|^{-1}_{1 \times \delta^{s} \times \delta^{s}} \delta^{-3s} \right)^{1/2}\\
         & \geq \delta^{-b_{\mathbf{X}}(t;0,0)}\Big(1 - \sup_{s \in [t,1]}\delta^{b_{\mathbf{X}}(t;0,0) + e_{\mathbf{X}}(s;0,0)} \Big) \geq \delta^{-b_{\mathbf{X}}(t;0,0)}(1 - \delta^{c_{1}}).
     \end{align*}
     Provided $\delta^{c_{1}} \leq \tfrac{1}{2}$, this shows that $\mathcal{I}_{\delta}^{+}(\mathbf{X})/(\delta |\mathbf{X}|^{2}) \geq \tfrac{1}{2}\delta^{-b_{\mathbf{X}}(t;0,0)}$.
     Moreover, $\delta^{-b_{\mathbf{X}}(t;0,0)} \geq \delta^{O(c_{2})}$ thanks to the (effective triple) assumption $t \leq c_{2}$, and the Lipschitz continuity of $t \mapsto b_{\mathbf{X}}(t;0,0)$.
     In summary, a $(c_{1},c_{2})$-effective triple of the form $(t;0,0)$ yields $\mathcal{I}_{\delta}^{+}(\mathbf{X})/(\delta |\mathbf{X}|^{2}) \geq \delta^{O(c_{2})}$.
     The same conclusion remains true for all $(c_{1},c_{2})$-effective triples $(t;x,y)$, roughly speaking by applying the argument above to $\mathbf{X} \cap \mathbf{R}$, where $\mathbf{R}$ is a $(\delta^{x} \times \delta^{x + y} \times \delta^{y})$-rectangle.
     For the details, see \cref{prop1}.
 \end{remark}

For the proof of \cref{prop1}, we record a corollary of the high-low theorem for rays (\cref{thm:highLowRays}) for uniform sets.
This is a counterpart of \cite[Proposition 4.1]{zbmath:8038252}.

\begin{corollary}\label{cor1}
    Let $w \in 2^{-\N} \cap (0,\tfrac{1}{2}]$, $K \geq 1$, and let $\mathbf{X} \subset \Omega$ be a finite set which is $K$-uniform on the scales $\{w \times w \times 1,1 \times w \times w\}$.
    Then, for every $A \in 2^{\N}$ with $Aw \leq \tfrac{1}{2}$, it holds
    \begin{equation}\label{form61}
        \Big| \frac{\mathcal{J}^{+}_{w}(\mathbf{X})}{w|\mathbf{X}|^{2}} - \frac{\mathcal{J}^{+}_{w/A}(\mathbf{X})}{(w/A)|\mathbf{X}|^{2}} \Big| \lesssim KA^{2}\left(|\mathbf{X}|_{w \times w \times 1}^{-1}|\mathbf{X}|_{1 \times w \times w}^{-1}w^{-3}\log w^{-1} \right)^{1/2},
    \end{equation}
    and
    \begin{equation}\label{form62}\Big|
        \frac{\mathcal{J}^{+}_{w}(\mathbf{X})}{w|\mathbf{X}|^{2}} - \frac{\mathcal{J}^{+}_{Aw}(\mathbf{X})}{(Aw)|\mathbf{X}|^{2}} \Big| \lesssim KA^{1/2}\left(|\mathbf{X}|_{w \times w \times 1}^{-1}|\mathbf{X}|_{1 \times w \times w}^{-1}w^{-3}\log w^{-1} \right)^{1/2}.
    \end{equation}
    The implicit constants depend on the function ``$\chi$'' for which $\mathcal{J}^{+}$ has been defined, see \cref{form40}.
\end{corollary}
\begin{proof}
    Fix $2^{j} \in 2^{\N} \cap [1,A/2]$.
    Applying \cref{thm:highLowRays} with $\mathbf{A} = \mathbf{X} = \mathbf{B}$, and using the trivial estimates $M_{2^{-j}w \times 2^{-j}w}(\mathbf{X}) \leq M_{w \times w}(\mathbf{X})$ and $M_{1 \times 2^{-j}w}(\mathbf{X}) \leq M_{1 \times w}(\mathbf{X})$,
    \begin{equation}\label{form60a}
        \Big| \frac{\mathcal{J}^{+}_{2^{-j}w}(\mathbf{X})}{(2^{-j}w)|\mathbf{X}|^{2}} - \frac{\mathcal{J}^{+}_{2^{-j - 1}w}(\mathbf{X})}{(2^{-j - 1}w)|\mathbf{X}|^{2}} \Big| \lesssim \Big(\frac{M_{w \times w}(\mathbf{X})M_{1 \times w}(\mathbf{X})}{|\mathbf{X}|^{2}} (2^{-j}w)^{-3} \log (2^{-j}w)^{-1} \Big)^{1/2}.
    \end{equation}
    To proceed, note that $M_{1 \times w}(\mathbf{X})/|\mathbf{X}| \lesssim K|\mathbf{X}|_{w \times w \times 1}^{-1}$ and $M_{1 \times w}(\mathbf{X})/|\mathbf{X}| \lesssim K|\mathbf{X}|_{1 \times w \times w}^{-1}$ by $K$-uniformity, see \cite[(27)]{zbmath:8038252}.
    We mention that $M_{w \times w}(\mathbf{X})$ and $M_{1 \times w}(\mathbf{X})$ defined at \cref{form29}-\cref{form29a} are \emph{a priori} different from $M_{w \times w \times 1}(\mathbf{X})$ and $M_{1 \times w \times w}(\mathbf{X})$ appearing in \cite[(27)]{zbmath:8038252}.
    However, $M_{w \times w}(\mathbf{X}) \sim M_{w \times w \times 1}(\mathbf{X})$ and $M_{1 \times w}(\mathbf{X}) \sim M_{1 \times w \times w}(\mathbf{X})$ by \cite[(21)]{zbmath:8038252}.

    Note also the general inequality $(ab)^{3}\log (ab) \leq a^{4}b^3\log b$, valid for $a \geq 1$ and $b \geq 2$ (in particular $a = 2^{j}$ and $b = w^{-1}$).
    Substituting these estimates back into \cref{form60a},
    \begin{equation*}
        \Big| \frac{\mathcal{J}^{+}_{2^{-j}w}(\mathbf{X})}{(2^{-j}w)|\mathbf{X}|^{2}} - \frac{\mathcal{J}^{+}_{2^{-j - 1}w}(\mathbf{X})}{(2^{-j - 1}w)|\mathbf{X}|^{2}} \Big| \lesssim 2^{2j}K\left( |\mathbf{X}|_{w \times w \times 1}^{-1}|\mathbf{X}|_{1 \times w \times w}^{-1} w^{-3} \log w^{-1} \right)^{1/2}.
    \end{equation*}
    Summing this estimate over $2^{j} \in [1,A]$ yields \cref{form61}.

    For \cref{form62}, we also fix $2^{j} \in 2^{\N} \cap [1,A/2]$, and start with
    \begin{equation*}
        \Big| \frac{\mathcal{J}^{+}_{2^{j}w}(\mathbf{X})}{(2^{j}w)|\mathbf{X}|^{2}} - \frac{\mathcal{J}^{+}_{2^{j + 1}w}(\mathbf{X})}{(2^{j + 1}w)|\mathbf{X}|^{2}} \Big| \lesssim \Big(\frac{M_{2^{j + 1}w \times 2^{j + 1}w}(\mathbf{X})M_{1 \times 2^{j + 1}w}(\mathbf{X})}{|\mathbf{X}|^{2}} (2^{j}w)^{-3} \log (2^{j}w)^{-1} \Big)^{1/2}.
    \end{equation*}
      Now we estimate $M_{2^{j + 1}w \times 2^{j + 1}w}(\mathbf{X}) \lesssim 2^{2j}M_{w \times w}(\mathbf{X})$ and $M_{1 \times 2^{j + 1}w}(\mathbf{X}) \lesssim 2^{2j}M_{1 \times w}(\mathbf{X})$, and $(2^{j}w)^{-3}\log (2^{j}w)^{-1} \leq 2^{-3j}w^{-3}\log w^{-1}$ to get
    \begin{equation*}
        \Big| \frac{\mathcal{J}^{+}_{2^{j}w}(\mathbf{X})}{(2^{j}w)|\mathbf{X}|^{2}} - \frac{\mathcal{J}^{+}_{2^{j + 1}w}(\mathbf{X})}{(2^{j + 1}w)|\mathbf{X}|^{2}} \Big| \lesssim 2^{j/2}K\left(|\mathbf{X}|_{w \times w \times 1}^{-1}|\mathbf{X}|_{1 \times w \times w}^{-1}w^{-3}\log w^{-1} \right)^{1/2}.
    \end{equation*}
    Finally, \cref{form62} follows by summing over $2^{j} \in 2^{\N} \cap [1,A/2]$.
\end{proof}

\subsection{Effective triples yield ray incidences}
\cref{prop1} below is our counterpart of \cite[Proposition 4.2]{zbmath:8038252} where line incidences have been replaced by ray incidences.

 \begin{proposition}\label{prop1}
     Let $c_{1},c_{2} > 0$.
     Then there exist $m_{0} = m_{0}(c_{1}) \in \N$ such that the following holds for all $m \geq m_{0}$ and $T \geq T_{0}(c_{1},c_{2},m)$:

     Write $\delta \coloneqq 2^{-mT}$.
     Let $\mathbf{X} \subset \Omega$ be a finite $(m,T)$-uniform set.
     Assume that there exists a triple $(t;x,y)$ which is $(c_{1},c_{2})$-effective for $\mathbf{X}$.
     Then,
     \begin{equation}\label{form69}
         \mathcal{I}_{10\delta}^{+}(\mathbf{X}) \geq \delta^{1 + 33c_{2}}|\mathbf{X}|^{2}.
     \end{equation}
\end{proposition}
\begin{proof}
    The argument is a rigorous version of the sketch provided in \cref{rem4}.
    We start by fixing parameters.
    First, let $m_{0} = m_{0}(c_{1}) \in \N$ be so large that
    \begin{equation}\label{form65}
        3/m_{0} < c_{1}/2.
    \end{equation}
    Fix $m \geq m_{0}$, and let $a,c > 0$ and $C \geq 1$ be absolute constants to be determined later.
    Then, let $T_{0} = T_{0}(m,c_{1},c_{2}) \in \N$ be so large that the following holds:
    \begin{equation}\label{form66}
        \max\{2CK^{30}a^{-2}\delta^{c_{1}/2}(\log \tfrac{1}{\delta})^{2}, 2K^{21}\delta^{c_{2}}\} \leq c, \qquad T \geq T_{0}.
    \end{equation}
    This is possible, since $K = \delta^{-o_{T \to \infty}(1)}$ for $m \in \N$ fixed, recalling that $\delta = 2^{-mT}$.
    There will be additional requirements on the size of $T$, depending on $m,c_{1},c_{2}$, to be described later.

    We claim that we may assume, without loss of generality, that
    \begin{equation}\label{form73}
        |\mathbf{X}|_{\delta \times \delta \times 1} \geq \delta^{-1 - 32c_{2}}.
    \end{equation}
    The point is that \cref{form69} is easy to show in the opposite case.
    Assume that \cref{form73} fails.
    Let $b > 0$ be an absolute constant to be determined in a moment.
    Let $P_{\delta} \subset P[\mathbf{X}]$ be a maximal $10\delta$-separated subset, so $|P_{\delta}| \sim |\mathbf{X}|_{\delta \times \delta \times 1}$.
    Then, note that
    \begin{equation*}
        \mathcal{I}^{+}_{10\delta}(\mathbf{X}) \geq \sum_{z \in P_{\delta}} |\{(\omega_{1},\omega_{2}) \in \mathbf{X}^{2} : \max\{|z_{\omega_{1}} - z|,|z_{\omega_{2}} - z| < 5\delta\}\}|,
    \end{equation*}
    since the sets on the right are disjoint, and each pair $(\omega_{1},\omega_{2})$ in any of these sets satisfies $\dist(z_{\omega_{2}},\ell_{\omega_{1}}^{+}) \leq 10\delta$.
    Next, note that for $z = z_{\omega_{0}} \in P_{\delta}$ fixed with $\omega_{0} \in \mathbf{X}$,
    \begin{align*}
        |\{(\omega_{1},\omega_{2}) \in \mathbf{X}^{2} : \max\{|z_{\omega_{1}} - z|,|z_{\omega_{2}} - z| < 5\delta\}|
        & = |\{\omega \in \mathbf{X} : |z_{\omega} - z| < 5\delta\}|^{2}\\
        & \stackrel{\mathclap{\cref{form76}}}{\geq} |\mathbf{X} \cap \mathbf{R}_{\delta \times \delta \times 1}(\omega_{0})|^{2}\\
        & \geq K^{-2}M_{\delta \times \delta \times 1}(\mathbf{X})^{2}
    \end{align*}
    by $K$-uniformity.
    Thus, using also $|\mathbf{X}| \leq M_{\delta \times \delta \times 1}(\mathbf{X})|\mathbf{X}|_{\delta \times \delta \times 1} \leq M_{\delta \times \delta \times 1}(\mathbf{X})\delta^{-1 - 32c_{2}}$ (by hypothesis), we obtain
    \begin{equation*}
        \mathcal{I}_{10\delta}^{+}(\mathbf{X}) \gtrsim K^{-2}|\mathbf{X}|_{\delta \times \delta \times 1}M_{\delta \times \delta \times 1}(\mathbf{X})^{2} \geq K^{-2} |\mathbf{X}|M_{\delta \times \delta \times 1}(\mathbf{X}) \geq K^{-2}\delta^{1 + 32c_{2}}|\mathbf{X}|^{2}.
    \end{equation*}
    This proves \cref{form69} for $T \in \N$ large enough (depending on $c_{2}$).
    In the sequel we may therefore assume \cref{form73}.

    Fix $\omega_{0} \in \mathbf{X}$ arbitrarily, and let $\mathbf{R} \coloneqq \mathbf{R}_{\delta^{x} \times \delta^{x + y} \times \delta^{y}}(\omega_{0})$.
    As in \cref{lemma2}, let $\mathbf{X}' \subset \psi_{\mathbf{R}}(\mathbf{X} \cap \mathbf{R})$ be an $(m',T)$-uniform subset with $|\mathbf{X}'| \geq K^{-1}|\mathbf{X} \cap \mathbf{R}|$, where $m' = m(1 - x - y)$ (this subset is provided by \cref{lemma1}).
    By \cref{lemma2},
    \begin{equation}\label{form60}
        |\mathbf{X}'|_{\delta^{x'} \times \delta^{z'} \times \delta^{y'}} \sim_{K^{7}} \delta^{-f_{\mathbf{X}}(x',y',z';x,y)}, \qquad (x',y',z') \in \mathbf{D}_{m'}.
    \end{equation}
    By \cref{lemma:initialEstimate},
    \begin{equation}\label{form64}
        \frac{\mathcal{J}^{+}_{10\delta^{t}}(\mathbf{X}')}{10\delta^{t}|\mathbf{X}'|^{2}} \geq \frac{\mathcal{I}_{10\delta^{t}}^{+}(\mathbf{X}')}{10\delta^{t}|\mathbf{X}'|^{2}} \gtrsim_{K^{5}} \frac{|\mathbf{X}'|_{\delta^{t} \times \delta^{t} \times \delta^{t}}}{\delta^{t}|\mathbf{X}'|_{\delta^{t} \times \delta^{t} \times 1}|\mathbf{X}'|_{1 \times \delta^{t} \times \delta^{t}}} \stackrel{\cref{form60}}{\sim_{K^{21}}} \delta^{-b_{\mathbf{X}}(t;x,y)}.
    \end{equation}
    Here, and below, $\mathcal{J}^{+}$ is defined for some function $\chi \in C_{c}^{\infty}(\R^{2})$ satisfying $\mathbf{1}_{B(1/2)} \leq \chi \leq \mathbf{1}_{B(1)}$; now the first inequality above follows from \cref{lemma:incidenceComparison}.

    Recall the (small) absolute constant $a > 0$ defined above \cref{form66}.
    Recall that $\delta = 2^{-mT}$, abbreviate $\bar{\delta} \coloneqq \delta^{1 - x - y}$, and write
    \begin{equation}\label{form63}
        \begin{aligned}
            \Big| \frac{\mathcal{J}^{+}_{10\delta^{t}}(\mathbf{X}')}{10\delta^{t}|\mathbf{X}'|^{2}}
        & - \frac{\mathcal{J}^{+}_{a\bar{\delta}}(\mathbf{X}')}{a\bar{\delta}|\mathbf{X}'|^{2}} \Big| \leq \Big| \frac{\mathcal{J}^{+}_{10\delta^{t}}(\mathbf{X}')}{10\delta^{t}|\mathbf{X}'|^{2}} - \frac{\mathcal{J}^{+}_{\delta^{t}}(\mathbf{X}')}{\delta^{t}|\mathbf{X}'|^{2}}\Big| + \Big| \frac{\mathcal{J}^{+}_{a\bar{\delta}}(\mathbf{X}')}{a\bar{\delta}|\mathbf{X}'|^{2}} - \frac{\mathcal{J}^{+}_{\bar{\delta}}(\mathbf{X}')}{\bar{\delta}|\mathbf{X}'|^{2}}\Big| \\*
        &\quad + \sum_{j = mt}^{m(1 - x - y) - 1} \Big| \frac{\mathcal{J}^{+}_{2^{-jT}}(\mathbf{X}')}{2^{-jT}|\mathbf{X}'|^{2}} - \frac{\mathcal{J}^{+}_{2^{-(j + 1)T}}(\mathbf{X}')}{2^{-(j + 1)T}|\mathbf{X}'|^{2}} \Big|.
        \end{aligned}
    \end{equation}
    We will apply \cref{cor1} to each of the terms with constants $A \in \{2^{T},1/10,1/a\}$.
    Writing $2^{-jT} = \delta^{j/m}$, using \cref{form60}, and estimating $j \leq m$,
    \begin{align*}\Big|
        \frac{\mathcal{J}^{+}_{2^{-jT}}(\mathbf{X}')}{2^{-jT}|\mathbf{X}'|^{2}} - \frac{\mathcal{J}^{+}_{2^{-(j + 1)T}}(\mathbf{X}')}{2^{-(j + 1)T}|\mathbf{X}'|^{2}} \Big|
        & \lesssim K2^{2T}\left( |\mathbf{X}'|_{2^{-jT} \times 2^{-jT} \times 1}^{-1}|\mathbf{X}'|^{-1}_{1 \times 2^{-jT} \times 2^{-jT}} 2^{3jT} \log 2^{jT} \right)^{1/2}\\*
        & \lesssim_{K^{8}} m2^{3T} \cdot \delta^{e_{\mathbf{X}}(j/m;x,y)}, \quad j \in \{mt,\ldots,m(1 - x - y) -1\}.
    \end{align*}
    Similarly, using $\log \bar{\delta}^{-1} \leq \log \delta^{-1} = mT$,
    \begin{align*}
        \Big| \frac{\mathcal{J}^{+}_{a\bar{\delta}}(\mathbf{X}')}{a\bar{\delta}|\mathbf{X}'|^{2}} - \frac{\mathcal{J}^{+}_{\bar{\delta}}(\mathbf{X}')}{\bar{\delta}|\mathbf{X}'|^{2}}\Big|
        &\lesssim Ka^{-2}\left(|\mathbf{X}|^{-1}_{\bar{\delta} \times \bar{\delta} \times 1}|\mathbf{X}|^{-1}_{1 \times \bar{\delta} \times \bar{\delta}} \bar{\delta}^{-3} \log \bar{\delta}^{-1}\right)^{1/2}\\*
        &\lesssim_{K^{8}} a^{-2}mT \cdot \delta^{e_{\mathbf{X}}(1 - x - y;x,y)}.
    \end{align*}
    Substituting these estimates back into \cref{form63}, and also recalling \cref{form64}, we find, for some absolute constants $c > 0$ and $C \geq 1$,
    \begin{equation}\label{form67}
        \begin{aligned}
            \frac{\mathcal{J}^{+}_{a\bar{\delta}}(\mathbf{X}')}{a\bar{\delta}|\mathbf{X}'|^{2}}
        & \geq cK^{-21} \delta^{-b_{\mathbf{X}}(t;x,y)} - CK^{8}a^{-2}m2^{3T} \sum_{j = mt}^{m(1 - x - y)}\delta^{e_{\mathbf{X}}(j/m;x,y)}\\*
        & \geq \delta^{-b_{\mathbf{X}}(t;x,y)} (cK^{-21} - Cm^{2}K^{8}2^{3T}\delta^{b_{\mathbf{X}}(t;x,y)} \sup_{t \leq s \leq 1 - x - y} \delta^{e_{\mathbf{X}}(s;x,y)}).
        \end{aligned}
    \end{equation}
    This estimate determines the absolute constants $c,C$ in \cref{form66}.
    Since $(t;x,y)$ is a $(c_{1},c_{2})$-effective triple, above
    \begin{equation*}
        \delta^{b_{\mathbf{X}}(t;x,y)} \sup_{t \leq s \leq 1 - x - y} \delta^{e_{\mathbf{X}}(s;x,y)} \leq \delta^{c_{1}}.
    \end{equation*}
    Next, note that $2^{3T} = \delta^{-3/m} \leq \delta^{-3/m_{0}} \leq \delta^{-c_{1}/2}$ according to \cref{form65}.
    Since furthermore $m \leq mT = \log \tfrac{1}{\delta}$, we find
    \begin{equation*}
        Cm^{2}K^{8}2^{3T}\delta^{b_{\mathbf{X}}(t;x,y)} \sup_{t \leq s \leq 1 - x - y} \delta^{e_{\mathbf{X}}(s;x,y)} \leq CK^{8}\delta^{c_{1}/2}(\log \tfrac{1}{\delta})^{2} \stackrel{\cref{form66}}{\leq} \tfrac{c}{2}K^{-21}.
    \end{equation*}
    Substituting this upper bound back into \cref{form67}, and also using \cref{lemma:incidenceComparison},
    \begin{equation}\label{form68}
        \frac{\mathcal{I}^{+}_{2a\delta^{1 - x - y}}(\mathbf{X}')}{a\delta^{1 - x - y}|\mathbf{X}'|^{2}} \gtrsim \frac{\mathcal{J}^{+}_{a\bar{\delta}}(\mathbf{X}')}{a\bar{\delta}|\mathbf{X}'|^{2}} \geq \tfrac{c}{2}K^{-21}\delta^{-b_{\mathbf{X}}(t;x,y)} \stackrel{(\ast)}{\geq} \tfrac{c}{2}K^{-21}\delta^{10c_{2}} \stackrel{\mathclap{\cref{form66}}}{\geq} \delta^{11c_{2}}.
    \end{equation}
    The estimate $(\ast)$, namely $\delta^{-b_{\mathbf{X}}(t;x,y)} \geq \delta^{10c_{2}}$, follows from the hypothesis (in the definition of effective triples) that $\max\{x,y,t\} \leq c_{2}$, and the Lipschitz property of the branching function $f_{\mathbf{X}}$, see \cite[Lemma 3.8]{zbmath:8038252}.
    Here we also need $T$ to be large enough in terms of $c_{2}$.

    Next, as in the proof of \cite[Proposition 4.2]{zbmath:8038252}, we would like to show that ``rescaling preserves incidences'', and therefore \cref{form68} implies \cref{form69}.
    There is a small technicality here (explained in \cref{rem3}) which requires us to perform an additional argument.
    Unwrapping \cref{form68} (and estimating $\delta^{1 - x - y} \geq \delta$), we have so far shown that
    \begin{equation}\label{form70}
        |\{(\omega_{1},\omega_{2}) \in \mathbf{X}' \times \mathbf{X}' : \dist(z_{\omega_{2}},\ell_{\omega_{1}}^{+}) \leq 2a\bar{\delta}\}| \gtrsim a\delta^{1 + 11c_{2}}|\mathbf{X}'|^{2} \geq \delta^{1 + 24c_{2}}|\mathbf{X}|^{2}.
    \end{equation}
    The final inequality follows from $|\mathbf{X}'| \gtrsim_{K^{2}} M_{\delta^{x} \times \delta^{x + y} \times \delta^{y}}(\mathbf{X}) \geq |\mathbf{X}|/|\mathbf{X}|_{\delta^{x} \times \delta^{x + y} \times \delta^{y}} \gtrsim \delta^{3(x + y)}|\mathbf{X}|$ and $\max\{x,y\} \leq c_{2}$, thus $a|\mathbf{X}'|^{2} \geq \delta^{13c_{2}}|\mathbf{X}|^{2}$ if $T \geq 1$ is large enough in terms of $c_{2}$.
    For reasons to become apparent in a moment, we need to upgrade \cref{form70} to the following:
    \begin{equation}\label{form71}
        |\{(\omega_{1},\omega_{2}) \in \mathbf{X}' \times \mathbf{X}' : x_{\omega_{1}} \leq x_{\omega_{2}} \text{ and } \dist(z_{\omega_{2}},\ell_{\omega_{1}}^{+}) \leq 2a\bar{\delta}\}| \gtrsim \delta^{1 + 24c_{2}}|\mathbf{X}|^{2},
    \end{equation}
    where $x_{\omega_{j}}$ refers to the first coordinate of $z_{\omega_{j}}$.
    To see that this can be done, note that if $x_{\omega_{1}} > x_{\omega_{2}}$ and $\dist(z_{\omega_{2}},\ell_{\omega_{1}}^{+}) \leq 3a\bar{\delta}$, then $|z_{\omega_{1}} - z_{\omega_{2}}| \lesssim a\bar{\delta} = a\delta^{1 - x - y}$.
    So, the difference set of the pairs in \cref{form70} and \cref{form71} is contained in the set of ``nearby pairs''
    \begin{equation*}
        \mathcal{N} \coloneqq \{(\omega_{1},\omega_{2}) \in \mathbf{X}' \times \mathbf{X}' : |z_{\omega_{1}} - z_{\omega_{2}}| \lesssim a\delta^{1 - x - y}\}.
    \end{equation*}
    Further, for $\omega_{1} \in \mathbf{X}'$ fixed and $a > 0$ small enough, the set $\{\omega_{2} \in \mathbf{X}' : |z_{\omega_{2}} - z_{\omega_{1}}| \lesssim a\delta^{1 - x - y}\}$ is contained in $\mathbf{X}' \cap \mathbf{R}_{\delta^{1 - x - y} \times \delta^{1 - x - y} \times 1}(\omega_{1})$.
    Therefore, by the $K$-uniformity of $\mathbf{X}'$,
    \begin{align*}
        |\mathcal{N}| \leq |\mathbf{X}'|M_{\delta^{1 - x - y} \times \delta^{1 - x - y} \times 1}(\mathbf{X}')  & \lesssim_{K} \frac{|\mathbf{X}|^{2}}{|\mathbf{X}'|_{\delta^{1 - x - y} \times \delta^{1 - x - y} \times 1}}\\
                                                                                                                & \stackrel{\cref{form60}}{\lesssim_{K^{7}}} \frac{|\mathbf{X}|^{2}|\mathbf{X}|_{\delta^{x} \times \delta^{x + y} \times \delta^{y}}}{|\mathbf{X}|_{\delta^{1 - y} \times \delta \times \delta^{y}}}.
    \end{align*}
    Here moreover $|\mathbf{X}|_{\delta^{x} \times \delta^{x + y} \times \delta^{y}} \leq |\mathbf{X}|_{\delta^{2c_{2}} \times \delta^{2c_{2}} \times \delta^{2c_{2}}} \lesssim \delta^{-6c_{2}}$, and
    \begin{equation*}
        |\mathbf{X}|_{\delta^{1 - y} \times \delta \times \delta^{y}} \geq |\mathbf{X}|_{\delta^{1 - y} \times \delta \times 1} \gtrsim \delta^{y}|\mathbf{X}|_{\delta \times \delta \times 1} \geq \delta^{c_{2}}|\mathbf{X}|_{\delta \times \delta \times 1} \stackrel{\cref{form73}}{\geq} \delta^{-1 - 31c_{2}}.
    \end{equation*}
    Altogether, $|\mathcal{N}| \lesssim_{K^{7}} \delta^{1 + 25c_{2}}|\mathbf{X}|^{2}$.
    Comparing this with the lower bound \cref{form70}, we conclude that \cref{form71} holds, provided $T$ is large enough in terms of $c_{2}$.

    We may now complete the proof of \cref{form69}.
    Recall that $\mathbf{X}' \subset \psi_{\mathbf{R}}(\mathbf{X} \cap \mathbf{R})$, where $\mathbf{R} = \mathbf{R}_{\delta^{x} \times \delta^{x + y} \times \delta^{y}}(a_{0},b_{0},\sigma_{0})$.
    Fix a pair $(\omega_{1}',\omega_{2}') = (\psi_{\mathbf{R}}(\omega_{1}),\psi_{\mathbf{R}}(\omega_{2})) \in \mathbf{X}' \times \mathbf{X}'$ such that
    \begin{equation*}
        x_{\omega_{1}'} \leq x_{\omega_{2}'} \quad \text{and} \quad \dist(z_{\omega_{2}'},\ell_{\omega_{1}'}^{+}) \leq 2a\delta^{1 - x - y}.
    \end{equation*}
    It is shown at the end of the proof of \cite[Proposition 4.2]{zbmath:8038252} that $\dist(z_{\omega_{2}},\ell_{\omega_{1}}) \leq \delta$, provided that $a > 0$ is a small enough (absolute) constant.
    To show \emph{a fortiori} that $\dist(z_{\omega_{2}},\ell_{\omega_{1}}^{+}) \leq \delta$, it suffices to demonstrate that $x_{\omega_{1}} \leq x_{\omega_{2}}$.
    Write $\omega_{j} = (a_{j},b_{j},\sigma_{j})$ for $j \in \{1,2\}$.
    With this notation $x_{\omega_{j}} = a_{j}$, and (recalling the form of $\psi_{\mathbf{R}}$ from \cref{def:rescalingMap}),
    \begin{equation}\label{form74}
        \frac{a_{1} - a_{0}}{\delta^{x}} = x_{\omega_{1}'} \leq x_{\omega_{2}'} = \frac{a_{2} - a_{0}}{\delta^{x}}.
    \end{equation}
    Therefore $x_{\omega_{1}} \leq x_{\omega_{2}}$, as desired.
    Now \cref{form71} implies that $\mathcal{I}^{+}_{\delta}(\mathbf{X}) \geq \delta^{1 + 24c_{2}}|\mathbf{X}|^{2}$, which is stronger than \cref{form69}.
    The proof of \cref{prop1} is complete.
\end{proof}

\begin{remark}\label{rem3}
    Why did we need to upgrade \cref{form70} to \cref{form71}? Assume that $(\omega_{1}',\omega_{2}') = (\psi_{\mathbf{R}}(\omega_{1}),\psi_{\mathbf{R}}(\omega_{2}))$ is a pair satisfying $\dist(z_{\omega_{2}'},\ell_{\omega_{1}'}^{+}) \leq \delta^{1 - x - y}$ (but perhaps not $x_{\omega_{1}'} \leq x_{\omega_{2}'}$).
    From the information $\dist(z_{\omega_{2}'},\ell_{\omega_{1}'}^{+}) \leq \delta^{1 - x - y}$ alone we may infer that $x_{\omega_{2}'} \geq x_{\omega_{1}'} - \delta^{1 - x - y}$.
    Using the relation \cref{form74} this yields $x_{\omega_{2}} \geq x_{\omega_{1}} - \delta^{1 - y}$.
    Since $y > 0$, this information is weaker than $x_{\omega_{2}} \geq x_{\omega_{1}} - \delta$, which is what we would need to conclude $\dist(z_{\omega_{2}},\ell_{\omega_{1}}^{+}) \leq \delta$.
\end{remark}

\subsection{Completing the proof of the upper bound}

With \cref{prop1} in hand, the rest of the proof of \cref{thm:CPZ2} is the same as the proof of \cite[Theorem 1.9]{zbmath:8038252} (or \cref{thm:CPZ}).
We recap the main steps.
First, following \cite[Section 3.9]{zbmath:8038252}, one defines the space $\mathcal{L}$ of \emph{limiting branching functions} $f \colon \mathbf{D} \to [0,\infty)$, where $\mathbf{D} = \{(x,y,z) : x,y \in [0,1] \text{ and } z \in [0,\min\{1,x + y\}]\}$.
We will not need the specifics of the definition here, so we refer the reader to \cite[Section 3.9]{zbmath:8038252}.

\begin{definition}[Effective triples for limiting branching functions]\label{def:effectiveTriple2}
    Let $c_{1}\geq 0$, $c_{2} \geq 0$, and $f \in \mathcal{L}$.
    A triple $(t;x,y)$ with $0 \leq t \leq 1 - (x + y)$ is \emph{$(c_{1},c_{2})$-effective for $f$} if $\max\{t,x,y\} \leq c_{2} \leq \tfrac{1}{3}$, and
    \begin{equation*}
        b(t;x,y) + e(s;x,y) > c_{1}, \qquad t \leq s \leq 1 - (x + y).
    \end{equation*}
\end{definition}
The functions $b,e$ are defined by the same formulae as in \cref{def:b} and \cref{def:e}.

This definition of effective triples is taken from \cite[Section 4.3]{zbmath:8038252}.
Next, we follow the notation in \cite[Section 4.3]{zbmath:8038252} verbatim to define the family $\mathcal{L}^{\mathrm{good}} \subset \mathcal{L}$.
This subset consists of the functions $f \in \mathcal{L}$ such that for every $c_{2} > 0$, there exist $c_{1} > 0$ and $(t;x,y)$ such that $(t;x,y)$ is a $(c_{1},c_{2})$-effective triple for $f$.
In other words, $f$ has effective triples for arbitrarily small values of ``$c_{2}$'' while also retaining $c_{1} > 0$.

Next, following \cite[Section 5.1]{zbmath:8038252}, we define $\mathcal{L}_{\alpha,\beta} \subset \mathcal{L}$, for $\alpha,\beta \geq 0$, to consist of the functions $f \in \mathcal{L}$ such that $f(x,y,x + y) \geq \alpha x + \beta y$ for all $(x,y) \in \mathbf{D}$.
With this notation, \cite[Theorem 5.1]{zbmath:8038252} (one of the main results in \cite{zbmath:8038252}) is the following statement:
\begin{theorem}\label{thm3}
    If $\alpha,\beta \in (1,2]$ with $\alpha + \beta > 3$, then $\mathcal{L}_{\alpha,\beta} \subset \mathcal{L}^{\mathrm{good}}$.
\end{theorem}

Next, we state \cite[Proposition 4.4]{zbmath:8038252} for ray incidences:
\begin{proposition}\label{prop2}
    Let $\{(\mathbf{X}_{k},\delta_{k})\}_{k= 1}^{\infty}$ be a sequence, where each $\mathbf{X}_{k} \subset \Omega$ is a finite set, $\delta_{k} \searrow 0$, and every limiting branching function for $\{(\mathbf{X}_{k},\delta_{k})\}_{k= 1}^{\infty}$ is in $\mathcal{L}^{\mathrm{good}}$.
    (The precise meaning of this final hypothesis is clarified above \cite[Proposition 4.4]{zbmath:8038252}, but we do not need the details here.)
    Then, for every $\varepsilon > 0$ there exists $k_{0}(\varepsilon) \in \N$ such that
    \begin{equation*}
        \mathcal{I}^{+}_{\delta_{k}}(\mathbf{X}_{k}) \geq \delta_{k}^{1 + \varepsilon}|\mathbf{X}_{k}|^{2}, \qquad k \geq k_{0}(\varepsilon).
    \end{equation*}
\end{proposition}
The proof of \cref{prop2} is the same as the proof of \cite[Proposition 4.4]{zbmath:8038252}, and the difference between $\mathcal{I}^{+}_{\delta_{k}}(\mathbf{X}_{k})$ and $\mathcal{I}_{\delta_{k}}(\mathbf{X}_{k})$ has no effect on the argument (except that one needs to apply \cref{prop1} inside the proof in place of \cite[Proposition 4.2]{zbmath:8038252}).

We now have all the ingredients in place to prove \cref{thm:CPZ2}.
This short argument may be copied from \cite{zbmath:8038252} verbatim, and we do so for completeness.
\begin{proofref}{thm:CPZ2}
    Fix $\alpha,\beta \in (1,2]$ with $\alpha + \beta > 3$, and assume that the conclusion fails.
    Then, for some $\varepsilon > 0$ there exist sequences $\eta_{k} \to 0$ and $\delta_{k} \to 0$, and a sequence of Frostman $(\delta_{k},\alpha,\beta,\delta_k^{-\eta_{k}})$-sets $\mathbf{X}_{k} \subset \Omega$ such that $\mathcal{I}^{+}_{\delta_{k}}(\mathbf{X}_{k}) \leq \delta_{k}^{1 + \varepsilon}|\mathbf{X}_{k}|^{2}$ for all $k \in \N$.
    Any limiting branching function for $\{(\mathbf{X}_{k},\delta_{k})\}_{k = 1}^{\infty}$ is in $\mathcal{L}_{\alpha,\beta} \subset \mathcal{L}^{\mathrm{good}}$ (the last inclusion being \cref{thm3}), so \cref{prop2} produces a contradiction.
\end{proofref}

\section{Construction of the lower bound}\label{s:lower}
\subsection{Discretised results}\label{s2}
In this section we establish $\delta$-discretised constructions to help construct the compact set $K$ in the lower bound for \cref{main}.
The final construction can be found in \cref{s:lower-bound}.
A key component is the ``basic building block'' (BBB) in \cref{p:bbb}.
Informally speaking, the BBB is a compact set $K \subset [0,1]^{2}$ with $\Hd K = \tfrac{3}{2}$ and the following property: for every slope $\sigma \in [0,1)$, there exists a $\tfrac{1}{2}$-dimensional family of lines with slope $\sigma$ which ``pass near $K$ but do not touch $K$''.
For more precision on this mysterious property, see \cref{p:bbb}~\cref{2}-\cref{3}.

The BBB is closely related to another, well-known, construction in projection theory: there exists a compact set $K \subset \R^{2}$ such that $\Hd K = \tfrac{3}{2}$ and 
\begin{equation}\label{form23}
    \Hd \{\sigma \in [0,1) : \mathcal{H}^{1}(\pi_{\sigma}(K)) = 0\} = \tfrac{1}{2}.
\end{equation}
Here $\pi_{\sigma}(x,y) = \sigma x + y$.
As far as we are aware, a set $K$ like this was first constructed by Peltomäki \cite{local:peltomaki-thesis} in the late 80s.
We will however require a more recent (and $\delta$-discretised) incarnation presented in \cite[Appendix A.1]{zbmath:7846313}.
In fact, $K$ is nothing but a ``$\tfrac{3}{2}$-dimensional grid'' and the slopes $\sigma \in [0,1)$ appearing in \cref{form23} come from a ``$\tfrac{1}{2}$-dimensional arithmetic progression''.
It is no coincidence that the pair $(\tfrac{3}{2},\tfrac{1}{2})$ appears both in the description of the BBB, and in \cref{form23}: the BBB is obtained by suitably applying point-line duality to the set underlying \cref{form23}.
The details are so lengthy that they are postponed to \cref{appA}.
However, a depiction of the BBB is shown in \cref{fig1}.

The main purpose of this section is to apply the BBB to build something that can be directly used for the lower bound in \cref{main}.
The main result of the section is \cref{thm1}.
Roughly speaking, \cref{thm1} contains the construction of (nearly) $\tfrac{3}{2}$-dimensional families $\mathcal{P}_{\theta} \subset \mathcal{D}_{\delta}$, indexed by $\theta \in \mathcal{D}_{\delta}([0,1))$, with the following property: $\mathcal{P}_{\theta}$ cannot ``block the view'' to $\mathcal{P}_{\theta'}$ along lines parallel to $\theta$ (or $\theta'$), except if $\theta,\theta'$ are ``neighbours''.

The families $\mathcal{P}_{\theta}$ in \cref{thm1} are obtained by iterating the BBB many times.
All direct communication with the BBB happens within \cref{p:bbb-iter}, which iterates the BBB many times.
Then the proof of \cref{thm1} iterates \cref{p:bbb-iter} a few times more.
\begin{proposition}[Basic building block]\label{p:bbb}
    For every $\epsilon,\Delta \in 2^{-\N}$, there exists $\delta_{0} = \delta_{0}(\epsilon,\Delta) \in (0,\tfrac{1}{2}\Delta]$ such that the following holds for all $\delta \in 2^{-\N} \cap (0,\delta_{0}]$:

    There exists a family $\mathfrak{P} \subset \mathcal{D}_{\delta}$ with the following property:
    Assume that $t \in [1 + \epsilon,\tfrac{3}{2}]$, and $\mathcal{Q} \subset \mathcal{D}_{\Delta}$.
    Then 
    \begin{equation*}
        \mathcal{H}_{\infty}^{t - \epsilon}(\mathcal{Q} \cap \mathfrak{P}) \geq \mathcal{H}^{t}_{\infty}(\mathcal{Q}).
    \end{equation*}
    Moreover, for each $\sigma \in \delta \Z \cap [0,1)$ there exists a family $\mathcal{R}_{\sigma}$ of $(\delta \times \Delta)$-rectangles whose longer $\Delta$-sides have slope $\sigma$, with the following properties:
    \begin{enumerate}[nl,label=(\arabic*)]
        \item \label{1} Let $\mathcal{T}_{\sigma}$ be the $\delta$-tubes spanned by the rectangles $\mathcal{R}_{\sigma}$.
            Every tube in $\mathcal{T}_{\sigma}$ contains exactly one element of $\mathcal{R}_{\sigma}$, and the tubes in $\mathcal{T}_{\sigma}$ are $10\delta$-separated.
        \item \label{2} $\cup \mathfrak{P} \cap (\cup \{10T : T \in \mathcal{T}_{\sigma}\}) = \emptyset$.
        \item \label{3} If $t \in [1 + \epsilon,\tfrac{3}{2}]$, and $\mathcal{Q} \subset \mathcal{D}_{\Delta}$  then 
            \begin{equation*}
                \mathcal{H}^{t - \epsilon}_{\infty}(\mathcal{Q} \cap \mathcal{D}_{\delta}(\cup \mathcal{R}_{\sigma})) \geq \mathcal{H}^{t}_{\infty}(\mathcal{Q}).
            \end{equation*}
    \end{enumerate}
\end{proposition} 
\begin{remark}
    The family $\mathfrak{P}$ and the families $\mathcal{R}_{\sigma},\mathcal{T}_{\sigma}$ do not depend on $\mathcal{Q}$, $t$.
\end{remark} 

\begin{figure}[t]
    \begin{center}
        \includegraphics[width=0.7\textwidth]{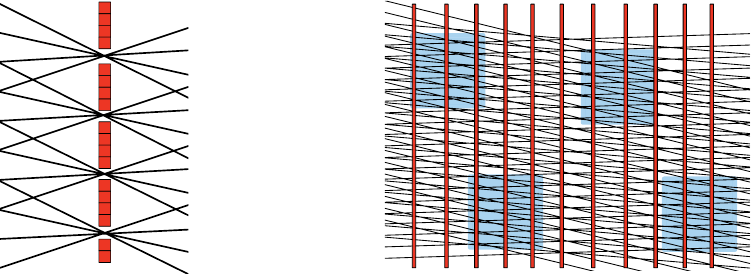}
        \caption{
            The basic building block from \cref{p:bbb} at a microscopic scale (left) and a macroscopic scale (right).
            The picture on the left shows how the tubes $T \in \mathcal{T}_{\sigma}$ avoid $\mathfrak{P}$ at scale $\delta$.
            The picture on the right shows how both $\mathfrak{P}$ and $\mathcal{T}_{\sigma}$ have large intersection with any family of $\mathcal{Q} \in \mathcal{D}_{\Delta}$.
        }
        \label{fig2}
    \end{center}
\end{figure}

From now on, the symbol $\theta$ will denote an interval of slopes, for example $\theta \in \mathcal{D}_{\delta}([0,1))$.
The symbol $\sigma$ will refer to slopes (points in $[0,1)$).
\begin{definition}[$\theta$-disconnected pair]\label{def:disconnectedPair}
    Let $\delta \in 2^{-\N}$ and $\theta \in \mathcal{D}_{\delta}([0,1))$.
    A pair $\mathcal{P}_{1},\mathcal{P}_{2} \subset \mathcal{D}$ is \emph{$\theta$-disconnected} if there exists no line $\ell \subset \R^{2}$ with slope $\sigma(\ell) \in \theta$ such that 
    \begin{equation*}
        (\cup \mathcal{P}_{1}) \cap [\ell]_{\delta/2} \neq \emptyset \quad \text{and} \quad (\cup \mathcal{P}_{2}) \cap [\ell]_{\delta/2} \neq \emptyset.
    \end{equation*}
\end{definition}

We also need the following related notion:
\begin{definition}[$\theta$-diameter]
    Let $\delta \in 2^{-\N}$ and $\theta \in \mathcal{D}_{\delta}([0,1))$.
    The \emph{$\theta$-diameter of $\mathcal{P} \subset \mathcal{D}$} is defined by 
    \begin{equation*}
        \diam_{\theta}(\mathcal{P}) \coloneqq \sup_{\sigma(\ell) \in \theta} \diam(\cup \mathcal{P} \cap [\ell]_{\delta/2}),
    \end{equation*} 
    where the supremum runs over all lines $\ell \subset \R^{2}$ with slope $\sigma(\ell) \in \theta$.
\end{definition} 
\begin{theorem}\label{thm1}
    For each $\epsilon,\Delta \in 2^{-\N}$ there exists $\delta = \delta(\epsilon,\Delta) \in 2^{-\N} \cap (0,\tfrac{1}{2}\Delta]$ such that the following holds:

    For each $\bm{\theta} \in \mathcal{D}_{\Delta}([0,1))$, let $\mathcal{P}_{\bm{\theta}} \subset \mathcal{D}_{\Delta}$ be an arbitrary non-empty family.
    Then, for each $\theta \in \mathcal{D}_{\delta}([0,1))$, there exists a family $\mathcal{P}_{\theta} \subset \mathcal{D}_{\delta}$ with the following properties:
    \begin{enumerate}[nl, A]
        \item\label{A} $\mathcal{P}_{\theta} \subset \mathcal{D}_{\delta}(\mathcal{P}_{\bm{\theta}})$ for $\bm{\theta} \in \mathcal{D}_{\Delta}([0,1))$ and $\theta \in \mathcal{D}_{\delta}(\bm{\theta})$.
        \item \label{B} $\diam_{\theta}(\mathcal{P}_{\theta}) \leq 2\Delta$ for all $\theta \in \mathcal{D}_{\delta}([0,1))$.
        \item\label{C} $\mathcal{H}^{t - \epsilon}_{\infty}(\mathcal{P}_{\theta}) \geq \mathcal{H}^{t}_{\infty}(\mathcal{P}_{\bm{\theta}})$ for all $\bm{\theta} \in \mathcal{D}_{\Delta}([0,1))$, $\theta \in \mathcal{D}_{\delta}(\bm{\theta})$, and $t \in [1 + \epsilon,\tfrac{3}{2}]$.
        \item\label{D} If $\bm{\theta},\bm{\theta}' \in \mathcal{D}_{\Delta}([0,1))$ are distinct, and $\theta \in \mathcal{D}_{\delta}(\bm{\theta})$ and $\theta' \in \mathcal{D}_{\delta}(\bm{\theta}')$, then the pair $\mathcal{P}_{\theta},\mathcal{P}_{\theta'}$ is $\theta$-disconnected and $\theta'$-disconnected.
    \end{enumerate}
\end{theorem} 

\cref{thm1} will be proved by iterating the following auxiliary result:
\begin{lemma}\label{p:bbb-iter}
    For every $\epsilon,\Delta \in 2^{-\N}$ and $M \in \N$ there exists $\delta = \delta(\Delta,\epsilon,M) \in 2^{-\N}\cap(0,\tfrac{1}{2}\Delta]$ such that the following holds:

    Assume that for each $\bm{\theta} \in \mathcal{D}_{\Delta}([0,1))$ we are given a family $\mathcal{P}_{\bm{\theta}} \subset \mathcal{D}_{\Delta}$.
    Let $\mathcal{M}$ be a finite family of subsets of $\mathcal{D}_{\Delta}$ with $|\mathcal{M}| \leq M$.
    \begin{enumerate}[nl]
        \item For each $\theta \in \mathcal{D}_{\delta}([0,1))$ there exists a family $\mathcal{P}_{\theta} \subset \mathcal{D}_{\delta}$ with the following properties:
            \begin{enumerate}[nl, r]
                \item\label{i} If $\bm{\theta} \in \mathcal{D}_{\Delta}([0,1))$, and $\theta \in \mathcal{D}_{\delta}(\bm{\theta})$, then $\mathcal{P}_{\theta} \subset \mathcal{D}_{\delta}(\mathcal{P}_{\bm{\theta}})$.
                    Moreover $\diam_{\theta}(\mathcal{P}_{\theta}) \leq 2\Delta$.
                \item\label{ii} $\mathcal{H}^{t - \epsilon}_{\infty}(\mathcal{P}_{\theta}) \geq \mathcal{H}^{t}_\infty(\mathcal{P}_{\bm{\theta}})$ for all $\theta \in \mathcal{D}_{\delta}(\bm{\theta})$ and $\bm{\theta} \in \mathcal{D}_{\Delta}([0,1))$, and all $t \in [1 + \epsilon,\tfrac{3}{2}]$.
            \end{enumerate}
        \item For each $\mathcal{F} \in \mathcal{M}$ there exists a subset $\mathcal{F}' \subset \mathcal{D}_{\delta}(\mathcal{F})$ such that $\mathcal{H}^{t - \epsilon}_{\infty}(\mathcal{F}') \geq \mathcal{H}^{t}_{\infty}(\mathcal{F})$ for all $t \in [1 + \epsilon,\tfrac{3}{2}]$.
            The family of these subsets $\mathcal{F}'$ is denoted $\mathcal{M}'$.

        \item Every pair $\mathcal{P}_{\theta},\mathcal{F}'$ is $\theta$-disconnected, with $\theta \in \mathcal{D}_{\delta}([0,1))$ and $\mathcal{F}' \in \mathcal{M}'$.
    \end{enumerate}
\end{lemma} 
\begin{proof}
    We start by reducing the proof to the case where $M = 1$.
    Indeed, if $M > 1$, we apply the proposition $M$ times consecutively.
    We give the details for the case $M = 2$, thus $\mathcal{M} = \{\mathcal{F}_{1},\mathcal{F}_{2}\}$; the general case should be clear after this, and only notationally messier.

    Start by applying the case $M = 1$ of the proposition with constants $\Delta,\epsilon/2$ and data 
    \begin{equation*}
        \{\mathcal{P}_{\bm{\theta}}\}_{\bm{\theta} \in \mathcal{D}_{\Delta}([0,1))} \quad \text{and} \quad \{\mathcal{F}_{1}\}.
    \end{equation*}
    This yields:
    \begin{itemize}[nl]
        \item a scale $\delta_{0} \in 2^{-\N} \cap (0,\tfrac{1}{2}\Delta]$, 
        \item families $\mathcal{P}_{\theta} \subset \mathcal{D}_{\delta_{0}}$, $\theta \in \mathcal{D}_{\delta_{0}}([0,1))$, which satisfy \cref{i}-\cref{ii} (with constant $\epsilon/2$), and
        \item a family $\mathcal{F}_{1}' \subset \mathcal{D}_{\delta_{0}}(\mathcal{F}_{1})$, such that $\mathcal{H}^{t - \epsilon}_{\infty}(\mathcal{F}_{1}') \geq \mathcal{H}^{t}_{\infty}(\mathcal{F}_{1})$ for every $t \in [1 + \epsilon,\tfrac{3}{2}]$, and every pair $\mathcal{P}_{\theta},\mathcal{F}_{1}'$ is $\theta$-disconnected, $\theta \in \mathcal{D}_{\delta_{0}}([0,1))$.
    \end{itemize}

    Next, we reapply the case $M = 1$ of the proposition, with constants $\delta_{0},\epsilon/2$ and data 
    \begin{equation*}
        \{\mathcal{P}_{\theta}\}_{\theta \in \mathcal{D}_{\delta_{0}}([0,1))} \quad \text{and} \quad \{\mathcal{F}_{2}\}.
    \end{equation*}
    This gives another scale $\delta \in 2^{-\N} \cap (0,\delta_{0}]$, new families $\mathcal{P}_{\theta'} \subset \mathcal{D}_{\delta}$, $\theta' \in \mathcal{D}_{\delta}([0,1))$, satisfying \cref{i}-\cref{ii} relative to the families $\{\mathcal{P}_{\theta}\}_{\theta \in \mathcal{D}_{\delta_{0}}([0,1))}$, and finally a set $\mathcal{F}_{2}' \in \mathcal{D}_{\delta}(\mathcal{F}_{2})$, such that every pair $\mathcal{P}_{\theta'},\mathcal{F}_{2}'$ is $\theta'$-disconnected, for $\theta' \in \mathcal{D}_{\delta}([0,1))$.

    We claim that the collections $\{\mathcal{P}_{\theta'}\}_{\theta' \in \mathcal{D}_{\delta}([0,1))}$ and $\mathcal{M}' = \{\mathcal{F}_{1}',\mathcal{F}_{2}'\}$ satisfy all the claims of \cref{p:bbb-iter}.
    To be precise, we redefine $\mathcal{F}_{1}' \coloneqq \mathcal{D}_{\delta}(\mathcal{F}_{1}')$ to comply with the requirement $\mathcal{F}_{1}',\mathcal{F}_{2}' \subset \mathcal{D}_{\delta}$.
    This has no effect on the properties claimed about $\mathcal{F}_{1}'$ so far (Hausdorff content, disconnectedness of $\mathcal{P}_{\theta},\mathcal{F}_{1}'$).

    We start with \cref{i}-\cref{ii}.
    Fix $\bm{\theta} \in \mathcal{D}_{\Delta}([0,1))$ and $\theta' \in \mathcal{D}_{\delta}(\bm{\theta})$.
    Let $\theta \in \mathcal{D}_{\delta_{0}}([0,1))$ be the unique dyadic arc with $\theta' \subset \theta \subset \bm{\theta}$.
    Then $\mathcal{P}_{\theta'} \subset \mathcal{P}_{\theta} \subset \mathcal{P}_{\bm{\theta}}$ by construction, and also
    \begin{equation*}
        \mathcal{H}_{\infty}^{t - \epsilon}(\mathcal{P}_{\theta'}) \geq \mathcal{H}_{\infty}^{t - \epsilon/2}(\mathcal{P}_{\theta}) \geq \mathcal{H}_{\infty}^{t}(\mathcal{P}_{\bm{\theta}}), \qquad t \in [1 + \epsilon,\tfrac{3}{2}].
    \end{equation*} 
    This verifies \cref{i}-\cref{ii} (note also that $\diam_{\theta'}(\mathcal{P}_{\theta'}) \leq \delta_{0} \leq 2\Delta$ for all $\theta' \in \mathcal{D}_{\delta}([0,1))$.) The families $\mathcal{F}_{j}'$, $j \in \{1,2\}$, evidently satisfy $\mathcal{H}^{t - \epsilon}_{\infty}(\mathcal{F}_{j}') \geq \mathcal{H}^{t}_{\infty}(\mathcal{F}_{j})$ for each $t \in [1 + \epsilon,\tfrac{3}{2}]$.
    So, we only need to check that every pair $\mathcal{P}_{\theta'},\mathcal{F}_{j}'$ is $\theta'$-disconnected, for $\theta' \in \mathcal{D}_{\delta}([0,1))$ and $j \in \{1,2\}$.
    This is clear for $j = 2$ by the choice of the sets $\mathcal{P}_{\theta'}$ and $\mathcal{F}_{2}'$.
    So, it remains to check that every pair $\mathcal{P}_{\theta}',\mathcal{F}_{1}'$ is $\theta$-disconnected, for $\theta' \in \mathcal{D}_{\delta}([0,1))$.

    For this we use property \cref{i}.
    Fix $\theta' \in \mathcal{D}_{\delta}([0,1))$, and let $\theta \in \mathcal{D}_{\delta_{0}}([0,1))$ be the dyadic parent of $\theta'$.
    Then $\mathcal{P}_{\theta'} \subset \mathcal{D}_{\delta}(\mathcal{P}_{\theta})$ according to \cref{i}, and moreover $\mathcal{P}_{\theta},\mathcal{F}_{1}'$ is $\theta$-disconnected by the first step of the construction.
    This means that if $\ell \subset \R^{2}$ is a line with $\sigma(\ell) \in \theta$, then $[\ell]_{\delta_{0}/2}$ cannot intersect both $\cup \mathcal{P}_{\theta}$ and $\cup \mathcal{F}_{1}'$.
    This certainly implies that if $\ell \subset \R^{2}$ is a line with $\sigma(\ell) \in \theta' \subset \theta$, then $[\ell]_{\delta/2}$ cannot intersect both $\cup \mathcal{P}_{\theta'} \subset \cup \mathcal{P}_{\theta}$ and $\cup \mathcal{F}_{1}'$.
    Therefore $\mathcal{P}_{\theta'},\mathcal{F}_{1}'$ is $\theta'$-disconnected.
    This concludes the proof of the case $M = 2$ (the proofs in the cases $M \geq 3$ are similar and left to the reader).

    It remains to prove the case $M = 1$.
    Write $\mathcal{M} \eqqcolon \{\mathcal{F}\}$ where $\mathcal{F}\subset \mathcal{D}_{\Delta}$.
    Recall that the families $\mathcal{P}_{\bm{\theta}} \subset \mathcal{D}_{\Delta}$ are indexed by the arcs $\bm{\theta} \in \mathcal{D}_{\Delta}([0,1))$.
    Let us enumerate these arcs $\{\bm{\theta}_{1},\ldots,\bm{\theta}_{m}\}$, where $m = \Delta^{-1}$.
    We first consider the pair $(\mathcal{P}_{\bm{\theta}_{1}},\mathcal{F})$.

    We apply \cref{p:bbb} with parameters $\epsilon/m = \Delta \epsilon$ and $\Delta$.
    This gives a scale $\delta_{1} \in 2^{-\N} \cap (0,\tfrac{1}{2}\Delta]$, and a family $\mathfrak{P}_{1} \subset \mathcal{D}_{\delta_{1}}$ satisfying $\mathcal{H}_{\infty}^{t - \epsilon/m}(\mathcal{F} \cap \mathfrak{P}_{1}) \geq \mathcal{H}^{t}_{\infty}(\mathcal{F})$ for all $t \in [1+\epsilon/m,\tfrac{3}{2}]$.
    Set $\mathcal{F}_{1} \coloneqq \mathcal{F} \cap \mathfrak{P}_{1}$, thus 
    \begin{equation}\label{form131}
        \mathcal{F}_{1} \subset \mathcal{D}_{\delta_{1}}(\mathcal{F}) \quad \text{and} \quad \mathcal{H}^{t - \epsilon/m}_{\infty}(\mathcal{F}_{1}) \geq \mathcal{H}^{t}_{\infty}(\mathcal{F}) \text{ for } t \in [1 + \epsilon,\tfrac{3}{2}].
    \end{equation}
    Moreover, for each $\theta \in \mathcal{D}_{\delta_{1}}([0,1))$ (in particular $\theta \in \mathcal{D}_{\delta_{1}}(\bm{\theta}_{1})$), \cref{p:bbb} gives a family $\mathcal{T}_{\theta}$ of $\delta_{1}$-tubes parallel to (the slope determined by the left endpoint of) $\theta$, and associated $(\delta_1 \times \Delta)$-rectangles $\mathcal{R}_{\theta}$, such that
    \begin{enumerate}[nl,a]
        \item\label{a} $\cup \mathcal{F}_{1} \cap (\cup \{10T : T \in \mathcal{T}_{\theta}\}) = \emptyset$,
        \item\label{b} $\mathcal{H}^{t - \epsilon}_{\infty}(\mathcal{P}_{\bm{\theta}_{1}} \cap \mathcal{D}_{\delta_{1}}(\cup \mathcal{R}_{\theta})) \geq \mathcal{H}^{t}_{\infty}(\mathcal{P}_{\bm{\theta}_{1}})$, for $\theta \in \mathcal{D}_{\delta_{1}}(\bm{\theta}_{1})$ and $t \in [1 + \epsilon,\tfrac{3}{2}]$.
    \end{enumerate}
    We now define 
    \begin{equation*}
        \mathcal{P}_{\theta} \coloneqq \mathcal{P}_{\bm{\theta}_{1}} \cap \mathcal{D}_{\delta_{1}}(\cup \mathcal{R}_{\theta}), \qquad \theta \in \mathcal{D}_{\delta_{1}}(\bm{\theta}_{1}).
    \end{equation*}
    Thus $\mathcal{H}^{t - \epsilon}_{\infty}(\mathcal{P}_{\theta}) \geq \mathcal{H}^{t}_{\infty}(\mathcal{P}_{\bm{\theta}_{1}})$ for all $t \in [1 + \epsilon,\tfrac{3}{2}]$, and $\mathcal{P}_{\theta} \subset \mathcal{D}_{\delta_{1}}(\mathcal{P}_{\bm{\theta}_{1}})$ for $\theta \in \mathcal{D}_{\delta_{1}}(\bm{\theta}_{1})$.
    This means that the sets $\mathcal{P}_{\theta}$ (for $\theta \in \mathcal{D}_{\delta_{1}}(\bm{\theta}_{1})$) satisfy the requirements \cref{i}-\cref{ii} in the statement of \cref{p:bbb-iter}.
    The claim $\diam_{\theta}(\mathcal{P}_{\theta}) \leq 2\Delta$ follows from the inclusion $\mathcal{P}_{\theta} \subset \mathcal{D}_{\delta_{1}}(\cup \mathcal{R}_{\theta})$ (and $\delta_{1} \leq \Delta/2$), and recalling the key properties of the rectangles and tubes $\mathcal{R}_{\theta},\mathcal{T}_{\theta}$ from \cref{p:bbb}: each tube $T \in \mathcal{T}_{\theta}$ only contains one rectangle from $\mathcal{R}_{\theta}$, the tubes $T \in \mathcal{T}_{\theta}$ have slope $\theta$, and the tubes $T \in \mathcal{T}_{\theta}$ are $10\delta$-separated.

    A small caveat is that ``$\delta_{1}$'' is not yet the final scale $\delta \in 2^{-\N} \cap (0,\delta_{0}]$ announced in the proposition, and we claimed in the statement to choose $\mathcal{P}_{\theta} \subset \mathcal{D}_{\delta}$.
    Once the final scale $\delta$ has been located (after altogether $m$ steps of the argument), we will re-define the sets $\mathcal{P}_{\theta}$ as follows: for each $\theta \in \mathcal{D}_{\delta_{1}}(\bm{\theta}_{1})$ and $\theta' \in \mathcal{D}_{\delta}(\theta)$, we set $\mathcal{P}_{\theta'} \coloneqq \mathcal{D}_{\delta}(\mathcal{P}_{\theta})$.
    Then $\mathcal{H}^{t - \epsilon}_{\infty}(\mathcal{P}_{\theta'}) \geq \mathcal{H}^{t}_{\infty}(\mathcal{P}_{\bm{\theta}_{1}})$ for all $\theta' \in \mathcal{D}_{\delta}(\bm{\theta}_{1})$.
    We will next check that the pairs $\mathcal{P}_{\theta},\mathcal{F}_{1}$ are $\theta$-disconnected.
    This will clearly imply that the pairs $\mathcal{P}_{\theta'},\mathcal{F}_{1}$ are $\theta'$-disconnected.

    Let us check that $\mathcal{P}_{\theta},\mathcal{F}_{1}$ is $\theta$-disconnected for all $\theta \in \mathcal{D}_{\delta_{1}}(\bm{\theta}_{1})$.
    Indeed, fix $\theta \in \mathcal{D}_{\delta_{1}}(\bm{\theta}_{1})$, and let $\ell \subset \R^{2}$ be a line with $\sigma(\ell) \in \theta$.
    Now, if $[\ell]_{\delta_{1}/2}$ intersects $\cup \mathcal{P}_{\theta} \subset \cup \mathcal{D}_{\delta_{1}}(\cup \mathcal{T}_{\theta})$, then $[0,1]^2\cap [\ell]_{\delta_{1}/2} \subset \cup \{10T : T \in \mathcal{T}_{\theta}\}$.
    Therefore $[\ell]_{\delta_{1}/2}$ does not intersect $\cup \mathcal{F}_{1}$ by \cref{a}.

    Next we deal with the arc $\bm{\theta}_{k}$ with $2 \leq k \leq m$.
    Assume that we have already constructed a decreasing sequence of scales $\delta_{1},\ldots,\delta_{k - 1} \in 2^{-\N} \cap (0,\Delta]$, and a nested sequence of families $\mathcal{F}_{i} \subset \mathcal{D}_{\delta_{i}}(\mathcal{F})$, $1 \leq i \leq k - 1$, satisfying 
    \begin{equation}\label{form14}
        \mathcal{H}_{\infty}^{t - i\epsilon/m}(\mathcal{F}_{i}) \geq \mathcal{H}_{\infty}^{t}(\mathcal{F}), \qquad 1 \leq i \leq k - 1, \, t \in [1 + \epsilon,\tfrac{3}{2}].
    \end{equation}
    (By ``nested'' we mean that the squares in $\mathcal{F}_{i + 1}$ are contained in the union of the squares in $\mathcal{F}_{i}$.) The case $k = 2$ is true by \cref{form131}.
    We also assume that for each $1 \leq i \leq k - 1$, and for each $\theta \in \mathcal{D}_{\delta_{i}}(\bm{\theta}_{i})$, we have constructed a family $\mathcal{P}_{\theta} \subset \mathcal{D}_{\delta_{i}}(\mathcal{P}_{\bm{\theta}_{i}})$ such that $\mathcal{P}_{\theta},\mathcal{F}_{k - 1}$ is $\theta$-disconnected.
    We verified this just above in the case $k = 2$.

    We next construct the scale $\delta_{k} \in 2^{-\N} \cap (0,\delta_{k - 1}]$, the family $\mathcal{F}_{k} \subset \mathcal{D}_{\delta_{k}}(\mathcal{F}_{k - 1}) \subset \mathcal{D}_{\delta_{k}}(\mathcal{F})$, and the sets $\mathcal{P}_{\theta} \subset \mathcal{D}_{\delta_{k}}(\mathcal{P}_{\bm{\theta}_{k}})$ for $\theta \in \mathcal{D}_{\delta_{k}}(\bm{\theta}_{k})$.
    To do this, we apply \cref{p:bbb} with parameters $\epsilon/m = \Delta \epsilon$ and $\delta_{k - 1}$.
    This gives a scale $\delta_{k} \in 2^{-\N} \cap (0,\tfrac{1}{2}\delta_{k - 1}]$ and a family $\mathfrak{P}_{k} \subset \mathcal{D}_{\delta_{k}}$ with the property 
    \begin{equation}\label{form15}
        \mathcal{H}^{t - k \epsilon/m}_\infty(\mathcal{F}_{k - 1} \cap \mathfrak{P}_{k}) \geq \mathcal{H}^{t - (k - 1)\epsilon/m}_\infty(\mathcal{F}_{k - 1}) \stackrel{\cref{form14}}{\geq} \mathcal{H}^{t}_{\infty}(\mathcal{F}), \quad t \in [1 + \epsilon,\tfrac{3}{2}].
    \end{equation}
    Define $\mathcal{F}_{k} \coloneqq \mathcal{F}_{k - 1} \cap \mathfrak{P}_{k}$.
    Moreover, for each $\theta \in \mathcal{D}_{\delta_{k}}(\bm{\theta}_{k})$, \cref{p:bbb} gives a family $\mathcal{T}_{\theta}$ of $\delta_{k}$-tubes parallel to $\theta$, and associated $(\delta_{k} \times \delta_{k - 1})$-rectangles $\mathcal{R}_{\theta}$ such that
    \begin{itemize}
        \item[($a_{k}$)] $\cup \mathcal{F}_{k} \cap (\cup \{10T : T \in \mathcal{T}_{\theta}\}) = \emptyset$, and
        \item[($b_{k}$)] $\mathcal{H}^{t - \epsilon}_{\infty}(\mathcal{P}_{\bm{\theta}_{k}} \cap \mathcal{D}_{\delta_{k}}(\cup \mathcal{R}_{\theta})) \geq \mathcal{H}^{t}_{\infty}(\mathcal{P}_{\bm{\theta}_{k}})$, for $\theta \in \mathcal{D}_{\delta_{k}}(\bm{\theta}_{k})$ and $t \in [1 + \epsilon,\tfrac{3}{2}]$.
    \end{itemize}
    We now define 
    \begin{equation*}
        \mathcal{P}_{\theta} \coloneqq \mathcal{P}_{\bm{\theta}_{k}} \cap \mathcal{D}_{\delta_{k}}(\cup \mathcal{R}_{\theta}), \qquad \theta \in \mathcal{D}_{\delta_{k}}(\bm{\theta}_{k}).
    \end{equation*}
    Thus $\mathcal{P}_{\theta} \subset \mathcal{D}_{\delta_{k}}(\mathcal{P}_{\bm{\theta}_{k}})$ and $\mathcal{H}_\infty^{t - \epsilon}(\mathcal{P}_{\theta}) \geq \mathcal{H}^{t}_{\infty}(\mathcal{P}_{\bm{\theta}_{k}})$ for all $\theta \in \mathcal{D}_{\delta_{k}}(\bm{\theta}_{k})$ and $t \in [1 + \epsilon,\tfrac{3}{2}]$.
    We also note that $\diam_{\theta}(\mathcal{P}_{\theta}) \leq 2\delta_{k - 1} \leq 2\Delta$ for all $\theta \in \mathcal{D}_{\delta_{k}}(\bm{\theta}_{k})$.

    To complete the induction, it remains to check that every pair $\mathcal{P}_{\theta},\mathcal{F}_{k}$ is $\theta$-disconnected for $\theta \in \mathcal{D}_{\delta_{i}}(\bm{\theta}_{i})$, $1 \leq i \leq k$.
    For $1 \leq i \leq k - 1$, this follows immediately from the $\theta$-disconnectedness of $\mathcal{P}_{\theta},\mathcal{F}_{k - 1}$, and $\cup \mathcal{F}_{k} \subset \cup \mathcal{F}_{k - 1}$.
    For $i = k$, the argument is the same as the argument we already recorded above for the $\theta$-disconnectedness of the pairs $\mathcal{P}_{\theta},\mathcal{F}_{1}$, with $\theta \in \mathcal{D}_{\delta_{1}}(\bm{\theta}_{1})$, now using ($a_{k}$) and $\mathcal{P}_{\theta} \subset \mathcal{D}_{\delta_{k}}(\cup \mathcal{T}_{\theta})$.
    We do not repeat the details.

    We complete the proof by setting $\delta \coloneqq \delta_{m}$, and $\mathcal{F}' \coloneqq \mathcal{F}_{m}$.
    Then $\mathcal{H}^{t - \epsilon}_{\infty}(\mathcal{F}') \geq \mathcal{H}^{t}_{\infty}(\mathcal{F})$ for all $t \in [1 + \epsilon,\tfrac{3}{2}]$ by \cref{form15} (with $k = m$).
    As already mentioned when dealing with the case $k = 1$, we finally have to redefine all the sets $\mathcal{P}_{\theta} \subset \mathcal{D}_{\delta_{k}}(\bm{\theta}_{k})$, $1 \leq k \leq m$, in a trivial way, to comply with the requirement that they are all subsets of $\mathcal{D}_{\delta}$, and indexed by $\mathcal{D}_{\delta}([0,1))$.

    For $1 \leq k \leq m$, $\theta \in \mathcal{D}_{\delta_{k}}(\bm{\theta}_{k})$, and $\theta' \in \mathcal{D}_{\delta}(\theta)$, we set $\mathcal{P}_{\theta'} \coloneqq \mathcal{D}_{\delta}(\mathcal{P}_{\theta})$.
    With this definition $\mathcal{P}_{\theta'} \subset \mathcal{D}_{\delta}(\mathcal{P}_{\bm{\theta}})$ and $\mathcal{H}^{t - \epsilon}_{\infty}(\mathcal{P}_{\theta'}) \geq \mathcal{H}^{t}_{\infty}(\mathcal{P}_{\bm{\theta}})$ for all $\bm{\theta} \in \mathcal{D}_{\Delta}([0,1))$, all $\theta' \in \mathcal{D}_{\delta}([0,1))$, and all $t \in [1 + \epsilon,\tfrac{3}{2}]$.
    The $\theta'$-disconnectedness of the pairs $\mathcal{P}_{\theta'},\mathcal{F}'$ follows from the $\theta$-disconnectedness of the pairs $\mathcal{P}_{\theta},\mathcal{F}'$.
    The bound $\diam_{\theta'}(\mathcal{P}_{\theta'}) \leq \Delta$ is inherited from $\diam_{\theta}(\mathcal{P}_{\theta}) \leq \Delta$.
\end{proof} 
 
 We are then equipped to prove \cref{thm1}.
\begin{proofref}{thm1}
    Recall that $\Delta \in 2^{-\N}$.
    We may assume that $\Delta \in (0,\tfrac{1}{2}]$.
    Namely, if $\Delta = 1$ then we are given a single non-empty family $\mathcal{P}_{[0,1)} \subset \mathcal{D}_{1}$, so $\mathcal{P}_{[0,1)} = \{[0,1)^{2}\}$.
    We then define $\delta \coloneqq 1/2$ and $\mathcal{P}_{[0,1/2)} \coloneqq \mathcal{D}_{1/2} \eqqcolon \mathcal{P}_{[1/2,1)}$.
    The reader may check that all properties \cref{A}-\cref{D} satisfies (\cref{D} is vacuous, because $\mathcal{D}_{1}([0,1))$ contains a single interval).

    We then assume that $\Delta \in (0,\tfrac{1}{2}]$.
    Let us enumerate 
    \begin{equation*}
        \mathcal{D}_{\Delta}([0,1)) \eqqcolon \{\bm{\theta}_{1},\ldots,\bm{\theta}_{m}\}
    \end{equation*}
    with $m = \Delta^{-1} \geq 2$.
    The proof consists of $m$ applications of \cref{p:bbb-iter}.
    At step $k \in \{1,\ldots,m\}$, ``the sets associated to $\bm{\theta}_{k}$ are disconnected from the sets associated to all the other arcs $\bm{\theta}_{j}$ with $j \neq k$.''
    This isn't precise sense at the moment, since the ``sets associated to $\bm{\theta}_{j}$'' are not properly defined.

    Write $\delta_{0} \coloneqq \Delta$.
    Fix $0 \leq k < m$, and assume that the following objects have been constructed.
    First, a decreasing sequence of scales $\delta_{0},\delta_{1},\ldots,\delta_{k} \in 2^{-\N} \cap (0,\Delta]$.

    Second, for each $0 \leq j \leq k$ and $\theta \in \mathcal{D}_{\delta_{j}}([0,1))$, a family $\mathcal{P}_{\theta} \subset \mathcal{D}_{\delta_{j}}$ as follows:
    \begin{enumerate}[nl, label=(I\arabic*)]
        \item\label{I1} If $0 \leq i \leq j$, $\theta_{i} \in \mathcal{D}_{\delta_{i}}([0,1))$ and $\theta_{j} \in \mathcal{D}_{\delta_{j}}(\theta_{i})$, then $\mathcal{P}_{\theta_{j}} \subset \mathcal{D}_{\delta_{j}}(\mathcal{P}_{\theta_{i}})$.
        \item\label{I2} If $\theta \in \mathcal{D}_{\delta_{k}}(\bm{\theta})$ for some $\bm{\theta} \in \mathcal{D}_{\Delta}([0,1))$, then
            \begin{equation*}
                \mathcal{H}^{t - k\epsilon/m}_{\infty}(\mathcal{P}_{\theta}) \geq \mathcal{H}^{t}_{\infty}(\mathcal{P}_{\bm{\theta}}), \qquad t \in [1 + \epsilon,\tfrac{3}{2}].
            \end{equation*}
        \item\label{I3} If $1 \leq i \leq m$ and $j \in \{1,\ldots,k\} \, \setminus \, \{i\}$, then the following holds:
            Let $\theta_{i} \in \mathcal{D}_{\delta_{k}}(\bm{\theta}_{i})$ and $\theta_{j} \in \mathcal{D}_{\delta_{k}}(\bm{\theta}_{j})$.
            Then the pair $\mathcal{P}_{\theta_{i}},\mathcal{P}_{\theta_{j}}$ is $\theta_{i}$-disconnected.
        \item\label{I4} If $1 \leq k \leq m$ and $j \in \{1,\ldots,m\} \, \setminus \, \{k\}$, then $\diam_{\theta}(\mathcal{P}_{\theta}) \leq \Delta$ for all $\theta \in \mathcal{D}_{\delta_{k}}(\bm{\theta}_{j})$.
    \end{enumerate}
    Note that conditions \cref{I1}-\cref{I2} with $k = 0$ are trivially satisfied by the families $\mathcal{P}_{\bm{\theta}_{1}},\ldots,\mathcal{P}_{\bm{\theta}_{m}}$ provided to us by the statement of the proposition, while \cref{I3}-\cref{I4} are vacuous for $k = 0$.

    We will next construct the scale $\delta_{k + 1}$ and the families $\mathcal{P}_{\theta} \subset \mathcal{D}_{\delta_{k + 1}}$, $\theta \in \mathcal{D}_{\delta_{k + 1}}([0,1))$, by applying \cref{p:bbb-iter} with parameters $\delta_{k} \in 2^{-\N} \cap (0,\Delta]$, $\epsilon/m$, and $M \in \N$ to be determined in a moment.
    The families $\mathcal{P}_{\theta} \subset \mathcal{D}_{\delta_{k}}$, $\theta \in \mathcal{D}_{\delta_{k}}([0,1))$, to which \cref{p:bbb-iter} is applied are the ones provided by the inductive hypothesis.
    The family $\mathcal{M}$ is defined by
    \begin{equation*}
        \mathcal{M} \coloneqq \{\mathcal{P}_{\theta_{k}} : \theta_{k} \in \mathcal{D}_{\delta_{k}}(\bm{\theta}_{k + 1})\}.
    \end{equation*}
    Set $M \coloneqq |\mathcal{M}|$.
    The results of the application of \cref{p:bbb-iter} are the following.
    We are provided with a scale $\delta_{k + 1} = \delta_{k + 1}(\delta_{k},\epsilon/m,M) \in (0,\tfrac{1}{2}\delta_{k}]$.
    For each $\theta \in \mathcal{D}_{\delta_{k + 1}}([0,1))$ we are provided with a family $\overline{\mathcal{P}}_{\theta} \subset \mathcal{D}_{\delta_{k + 1}}$ with the following properties (we use the overline symbol to signify that the sets $\overline{\mathcal{P}}_{\theta}$ are not the ``final'' sets of generation $k + 1$):
    \begin{enumerate}[nl,a]
        \item If $\theta_{k} \in \mathcal{D}_{\delta_{k}}([0,1))$ and $\theta \in \mathcal{D}_{\delta_{k + 1}}(\theta_{k})$, then $\overline{\mathcal{P}}_{\theta} \subset \mathcal{D}_{\delta_{k + 1}}(\mathcal{P}_{\theta_{k}})$.
        \item $\diam_{\theta}(\overline{\mathcal{P}}_{\theta}) \leq \delta_{k} \leq \Delta$ for all $\theta \in \mathcal{D}_{\delta_{k + 1}}([0,1))$.
        \item If $\theta_{k} \in \mathcal{D}_{\delta_{k}}([0,1))$ and $\theta \in \mathcal{D}_{\delta_{k + 1}}(\theta_{k})$, then
            \begin{equation*} \mathcal{H}^{t - (k + 1)\epsilon/m}_{\infty}(\overline{\mathcal{P}}_{\theta}) \geq \mathcal{H}^{t - k\epsilon/m}_\infty(\mathcal{P}_{\theta_{k}}), \qquad t \in [1 + \epsilon,\tfrac{3}{2}].
            \end{equation*}
    \end{enumerate}
    Moreover, for each $\mathcal{P}_{\theta_{k}} \in \mathcal{M}$, $\theta_{k} \in \mathcal{D}_{\delta_{k}}(\bm{\theta}_{k + 1})$, we are provided with a family $\mathcal{P}_{\theta_{k}}' \subset \mathcal{D}_{\delta_{k + 1}}(\mathcal{P}_{\theta_{k}})$ such that 
    \begin{equation*}
        \mathcal{H}^{t - (k + 1)\epsilon/m}_{\infty}(\mathcal{P}_{\theta_{k}}') \geq \mathcal{H}^{t - k\epsilon/m}_{\infty}(\mathcal{P}_{\theta_{k}}), \qquad t \in [1 + \epsilon,\tfrac{3}{2}],
    \end{equation*}
    and such that 
    \begin{equation}\label{form18}
        \text{every pair $\overline{\mathcal{P}}_{\theta},\mathcal{P}_{\theta_{k}}'$ is $\theta$-disconnected, for $\theta \in \mathcal{D}_{\delta_{k + 1}}([0,1))$ and $\theta_{k} \in \mathcal{D}_{\delta_{k}}(\bm{\theta}_{k + 1})$.}
    \end{equation}

    We are now prepared to define the sets $\mathcal{P}_{\theta}$, $\theta \in \mathcal{D}_{\delta_{k + 1}}([0,1))$, of generation $k + 1$.
    Set
    \begin{equation}\label{form17}
        \mathcal{P}_{\theta} \coloneqq
        \begin{cases}
            \overline{\mathcal{P}}_{\theta}, & \theta \in \mathcal{D}_{\delta_{k + 1}}([0,1)) \, \setminus \, \mathcal{D}_{\delta_{k + 1}}(\bm{\theta}_{k + 1}), \\
            \mathcal{P}_{\theta_{k}}', & \theta \in \mathcal{D}_{\delta_{k + 1}}(\theta_{k}), \, \theta_{k} \in \mathcal{D}_{\delta_{k}}(\bm{\theta}_{k + 1}).
        \end{cases}
    \end{equation} 

    We then verify that these sets $\mathcal{P}_{\theta}$ satisfy \cref{I1}-\cref{I4}.
    To verify \cref{I1}, fix $0 \leq i \leq j \leq k + 1$.
    If $j \leq k$, the property \cref{I1} follows from induction, and the case $i = k + 1 = j$ is trivial, so we only need to consider pairs $(i,k + 1)$ with $i < k$.
    If $\theta_{i} \in \mathcal{D}_{\delta_{i}}([0,1))$ and $\theta \in \mathcal{D}_{\delta_{k + 1}}(\theta_{i})$, we need to prove that $\mathcal{P}_{\theta} \subset \mathcal{D}_{\delta_{k + 1}}(\mathcal{P}_{\theta_{i}})$.
    In both cases of the definition of $\mathcal{P}_{\theta}$, the following is true: if $\theta_{k} \in \mathcal{D}_{\delta_{k}}([0,1))$ is the dyadic parent of $\theta$, then $\mathcal{P}_{\theta} \subset \mathcal{D}_{\delta_{k + 1}}(\mathcal{P}_{\theta_{k}})$.
    Now the claim follows by applying the inductive hypothesis to the pair $(i,k)$.

    Next, we verify property \cref{I2}.
    Fix $\theta \in \mathcal{D}_{\delta_{k + 1}}(\bm{\theta})$ with $\bm{\theta} \in \mathcal{D}_{\Delta}([0,1))$.
    Let $\theta_{k} \in \mathcal{D}_{\delta_{k}}(\bm{\theta})$ be the dyadic parent of $\theta$.
    In both cases of the definition of $\mathcal{P}_{\theta}$, it holds 
    \begin{equation*}
        \mathcal{H}^{t - (k + 1)\epsilon/m}_{\infty}(\mathcal{P}_{\theta}) \geq \mathcal{H}_{\infty}^{t - k\epsilon/m}(\mathcal{P}_{\theta_{k}}) \geq \mathcal{H}^{t}_{\infty}(\mathcal{P}_{\bm{\theta}}), \qquad t \in [1 + \epsilon,\tfrac{3}{2}],
    \end{equation*}
    where the second inequality uses \cref{I2} inductively.
    This verifies \cref{I2}.

    We move on to property \cref{I3}.
    Fix $1 \leq i \leq m$ and $j \in \{1,\ldots,k + 1\} \, \setminus \, \{i\}$, and let $\theta_{i} \in \mathcal{D}_{\delta_{k + 1}}(\bm{\theta}_{i})$ and $\theta_{j} \in \mathcal{D}_{\delta_{k + 1}}(\bm{\theta}_{j})$.
    The claim is that $\mathcal{P}_{\theta_{i}},\mathcal{P}_{\theta_{j}}$ is $\theta_{i}$-disconnected.

    Consider first the case where $j \leq k$.
    Let $\bar{\theta}_{i} \in \mathcal{D}_{\delta_{k}}(\bm{\theta}_{i})$ and $\bar{\theta}_{j} \in \mathcal{D}_{\delta_{k}}(\bm{\theta}_{j})$ be the dyadic parents of $\theta_{i},\theta_{j}$.
    Since $j \in \{1,\ldots,k\} \, \setminus \, \{i\}$, our inductive hypothesis \cref{I3} guarantees that the pair $\mathcal{P}_{\bar{\theta}_{i}},\mathcal{P}_{\bar{\theta}_{j}}$ is $\bar{\theta}_{i}$-disconnected.
    In particular, if $\ell \subset \R^{2}$ is a line with $\sigma(\ell) \in \theta_{i} \subset \bar{\theta}_{i}$, then $[\ell]_{\delta_{k}/2}$ does not intersect both $\cup \mathcal{P}_{\bar{\theta}_{i}},\cup \mathcal{P}_{\bar{\theta}_{j}}$.
    Since $\cup \mathcal{P}_{\theta_{i}} \subset \cup \mathcal{P}_{\bar{\theta}_{i}}$ and $\cup \mathcal{P}_{\theta_{j}} \subset \cup \mathcal{P}_{\bar{\theta}_{j}}$, we infer that $[\ell]_{\delta_{k}/2}$ also does not intersect both $\cup \mathcal{P}_{\theta_{i}},\cup \mathcal{P}_{\theta_{j}}$.
    Thus $\mathcal{P}_{\theta_{i}},\mathcal{P}_{\theta_{j}}$ is $\theta_{i}$-disconnected.

    Consider finally the case $j = k + 1$, thus $i \in \{1,\ldots,m\} \, \setminus \, \{k + 1\}$.
    By the definition \cref{form17} of the sets $\mathcal{P}_{\theta}$,
    \begin{equation*}
        \mathcal{P}_{\theta_{i}} = \overline{\mathcal{P}}_{\theta_{i}} \quad \text{and} \quad \mathcal{P}_{\theta_{j}} = \mathcal{P}_{\theta_{k}}' \subset \mathcal{D}_{\delta_{k + 1}}(\mathcal{P}_{\theta_{k}}),
    \end{equation*} 
    where $\theta_{k} \in \mathcal{D}_{\delta_{k}}(\bm{\theta}_{k + 1})$ is the dyadic parent of $\theta_{j}$.
    But now the pair $\overline{\mathcal{P}}_{\theta_{i}},\mathcal{P}_{\theta_{k}}'$ is $\theta_{i}$-disconnected by construction, see \cref{form18}.
    Thus $\mathcal{P}_{\theta_{i}},\mathcal{P}_{\theta_{j}}$ is $\theta_{i}$-disconnected.

    Finally, we verify property \cref{I4}.
    Fix $j \in \{1,\ldots,m\} \, \setminus \, \{k + 1\}$ and $\theta \in \mathcal{D}_{\delta_{k + 1}}(\bm{\theta}_{j})$.
    In this case $\mathcal{P}_{\theta}$ coincides with $\overline{\mathcal{P}}_{\theta}$ by \cref{form17}, and property (b) yields $\diam_{\theta}(\mathcal{P}_{\theta}) \leq \Delta$.

    We have now verified all the properties \cref{I1}-\cref{I4} for the scale $\delta_{k + 1}$ and the families $\mathcal{P}_{\theta} \subset \mathcal{D}_{\delta_{k + 1}}$ with $\theta \in \mathcal{D}_{\delta_{k + 1}}([0,1))$.
    We complete the proof of \cref{thm1} as follows.
    Let $\delta \coloneqq \delta_{m} \in 2^{-\N} \cap (0,\Delta]$, and let $\mathcal{P}_{\theta} \subset \mathcal{D}_{\delta}$, $\theta \in \mathcal{D}_{\delta}([0,1))$, be the scale, and the families, obtained at the (final) step $m$ of the construction.
    We claim that these objects satisfy the claims \cref{A}-\cref{D} of \cref{thm1}.
    For \cref{A}, if $\bm{\theta} \in \mathcal{D}_{\Delta}([0,1))$ and $\theta \in \mathcal{D}_{\delta}(\bm{\theta})$, the inclusion $\mathcal{P}_{\theta} \subset \mathcal{D}_{\delta}(\mathcal{P}_{\bm{\theta}})$ follows from the nestedness property \cref{I1} applied with $i = 0$ and $j = m$.

    For \cref{B}, if $\theta \in \mathcal{D}_{\delta}([0,1)) \, \setminus \, \mathcal{D}_{\delta}(\bm{\theta}_{m})$, the diameter bound $\diam_{\theta}(\mathcal{P}_{\theta}) \leq \Delta$ follows directly from property \cref{I4}.
    Consider then the case $\theta \in \mathcal{D}_{\delta}(\bm{\theta}_{m})$.
    Let $\bar{\theta} \in \mathcal{D}_{\delta_{m - 1}}(\bm{\theta}_{m})$ be the dyadic parent of $\theta$.
    Then $m \in \{1,\ldots,m\} \, \setminus \, \{m - 1\}$, so \cref{I4} applied with $(j,k) = (m,m - 1)$ shows that $\diam_{\bar{\theta}}(\mathcal{P}_{\bar{\theta}}) \leq \Delta$.
    Now $\diam_{\theta}(\mathcal{P}_{\theta}) \leq \diam_{\theta}(\mathcal{P}_{\bar{\theta}}) \leq \Delta$ by $\theta \subset \bar{\theta}$ and $\mathcal{P}_{\theta} \subset \mathcal{P}_{\bar{\theta}}$.

    For \cref{C}, the lower bound $\mathcal{H}^{t - \epsilon}_{\infty}(\mathcal{P}_{\theta}) \geq \mathcal{H}^{t}_{\infty}(\mathcal{P}_{\bm{\theta}})$ for $\bm{\theta} \in \mathcal{D}_{\Delta}([0,1))$, $\theta \in \mathcal{D}_{\delta}(\bm{\theta})$, and $t \in [1 + \epsilon,\tfrac{3}{2}]$ follows from \cref{I2} applied with $k = m$.

    For \cref{D}, we need to check that if $\bm{\theta}_{i},\bm{\theta}_{j} \in \mathcal{D}_{\Delta}([0,1))$ with $i,j \in \{1,\ldots,m\}$, $i \neq j$, and $\theta_{i} \in \mathcal{D}_{\delta}(\bm{\theta}_{i})$ and $\theta_{j} \in \mathcal{D}_{\delta}(\bm{\theta}_{j})$, then $\mathcal{P}_{\theta_{i}},\mathcal{P}_{\theta_{j}}$ is $\theta_{i}$-disconnected and $\theta_{j}$-disconnected.
    This follows directly from \cref{I3} applied with $k = m$, and to both pairs $(i,j)$ and $(j,i)$.
\end{proofref}

\subsection{Construction of a set with large visible parts in all directions}\label{s:lower-bound}
We now construct the compact set $K \subset \R^{2}$ for the lower bound in \cref{main}.
Fix $\epsilon > 0$.
We will first construct a compact set $K \subset \R^{2}$ such that $\vis^{2}_{\sigma}(K) \geq t - \epsilon$ for all $\sigma \in [0,1)$.
At the end of the argument, we will explain the (standard but notationally cumbersome) edits required to get rid of the $\epsilon$.

Write $\epsilon_{0} \coloneqq 0$, and set
\begin{equation*}
    \epsilon_{j} \coloneqq \epsilon \cdot (\tfrac{1}{2} + \ldots + 2^{-j}) \in (0,\epsilon), \qquad j \geq 1.
\end{equation*}
We will recursively construct a sequence of dyadic scales $\delta_{0} > \delta_{1} > \delta_{2} > \ldots > 0$ and associated families $\{\mathcal{P}_{\theta} \subset \mathcal{D}_{\delta_{k}} : \theta \in \mathcal{D}_{\delta_{k}}([0,1))\}$.
Set $\delta_{0} \coloneqq 1$ and write $\mathcal{P}_{[0,1)} \coloneqq \{[0,1)^{2}\}$.
Note that 
\begin{equation*}
    \mathcal{H}^{3/2}_{\infty}(\mathcal{P}_{\theta}) = \mathcal{H}^{3/2}_{\infty}([0,1)^{2}) = 1, \qquad \theta \in \mathcal{D}_{\delta_{0}}([0,1)).
\end{equation*}

Assume next that the scales $\delta_{0} > \ldots > \delta_{k}$ have already been constructed for some $k \geq 0$.
Assume also that to every $j \in \{0,\ldots,k\}$ we have associated a family of the form $\{\mathcal{P}_{\theta} \subset \mathcal{D}_{\delta_{j}} : \theta \in \mathcal{D}_{\delta_{j}}([0,1))\}$, such that the following properties are valid:
\begin{enumerate}[nl,label=(J\arabic*)]
    \item\label{J1} Let $k \geq 1$.
        If $\bm{\theta} \in \mathcal{D}_{\delta_{k - 1}}([0,1))$ and $\theta \in \mathcal{D}_{\delta_{k}}(\bm{\theta})$, then $\mathcal{P}_{\theta} \subset \mathcal{D}_{\delta_{k}}(\mathcal{P}_{\bm{\theta}})$.
    \item\label{J2} Let $k \geq 1$.
        Then $\diam_{\theta}(\mathcal{P}_{\theta}) \leq 2\delta_{k - 1}$ for $\theta \in \mathcal{D}_{\delta_{k}}([0,1))$.
    \item\label{J3} Let $k \geq 0$.
        Then $\mathcal{H}^{3/2 - \epsilon_{k}}_{\infty}(\mathcal{P}_{\theta}) \geq 1$ for all $\theta \in \mathcal{D}_{\delta_{k}}([0,1))$.
    \item\label{J4} Let $k \geq 1$, $\bm{\theta},\bm{\theta}' \in \mathcal{D}_{\delta_{k - 1}}([0,1))$ with $\bm{\theta} \neq \bm{\theta}'$, and $\theta \in \mathcal{D}_{\delta_{k}}(\bm{\theta})$ and $\theta' \in \mathcal{D}_{\delta_{k}}(\bm{\theta}')$.
        Then the pair $\mathcal{P}_{\theta},\mathcal{P}_{\theta'}$ is $\theta$-disconnected and $\theta'$-disconnected.
\end{enumerate}
Only \cref{J3} says something non-vacuous for $k = 0$, and that condition with $\epsilon_{0} = 0$ is satisfied by the initial families $\mathcal{P}_{\theta} = \{p_{\theta}\}$, $\theta \in \mathcal{D}_{\delta_{0}}([0,1))$.

We next construct the scale $\delta_{k + 1} \in 2^{-\N} \cap (0,\tfrac{1}{2}\delta_{k}]$ and the family $\{\mathcal{P}_{\theta} \subset \mathcal{D}_{\delta_{k + 1}} : \theta \in \mathcal{D}_{\delta_{k + 1}}([0,1))\}$ by applying \cref{thm1} with parameters 
\begin{equation*}
    \Delta \coloneqq \delta_{k} \quad \text{and} \quad \epsilon \cdot 2^{-k - 1}.
\end{equation*}
The properties \cref{J1}-\cref{J4} are nothing but restatements of \cref{thm1}~\cref{A}-\cref{D}, but let us say a few words as a reminder.

The nestedness property \cref{J1} is immediate from \cref{thm1}~\cref{A}.
Regarding \cref{J2}, \cref{thm1}~\cref{B} tells us that $\diam_{\theta}(\mathcal{P}_{\theta}) \leq 2\Delta = 2\delta_{k}$ for $\theta \in \mathcal{D}_{\delta_{k + 1}}([0,1))$.
Regarding \cref{J3}, noting that $\epsilon_{k + 1} = \epsilon_{k} + \epsilon \cdot 2^{-k - 1}$, \cref{thm1}~\cref{C} tells us that
\begin{equation*}
    \mathcal{H}_{\infty}^{3/2 - \epsilon_{k + 1}}(\mathcal{P}_{\theta}) = \mathcal{H}_{\infty}^{3/2 - \epsilon_{k} - \epsilon \cdot 2^{-k - 1}}(\mathcal{P}_{\theta}) \stackrel{\textup{\cref{C}}}{\geq} \mathcal{H}^{3/2 - \epsilon_{k}}_\infty(\mathcal{P}_{\bm{\theta}}) \geq 1, \qquad \theta \in \mathcal{D}_{\delta_{k + 1}}([0,1)).
\end{equation*}
Finally, to verify \cref{J4}, fix distinct $\bm{\theta},\bm{\theta}' \in \mathcal{D}_{\delta_{k}}([0,1))$, and $\theta \in \mathcal{D}_{\delta_{k + 1}}(\bm{\theta})$ and $\theta' \in \mathcal{D}_{\delta_{k + 1}}(\bm{\theta}')$.
Since we applied \cref{thm1} with $\Delta = \delta_{k}$, \cref{thm1}~\cref{D} now directly tells us that $\mathcal{P}_{\theta},\mathcal{P}_{\theta'}$ is $\theta$-disconnected and $\theta'$-disconnected.

We now use the objects $\{\delta_{k}\}_{k \in \N}$ and the associated families $\{\mathcal{P}_{\theta} \subset \mathcal{D}_{\delta_{k}} : \theta \in \mathcal{D}_{\delta_{k}}([0,1))\}_{k \in \N}$ to build a compact set $K \subset [0,1]^{2}$ with large bi-visible parts.
For $k \in \N$, write
\begin{equation*}
    K_{k} \coloneqq \bigcup_{\theta \in \mathcal{D}_{\delta_{k}}([0,1))} \overline{\cup \mathcal{P}_{\theta}} \subset [0,1]^{2}.
\end{equation*}
Evidently each $K_{k}$ is compact.
We claim that $K_{k + 1} \subset K_{k}$ for all $k \in \N$.
This follows from \cref{J1}: if $\theta \in \mathcal{D}_{\delta_{k + 1}}([0,1))$, then $\mathcal{P}_{\theta} \subset \mathcal{D}_{\delta_{k + 1}}(\mathcal{P}_{\bm{\theta}})$, where $\bm{\theta} \in \mathcal{D}_{\delta_{k}}([0,1))$ is the dyadic parent of $\theta$.
Consequently 
\begin{equation*}
    \overline{\cup \mathcal{P}_{\theta}} \subset \overline{\cup \mathcal{P}_{\bm{\theta}}} \subset K_{k}, \qquad \theta \in \mathcal{D}_{\delta_{k + 1}}([0,1)),
\end{equation*}
which gives $K_{k + 1} \subset K_{k}$.
We now define the compact non-empty set 
\begin{equation*}
    K \coloneqq \bigcap_{k \geq 0} K_{k} \subset [0,1]^{2}.
\end{equation*}
Next, we define a compact subset $K(\sigma) \subset K$ for each $\sigma \in [0,1)$.
Fix $\sigma \in [0,1)$, and let $\{\theta_{k}(\sigma)\}_{k \in \N}$ be the sequence of dyadic arcs satisfying $\sigma \in \theta_{k}(\sigma) \in \mathcal{D}_{\delta_{k}}([0,1))$, $k \in \N$.
Define
\begin{equation}\label{form20}
    K_{k}(\sigma) \coloneqq \overline{\cup \mathcal{P}_{\theta_{k}(\sigma)}}, \qquad \sigma \in [0,1), \, k \in \N.
\end{equation}
Then $K_{k + 1}(\sigma) \subset K_{k}(\sigma)$ by \cref{J1} (same argument that proved $K_{k + 1} \subset K_{k}$).
Therefore
\begin{equation*}
    K(\sigma) \coloneqq \bigcap_{k \geq 0} K_{k}(\sigma), \qquad \sigma \in [0,1),
\end{equation*}
is a non-empty compact set.
Furthermore $K(\sigma) \subset K$, since $K_{k}(\sigma) \subset K_{k}$ for all $k \geq 0$.

Next, we claim that $K(\sigma) \subset \vis^{2}_{\sigma}(K)$ for all $\sigma \in [0,1)$.
Fix $\sigma \in [0,1)$ and $z \in K(\sigma)$.
Let $\ell \coloneqq \ell_{z,\sigma}$ (the line with slope $\sigma$ through $z$).
We claim that 
\begin{equation}\label{form19}
    \diam(\ell \cap K) \leq \diam(\ell \cap K_{k}) \leq 2\delta_{k - 2}, \qquad k \geq 2.
\end{equation}
(The first inequality is trivial.)
Taking $k\to\infty$, the second inequality will show that $\diam(\ell \cap K) = 0$, verifying that $z \in \vis_{\sigma}^{2}(K)$.

To prove \cref{form19}, fix $k \geq 2$, and let $\bm{\theta} \in \mathcal{D}_{\delta_{k - 1}}([0,1))$ and $\theta \in \mathcal{D}_{\delta_{k}}([0,1))$ be the $\delta_{k - 1}$-arc and $\delta_{k}$-arc containing $\sigma$, respectively.
Then
\begin{equation}\label{form21}
    z \in K(\sigma) \subset K_{k}(\sigma) \stackrel{\cref{form20}}{=} \overline{\cup \mathcal{P}_{\theta}}.
\end{equation}
Now comes a central observation.
Let $\theta' \in \mathcal{D}_{\delta_{k}}([0,1)) \, \setminus \, \mathcal{D}_{\delta_{k}}(\bm{\theta})$.
Then $\mathcal{P}_{\theta},\mathcal{P}_{\theta'}$ is $\theta$-disconnected by \cref{J4}.
Since $\ell$ is a line with slope in $\theta$, and $[\ell]_{\delta_{k}/2}$ intersects $\cup \mathcal{P}_{\theta}$ by \cref{form21}, we see that $[\ell]_{\delta_{k}/2} \cap (\cup \mathcal{P}_{\theta'}) = \emptyset$.
In particular
\begin{equation*}
    \ell \cap \left(\overline{\cup \mathcal{P}_{\theta'}}\right) = \emptyset, \qquad \theta' \in \mathcal{D}_{\delta_{k}}([0,1)) \, \setminus \, \mathcal{D}_{\delta_k}(\bm{\theta}).
\end{equation*}
Therefore
\begin{equation*}
    K_{k} \cap \ell \subset \bigcup_{\theta' \in \mathcal{D}_{\delta_{k}}(\bm{\theta})} \overline{\cup \mathcal{P}_{\theta'}} \, \stackrel{\textup{\cref{J1}}}{\subset} \, \overline{\cup \mathcal{P}_{\bm{\theta}}}.
\end{equation*}
Recall that $\diam_{\bm{\theta}}(\mathcal{P}_{\bm{\theta}}) \leq 2\delta_{k - 2}$ by \cref{J2}.
Since $\ell$ is a line with slope $\sigma(\ell) \in \bm{\theta}$, 
\begin{equation*}
    \diam(\overline{\cup \mathcal{P}_{\bm{\theta}}} \cap \ell) \leq \diam(\cup \mathcal{P}_{\bm{\theta}} \cap [\ell]_{\delta_{k - 1}/2}) \leq 2\delta_{k - 2}.
\end{equation*}
Combining the previous two displayed equations proves \cref{form19}.

We have now shown that $K(\sigma) \subset \vis^{2}_{\sigma}(K)$ for all $\sigma \in [0,1)$.
It remains to prove that $\Hd K(\sigma) \geq \tfrac{3}{2} - \epsilon$ for all $\sigma \in [0,1)$.
Let $\{U_{i}\}_{i \in \N}$ be a cover of $K(\sigma)$ by bounded open sets.
Since $K(\sigma)$ is compact,
\begin{equation*}
    \dist\Big(K(\sigma),\R^{2} \, \setminus \, \bigcup_{i \in \N} U_{i} \Big) \eqqcolon \eta > 0.
\end{equation*}
Therefore $K_{k}(\sigma) \subset \cup \{U_{i}\}_{i \in \N}$ for all $k \in \N$ so large that $\delta_{k} \leq \eta/2$.
For any such $k \in \N$,
\begin{equation}\label{form22} \sum_{i \in \N} \diam(U_{i})^{3/2 - \epsilon} \gtrsim \mathcal{H}^{3/2 - \epsilon}_{\infty}(K_{k}(\sigma)) \geq \mathcal{H}_{\infty}^{3/2 - \epsilon}(\mathcal{P}_{\theta_{k}(\sigma)}) \stackrel{\textup{\cref{J3}}}{\geq} 1.\end{equation} 
Therefore $\mathcal{H}^{3/2 - \epsilon}_\infty(K(\sigma)) > 0$, and $\Hd K(\sigma) \geq 3/2 - \epsilon$ for all $\sigma \in [0,1)$.

Let us finally explain how the construction is modified to get the stronger conclusion $\Hd \vis^{2}_{\sigma}(K) \geq 3/2$ for all $\sigma \in [0,1)$.
Recall the construction of the sequences $\{\delta_{k}\}_{k \in \N}$ and the families $\{\mathcal{P}_{\theta} : \theta \in \mathcal{D}_{\delta_{k}}([0,1))\}_{k \in \N}$.
At an arbitrary stage $k(1) \in \N$, it holds $\mathcal{H}^{3/2}_{\infty}(\mathcal{P}_{\theta}) \geq c_{1} > 0$ for all $\theta \in \mathcal{D}_{\delta_{k(1)}}([0,1))$, where $c_{1}$ is the $(3/2)$-dimensional (dyadic) Hausdorff content of a single $\delta_{k(1)}$-square.
From this point on, we may continue the construction in such a way that simultaneously 
\begin{equation*}
    \mathcal{H}^{3/2 - \epsilon_{k}}_{\infty}(\mathcal{P}_{\theta'}) \geq 1 \quad \text{and} \quad \mathcal{H}^{3/2 - \epsilon_{k}/2}_{\infty}(\mathcal{P}_{\theta'}) \geq c_{1}, \qquad \theta' \in \mathcal{D}_{\delta_{k}}([0,1)), \, k \geq k(1).
\end{equation*}
(This simultaneous control is possible because \cref{thm1}~\cref{C} holds for all $t \in [1 + \epsilon,\tfrac{3}{2}]$.)
We repeat this trick countably many times during the construction.
For every $n \in \N$, we designate an index $k(n) \in \N$, and a (decreasing) sequence of lower bounds $c_{1}\ldots,c_{n}> 0$ such that 
\begin{equation*}
    \mathcal{H}_{\infty}^{3/2 - \epsilon_{k}}(\mathcal{P}_{\theta}) \geq 1, \, \mathcal{H}_{\infty}^{3/2 - \epsilon_{k}/2}(\mathcal{P}_{\theta}) \geq c_{1}, \quad \ldots \quad \mathcal{H}_{\infty}^{3/2 - \epsilon_{k}/2^{n}}(\mathcal{P}_{\theta}) \geq c_{n},
\end{equation*}
for all $\theta \in \mathcal{D}_{\delta_{k}}([0,1))$ with $k \geq k(n)$.
Finally the estimate \cref{form22} will show that 
\begin{displaymath} \mathcal{H}^{3/2 - \epsilon/2^{n}}(K(\sigma)) \gtrsim c_{n} \end{displaymath}
for every $n \in \N$, and $\sigma \in [0,1)$.
Therefore $\Hd K(\sigma) \geq 3/2$ for all $\sigma \in [0,1)$.
\begin{remark}\label{r:s1}
    Often visible parts are indexed by directions $e \in S^{1}$ instead of $\sigma \in [-1,1]$.
    For instance, we have stated the main result in this way in the abstract.
    We thank E.\ Järvenpää for pointing out that \cref{main} implies the existence of a compact set $K \subset \R^{2}$ whose (bi-)visible parts are $\tfrac{3}{2}$-dimensional for all $e \in S^{1}$.
    In fact, a union of four rotated and suitably translated copies of the set $K_{0}$ in \cref{main} has the desired property (``$4$'' comes from the fact that the slopes in $[0,1)$ correspond to directions $e \in S^{1}$ making an angle in $[0,\pi/4)$ with the positive $x$-axis.)
    To be precise, let $R_{\varphi}$ be the rotation by angle $\varphi \in [0,\pi)$.
    Then, assuming $K_{0} \subset B(0,\tfrac{1}{4})$ (as we may by dilating), a set of the form
    \begin{equation*}
        K \coloneqq (K_{0} + w_{0}) \cup (R_{\pi/4}(K_{0}) + w_{1}) \cup (R_{\pi/2}(K_{0}) + w_{2}) + (R_{3\pi/4}(K_{0}) + w_{3})
    \end{equation*}
    has the desired property.
    There is plenty of freedom in choosing the translation vectors $w_{0},\ldots,w_{3}$.
    One only needs to ensure that lines passing through $B(w_{j},\tfrac{1}{4})$ with angle in $[j\pi/4,(j + 1)\pi/4)$ do not intersect the discs $B(w_{i},\tfrac{1}{4})$ with $i \in \{0,1,2,3\} \, \setminus \, \{j\}$.
    One concrete choice is given by $(w_{0},w_{1},w_{2},w_{3}) = ((0,0),(-1,2),(-2,0),(2,3))$.
\end{remark} 

\begin{acknowledgements}
    T.O.\ is supported by the European Research Council (ERC) under the European Union’s \emph{Horizon Europe research and innovation programme} (grant agreement no.\ 101087499).
    A.R.\ is supported by the Research Council of Finland via the project \emph{Approximate incidence geometry}, grant no.\ 355453.
\end{acknowledgements}

\appendix

\section{High-low lemma for rays}\label{s:highLowAppendix}

In this section we prove \cref{thm:highLowRays}, repeated below:

\begin{theorem}\label{thm:highLowRaysRestated}
    Let $\mathbf{A},\mathbf{B} \subset \Omega$ be finite sets.
    Then,
    \begin{equation*}
        \Big| \frac{\mathcal{J}_{w}^{+}(\mathbf{A},\mathbf{B})}{w|\mathbf{A}||\mathbf{B}|} - \frac{\mathcal{J}^{+}_{w/2}(\mathbf{A},\mathbf{B})}{(w/2)|\mathbf{A}||\mathbf{B}|} \Big| \leq A_{\chi}\Big(\frac{M_{1 \times w}(\mathbf{A})M_{w \times w}(\mathbf{B})}{|\mathbf{A}||\mathbf{B}|} \cdot w^{-3}\log w^{-1} \Big)^{1/2}, \quad w \in (0,\tfrac{1}{2}],
    \end{equation*}
    where the constant $A_{\chi} > 0$ depends only on $\chi$.
\end{theorem}
\begin{proof}
    The implicit constants in the $\lesssim$ notation in this argument may depend on $\chi$.
    Abbreviate $B(w) \coloneqq \mathcal{J}_{w}^{+}(\mathbf{A},\mathbf{B})/(|\mathbf{A}||\mathbf{B}|w)$.
    With this notation, using the definition \cref{form40},
    \begin{equation*}
        B(w) - B(w/2) = \tfrac{1}{|\mathbf{A}||\mathbf{B}|}\langle \chi_{w/2} \ast (\chi_{w} \ast f_{w} - \chi_{w/4} \ast f_{w/2}),g \rangle.
    \end{equation*}
    Therefore, by Cauchy-Schwarz,
    \begin{equation}\label{form30a}
        |B(w) - B(w/2)| \leq \tfrac{1}{|\mathbf{A}||\mathbf{B}|}\|\chi_{w} \ast f_{w} - \chi_{w/4} \ast f_{w/2}\|_{L^{2}([-2,2]^{2})}\|\chi_{w/2} \ast g\|_{L^{2}([-2,2]^{2})}.
    \end{equation}
    (The integral may be restricted to $[-2,2]^{2}$ since $\spt g = P[\mathbf{B}] \subset [-1,1]^{2}$.)  For the second factor, we use the definition of $M_{w \times w}(\mathbf{B})$ (see \cref{form29}) to deduce that
    \begin{equation}\label{form33a}
        \|\chi_{w/2} \ast g\|_{2}^{2} \lesssim \|g\|\|\chi_{w/2} \ast g\|_{L^{\infty}} \lesssim w^{-2} \cdot |\mathbf{B}|M_{w \times w}(\mathbf{B}).
    \end{equation}
    To bound the first factor in \cref{form30a}, we define
    \begin{equation}\label{form41}
        \Phi_{\omega}^{+} \coloneqq \chi_{w} \ast \frac{\mathbf{1}_{[\ell^{+}_{\omega}]_{w/2}}}{w} - \chi_{w/4} \ast \frac{\mathbf{1}_{[\ell^{+}_{\omega}]_{w/4}}}{w/2}, \qquad \omega \in \mathbf{A}.
    \end{equation}
    We also define a version of the $\Phi$-function where $\ell_{\omega}^{+}$ is replaced by $\ell_{\omega}$:
    \begin{equation}\label{form41a}
        \Phi_{\omega} \coloneqq \chi_{w} \ast \frac{\mathbf{1}_{[\ell_{\omega}]_{w/2}}}{w} - \chi_{w/4} \ast \frac{\mathbf{1}_{[\ell_{\omega}]_{w/4}}}{w/2}, \qquad \omega \in \mathbf{A}.
    \end{equation}
    With this notation, it holds
    \begin{equation*}
        \chi_{w} \ast f_{w} - \chi_{w/4} \ast f_{w/2} = \sum_{\omega \in \mathbf{A}} \Phi_{\omega}^{+}.
    \end{equation*}
    Let $\psi \in C_{c}^{\infty}(\R^{2})$ be a function with $\mathbf{1}_{[-2,2]^{2}} \leq \psi \leq \mathbf{1}_{[-3,3]^{2}}$.
    Then,
    \begin{equation}\label{form32a}
        \|\chi_{w} \ast f_{w} - \chi_{w/4} \ast f_{w/2}\|_{L^{2}([-2,2]^{2})}^{2} \leq \sum_{\omega_{1},\omega_{2} \in \mathbf{A}} \Big| \int \Phi^{+}_{\omega_{1}}\Phi^{+}_{\omega_{2}}\psi \Big|.
    \end{equation}
    The most technical argument in \cite[Appendix A.1]{zbmath:8038252} concerns the related integral $\int \Phi_{\omega_{1}}\Phi_{\omega_{2}}\psi$.
    It is proved (in an unnumbered display towards the end of \cite[Appendix A.1]{zbmath:8038252}) that
    \begin{equation}\label{form44a}
        \Big| \int \Phi_{\omega_{1}}\Phi_{\omega_{2}}\psi \Big| \lesssim \min \{w^{-1},w^{2}/d_{\mathcal{A}}(\ell_{\omega_{1}},\ell_{\omega_{2}})^{3}\}.
    \end{equation}
    This estimate is (as far as we can tell) not true with $\Phi_{\omega_{j}}^{+}$ in place of $\Phi_{\omega_{j}}$.
    However, it is true if the origin of $\ell_{1}$ stays at distance $\sim w$ away from $\ell_{2}$, and vice versa.

    \begin{claim}\label{c4}
        If $\dist(z_{\omega_{1}},\ell_{\omega_{2}}) \geq 20w$ and $\dist(z_{\omega_{2}},\ell_{\omega_{1}}) \geq 20w$, then
        \begin{equation}\label{form44}
            \Big| \int \Phi_{\omega_{1}}^{+}\Phi_{\omega_{2}}^{+}\psi \Big| \lesssim \min \{w^{-1},w^{2}/d_{\mathcal{A}}(\ell_{\omega_{1}},\ell_{\omega_{2}})^{3}\}.
        \end{equation}
    \end{claim}
    The proof is based on \cref{form44a} and the following geometric observation:
    \begin{claim}\label{c3}
        Assume that $\dist(z_{\omega_{1}},\ell_{\omega_{2}}) \geq 20w$ and $\dist(z_{\omega_{2}},\ell_{\omega_{1}}) \geq 20w$.
        Then
        \begin{equation}\label{form42}
            \Phi_{\omega_{1}}^{+}\Phi_{\omega_{2}}^{+} \equiv 0 \quad \text{or} \quad \Phi_{\omega_{1}}^{+}\Phi_{\omega_{2}}^{+} \equiv \Phi_{\omega_{1}}\Phi_{\omega_{2}}.
        \end{equation}
    \end{claim}
    \begin{proof}[of \cref{c4} assuming \cref{c3}]
        If the first alternative in \cref{form42} holds, then \cref{form44} is clear.
        If the second alternative holds, then
        \begin{equation*}
            \Big| \int \Phi_{\omega_{1}}^{+}\Phi_{\omega_{2}}^{+}\psi\Big| = \Big| \int \Phi_{\omega_{1}}\Phi_{\omega_{2}}\psi \Big|,
        \end{equation*}
        and \cref{form44} is a consequence of \cref{form44a}.
    \end{proof}
    We then prove \cref{c3}.
    \begin{proofref}{c3}
        Abbreviate $\ell_{\omega_{j}} \eqqcolon \ell_{j}$ and $\ell_{\omega_{j}}^{+} \eqqcolon \ell_{j}^{+}$ and similarly $\Phi_{\omega_{j}} \eqqcolon \Phi_{j}$ and $\Phi_{\omega_{j}}^{+} \eqqcolon \Phi_{j}^{+}$.
        Let us also write
        \begin{equation*}
            \ell_{j} = \{(x, a_{j}x + b_{j}): x \in \mathbb{R}\} \quad \text{and} \quad \ell_{j}^{+} = \{a_{j}x + b_{j} : x \geq c_{j}\}
        \end{equation*}
        for some $a_{j},c_{j} \in [-1,1]$.
        With this notation $z_{j} \coloneqq z_{\omega_{j}} = (c_{j},a_{j}c_{j} + b_{j})$.

        Assume first that $\ell_{1}^{+} \cap \ell_{2}^{+} = \emptyset$.
        We claim that $\Phi_{1}^{+}\Phi_{2}^{+} \equiv 0$.
        Assume to the contrary that $(\Phi_{1}^{+}\Phi_{2}^{+})(x,y) > 0$ for some $(x,y) \in \R^{2}$.
        We may and will assume that $\max\{c_{1},c_{2}\} = c_{1}$.
        Note that
        \begin{equation*}
            \spt \Phi_{j}^{+} \subset \{(x',y') : x' \geq c_{j} - 3w \text{ and } |y' - (a_{j}x' + b_{j})| \leq 3w\}, \qquad j \in \{1,2\}.
        \end{equation*}
        Therefore $x \geq \max\{c_{1},c_{2}\} - 3w = c_{1} - 3w$, and
        \begin{equation*}
            |h(x)| \leq 6w \quad \text{where} \quad h(x) \coloneqq (a_{1}x + b_{1}) - (a_{2}x + b_{2}).
        \end{equation*}
        We first dispose of the case where $c_{1} - 3w \leq x \leq c_{1}$.
        Since $h$ is $2$-Lipschitz, $|h(c_{1})| \leq 12w$.
        This implies $\dist(z_{1},\ell_{2}) < 20w$ contrary to hypothesis.
        So, we may assume that $x \geq c_{1}$.

        Either $h(x) \geq 0$ or $h(x) < 0$, and the treatment of these cases is similar, so we only consider $h(x) \geq 0$.
        Note that $h' = a_{1} - a_{2}$ is constant.
        Assume first that $h' \geq 0$, in which case $h(c_{1}) \leq h(x) \leq 6w$.
        There are now two options: either $h(c_{1}) \geq 0$ or $h(c_{1}) < 0$.
        In the first case $h(c_{1}) \in [0,h(x)] \subset [0,6w]$, which implies that $\dist(z_{1},\ell_{2}) < 20w$.
        In the opposite case $h(c_{1}) < 0$, there exists a point $\xi \in (c_{1},x] = (\max\{c_{1},c_{2}\},x]$ with $h(\xi) = 0$.
        This gives $\ell_{1}^{+} \cap \ell_{2}^{+} \neq \emptyset$, again contrary to assumption.

        Assume finally that $h' < 0$.
        Since $h(x) \geq 0$, there exists $\xi \geq x \geq \max\{c_{1},c_{2}\}$ such that $h(\xi) = 0$.
        This means that $\ell^{+}_{1} \cap \ell^{+}_{2} \neq \emptyset$, again contrary to assumption.
        We have now reached a contradiction in all cases, and proved that $\ell_{1}^{+} \cap \ell^{+}_{2} = \emptyset$ implies $\Phi_{1}^{+}\Phi_{2}^{+} \equiv 0$.

        Next, assume that $\ell_{1}^{+} \cap \ell^{+}_{2} \neq \emptyset$.
        In this case we will no longer use the hypothesis $\max\{c_{1},c_{2}\} = c_{1}$.
        We claim that $\Phi_{1}^{+}\Phi_{2}^{+} \equiv \Phi_{1}\Phi_{2}$.
        Assume to the contrary that there exists $(x,y) \in \R^{2}$ such that
        \begin{equation*}
            (\Phi_{1}^{+}\Phi_{2}^{+})(x,y) \neq (\Phi_{1}\Phi_{2})(x,y).
        \end{equation*}
        Then either $\Phi_{1}^{+}(x,y) \neq \Phi_{1}(x,y)$ or $\Phi_{2}^{+}(x,y) \neq \Phi_{2}(x,y)$; let us assume that $\Phi_{1}^{+}(x,y) \neq \Phi_{1}(x,y)$.
        Then $x \leq c_{1} + 3w$, since $\Phi_{1}^{+}(x',y') = \Phi_{1}(x',y')$ for $x' > c_{1} + 3w$ and $y' \in \R$.

        Recall that $h(x) = (a_{1}x + b_{1}) - (a_{2}x + b_{2})$.
        Observe that $h(x) \leq 6w$, since otherwise $(\Phi_{1}^{+}\Phi_{2}^{+})(x,y) = 0 = (\Phi_{1}\Phi_{2})(x,y)$.
        We dispose of the special case where $c_{1} \leq x \leq c_{1} + 2w$.
        Then $h(c_{1}) \leq 12w$, and therefore $\dist(z_{1},\ell_{2}) < 20w$ contrary to hypothesis.
        So, we may assume that $x < c_{1}$.

        Since we assumed that $\ell_{1}^{+} \cap \ell^{+}_{2} \neq \emptyset$, and $x < c_{1}$, there exists $\xi \geq c_{1} > x$ such that $h(\xi) = 0$.
        But now $h(c_{1}) \in [h(\xi),h(x)]$ (since $h$ is monotone), and consequently $|h(c_{1})| \leq 6w$.
        This implies that $\dist(z_{1},\ell_{2}) < 20w$, contrary to hypothesis.
        We have now proved that $\ell_{1}^{+} \cap \ell_{2}^{+} \neq \emptyset$ implies $\Phi_{1}^{+}\Phi_{2}^{+} \equiv \Phi_{1}\Phi_{2}$, and the proof of \cref{c3} is complete.
    \end{proofref}
    We proceed to estimate the right hand side of \cref{form32a}.
    Write
    \begin{equation*}
        \mathbf{Bad} \coloneqq \{(\omega_{1},\omega_{2}) \in \mathbf{A} \times \mathbf{A} : \dist(z_{\omega_{1}},\ell_{\omega_{2}}) \leq 20w \text{ or } \dist(z_{\omega_{2}},\ell_{\omega_{1}}) \leq 20w\}.
    \end{equation*}
    Also set $\mathbf{Bad}(\omega_{1}) \coloneqq \{\omega_{2} \in \mathbf{A} : (\omega_{1},\omega_{2}) \in \mathbf{Bad}\}$, $\omega_{1} \in \mathbf{A}$.
    Fixing $\omega_{1} \in \mathbf{A}$, and using \cref{c4},
    \begin{align*}
        \sum_{\omega_{2} \in \mathbf{A} \, \setminus \, \mathbf{Bad}(\omega_{1})} \Big| \int \Phi_{\omega_{1}}^{+}\Phi_{\omega_{2}}^{+}\psi \Big|
        & \lesssim w^{-1} |\{\omega_{2} \in \mathbf{A} : \ell_{\omega_{2}} \in B_{d_{\mathcal{A}}}(\ell_{\omega_{1}},w)\}|\\
        & \quad + \sum_{w \leq 2^{-j} \leq 100} w^{2} 2^{3j} |\{\omega_{2} \in \mathbf{A} : \ell_{\omega_{2}} \in B_{d_{\mathcal{A}}}(\ell_{\omega_{1}},2^{-j})\}|\\
        & \leq w^{-1}M_{1 \times w}(\mathbf{A}) + w^{2} \sum_{w \leq 2^{-j} \leq 100} 2^{3j} M_{1 \times 2^{-j}}(\mathbf{A}).
    \end{align*}
    Note that since $B_{d_{\mathcal{A}}}(\ell_{0},r)$ can be covered by $\lesssim (r/w)^{2}$ many $w$-balls for $r \geq w$, it holds $M_{1 \times r}(\mathbf{A}) \lesssim (r/w)^{2}M_{1 \times w}(\mathbf{A})$.
    Inputting this back into the estimate above,
    \begin{equation*} \sum_{\omega_{2} \in \mathbf{A} \, \setminus \, \mathbf{Bad}(\omega_{1})} \Big| \int \Phi_{\omega_{1}}^{+}\Phi_{\omega_{2}}^{+}\psi \Big|  \lesssim w^{-1}M_{1 \times w}(\mathbf{A}) + \sum_{w \leq 2^{-j} \leq 100} 2^{j} M_{1 \times w}(\mathbf{A}) \lesssim w^{-1}M_{1 \times w}(\mathbf{A}).
    \end{equation*}
    Summing over $\omega_{1} \in \mathbf{A}$, we therefore find
    \begin{equation}\label{form46}
        \sum_{(\omega_{1},\omega_{2}) \in (\mathbf{A} \times \mathbf{A}) \, \setminus \, \mathbf{Bad}} \Big| \int \Phi^+_{\omega_{1}}\Phi^+_{\omega_{2}}\psi \Big| \lesssim w^{-1}|\mathbf{A}|M_{1 \times w}(\mathbf{A}).
    \end{equation}

    We next aim for a similar bound for the sum over $(\omega_{1},\omega_{2}) \in \mathbf{Bad}$, which we organise as follows:
    \begin{align*}
        \sum_{(\omega_{1},\omega_{2}) \in \mathbf{Bad}} \Big| \int \Phi^+_{\omega_{1}}\Phi^+_{\omega_{2}}\psi \Big|
        \leq{}& \sum_{\omega_{1} \in \mathbf{A}} \sum_{\substack{\omega_{2} \in \mathbf{A}\\\dist(z_{\omega_{1}},\ell_{\omega_{2}}) \leq 20w}} \Big| \int \Phi^+_{\omega_{1}}\Phi^+_{\omega_{2}}\psi \Big|\\*
              &+ \sum_{\omega_{2} \in \mathbf{A}} \sum_{\substack{\omega_{1} \in \mathbf{A}\\\dist(z_{\omega_{2}},\ell_{\omega_{1}}) \leq 20w}} \Big| \int \Phi^+_{\omega_{1}}\Phi^+_{\omega_{2}}\psi \Big|\\*
              &\eqqcolon \Sigma_{1} + \Sigma_{2}.
    \end{align*}
    The bounds for $\Sigma_{1}$ and $\Sigma_{2}$ are symmetric, so we only consider $\Sigma_{1}$.
    Now we no longer have the ``orthogonality bound'' \cref{form44} available, so we will have to make do with the following ``non-cancellative'' bound:
    \begin{equation}\label{form47}
        \Big| \int \Phi_{\omega_{1}}^{+}\Phi^{+}_{\omega_{2}}\psi \Big| \lesssim \min\{w^{-1},d_{\mathcal{A}}(\ell_{\omega_{1}},\ell_{\omega_{2}})^{-1}\}.
    \end{equation}
    This follows from the observations that $\|\Phi_{\omega_{1}}^{+}\Phi_{\omega_{2}}^{+}\|_{L^{\infty}} \lesssim w^{-2}$, and if $d_{\mathcal{A}}(\ell_{\omega_{1}},\ell_{\omega_{2}}) \geq w$, then
    \begin{equation*}
        (\spt \Phi_{\omega_{1}}^{+}) \cap (\spt \Phi_{\omega_{2}}^{+})
    \end{equation*}
    is contained in a rectangle of dimensions $\sim w \times (w/d_{\mathcal{A}}(\ell_{1},\ell_{2}))$.
    For $d_{\mathcal{A}}(\ell_{\omega_{1}},\ell_{\omega_{2}}) \leq w$, one relies on the trivial fact that $\spt (\Phi_{\omega_{1}}^{+}\psi)$ is contained in a rectangle of dimensions $\sim w \times 1$.

    Using \cref{form47}, one gets
    \begin{equation}\label{form48}
        \Sigma_{1} \lesssim \sum_{\omega_{1} \in \mathbf{A}} \Big( \sum_{\substack{\omega_{2} \in \mathbf{A}\\d_{\mathcal{A}}(\ell_{\omega_{1}},\ell_{\omega_{2}}) \leq w}} w^{-1}
        + \sum_{\substack{\omega_{2} \in \mathbf{A}\\d_{\mathcal{A}}(\ell_{\omega_{1}},\ell_{\omega_{2}}) \geq w\\\dist(z_{\omega_{1}},\ell_{\omega_{2}}) \leq 20w}} d_{\mathcal{A}}(\ell_{\omega_{1}},\ell_{\omega_{2}})^{-1} \Big).
    \end{equation}
    The first sum in brackets is evidently bounded by $w^{-1}M_{1 \times w}(\mathbf{A})$.
    To bound the second sum, we make the following observation: the relevant family of lines
    \begin{equation*}
        A(z_{\omega_{1}}) \coloneqq \{\ell_{2} \in \mathcal{A}(2,1) : \dist(z_{\omega_{1}},\ell_{2}) \leq 20w\},
    \end{equation*}
    viewed as a set of parameters in $[-1,1]^{2}$, is contained in the $O(w)$-neighbourhood of a certain line, determined by $z_{\omega_{1}}$.
    We partition $A(z_{\omega_{1}}) \, \setminus \, B_{d_{\mathcal{A}}}(\ell_{\omega_{1}},w)$ according to the distance between $\ell_{\omega_1}$ and $\ell_2\in A(z_{\omega_1})$:
    \begin{equation*}
        A(z_{\omega_{1}}) \, \setminus \, B_{d_{\mathcal{A}}}(\ell_{\omega_{1}},w) \eqqcolon \bigcup_{w \leq 2^{-j} \leq 1} A_{j},
    \end{equation*}
    where $A_{j} \coloneqq \{\ell_{2} \in A(z_{\omega_{1}}) : d_{\mathcal{A}}(\ell_{\omega_{1}},\ell_{2}) \in [2^{-j},2^{-j + 1})\}$.
    In parameter space, $A_{j}$ is contained in a rectangle of dimensions $\sim w \times 2^{-j}$.
    Therefore $A_{j}$ may be covered by a family $\mathcal{B}_{j}$ of $w$-balls of cardinality $\sim (2^{-j}/w)$.
    Since $|\{\omega_{2} \in \mathbf{A} : \ell_{\omega_{2}} \in B\}| \leq M_{1 \times w}(\mathbf{A})$ for each $B \in \mathcal{B}_{j}$, we obtain
    \begin{align*}
        \sum_{\substack{\omega_{2} \in \mathbf{A} : d_{\mathcal{A}}(\ell_{\omega_{1}},\ell_{\omega_{2}}) \geq w\\d(z_{\omega_{1}},\ell_{\omega_{2}}) \leq 20w}} d(\ell_{\omega_{1}},\ell_{\omega_{2}})^{-1}
        & \lesssim \sum_{w \leq 2^{-j} \leq 1} 2^{j} \cdot |\{\omega_{2} \in \mathbf{A} : \ell_{\omega_{2}} \in A_{j}\}|\\
        &  \leq \sum_{w \leq 2^{-j} \leq 1} 2^{j} \sum_{B \in \mathcal{B}_{j}} |\{\omega_{2} \in \mathbf{A} : \ell_{\omega_{2}} \in B\}|\\
        & \lesssim (w^{-1}\log w^{-1})M_{1 \times w}(\mathbf{A}).
    \end{align*}
    Substituting this estimate back into \cref{form48} yields $\Sigma_{1} \lesssim (w^{-1} \log w^{-1})|\mathbf{A}|M_{1 \times w}(\mathbf{A})$, and therefore
    \begin{equation*}
        \sum_{(\omega_{1},\omega_{2}) \in \mathbf{Bad}}\Big| \int \Phi_{\omega_{1}}^{+}\Phi_{\omega_{2}}^{+}\psi \Big| \lesssim (w^{-1}\log w^{-1})|\mathbf{A}|M_{1 \times w}(\mathbf{A}).
    \end{equation*}
    Combining this estimate with the bound \cref{form46} for the ``good'' line-pairs, we find
    \begin{equation*}
        \sum_{\omega_{1},\omega_{2} \in \mathbf{A}} \Big| \int \Phi_{\omega_{1}}^{+}\Phi_{\omega_{2}}^{+}\psi \Big| \lesssim (w^{-1} \log w^{-1})|\mathbf{A}|M_{1 \times w}(\mathbf{A}).
    \end{equation*}
    Substituting this bound into \cref{form32a},
    \begin{equation*}
        \|\chi_{w} \ast f_{w} - \chi_{w/4} \ast f_{w/2}\|_{L^{2}([-2,2]^{2})}^{2} \lesssim (w^{-1} \log w^{-1})|\mathbf{A}|M_{1 \times w}(\mathbf{A}).
    \end{equation*}
    Finally, plugging this estimate and \cref{form33a} back into \cref{form30a} yields
    \begin{equation*}
        |B(w) - B(w/2)| \lesssim \Big(\frac{M_{1 \times w}(\mathbf{A})M_{w \times w}(\mathbf{B})}{|\mathbf{A}||\mathbf{B}|} \cdot w^{-3} \log w^{-1}\Big)^{1/2}.
    \end{equation*}
    This completes the proof of \cref{thm:highLowRays}.
\end{proof}

\section{Basic building block}\label{appA}

The set $\mathfrak{P}$ and the tubes $\mathcal{T}_{\sigma}$ claimed in \cref{p:bbb} can be obtained by ``dualising'' a suitable grid construction (\cref{fig1}) which was used to construct the main examples in \cite{zbmath:7846313}.
We start by recalling the details of that construction from \cite[Appendix A.1]{zbmath:7846313}, and then we will show to transform this object into the one claimed in \cref{p:bbb}.

We define a piece of notation.
Recall that $\pi_{\zeta}(x,y) \coloneqq \zeta x + y$ for $\zeta \in [0,1)$ and $(x,y) \in \R^{2}$.
If $\mathcal{P} \subset \mathcal{D}_{\delta}$ is a family of dyadic $\delta$-squares, $\pi_{0}(\cup \mathcal{P})$ is a family of dyadic $\delta$-intervals.
We write $Y_{\mathcal{P}} \subset \delta \Z$ for the left-endpoints of these intervals.

\begin{lemma}\label{lemma1a}
    Let $\tau \in (1,2]$, $s \in (0,2 - \tau]$, and $\Delta \in 2^{-\N}$.
    For all $\delta \in 2^{-\N} \cap (0,\Delta]$ sufficiently small in terms of $\Delta,s,\tau$, there exists a set $\mathcal{P} \subset \mathcal{D}_{\delta}([-3,3]^{2})$ with the following properties.
    \begin{enumerate}[nl, label=(P\arabic*)]
        \item\label{P1} If $\sigma \in \delta \Z \cap [0,1)$, then the vertical ``slice'' $\{p \in \mathcal{P} : p \cap (\{\sigma\} \times \R) \neq \emptyset\}$ contains a $\sim \delta^{\tau - 1}$-separated subset $\mathcal{P}_{\sigma}$ with the following property: if $I \subset [-2,2]$ is an interval of length $\ell(I) \geq \Delta$, then 
            \begin{equation}\label{form24} |Y_{\mathcal{P}_{\sigma}} \cap I| \sim \ell(I)\delta^{1 -\tau}.
            \end{equation}

        \item\label{P2} There exists a $\delta^{s}$-separated subset $\Sigma \subset \delta \Z \cap [0,1]$ such that $|\pi_{\zeta}(\mathcal{P})|_{\delta} \lesssim \delta^{-(s + \tau)/2}$ for all $\zeta \in \Sigma$, and such that $\Sigma$ has the following property.
            If $I \subset [0,1]$ is an interval of length $\ell(I) \geq \Delta$, then $|\Sigma \cap I| \sim \ell(I)\delta^{-s}$.
    \end{enumerate}
\end{lemma}

The family $\mathcal{P} \subset \mathcal{D}_{\delta}([-3,3]^{2})$ in \cref{lemma1a} is illustrated in \cref{fig1}, reproduced from \cite[Appendix A.1]{zbmath:7846313}.
The family $\mathcal{P}$ is obtained by slightly rotating a grid of $\delta$-squares of cardinality $\delta^{-\tau}$ in such a way that every vertical ``slice'' contains the expected number $\sim \delta^{1 - \tau}$ of $\delta$-squares.

\begin{figure}[t]
    \begin{center}
        \input{figures/grid.tex}
        \caption{The set in \cref{lemma1a}.}\label{fig1}
    \end{center}
\end{figure}

Slightly coarser versions of \cref{P1}-\cref{P2} are stated in \cite[(P1)-(P2)]{zbmath:7846313}; the level precision in \cref{P1}-\cref{P2} was not needed in \cite{zbmath:7846313}.
However, the proofs in \cite[Appendix A.1]{zbmath:7846313} give exactly what we need in \cref{P1}-\cref{P2}.
We explain this in the following two remarks.

\begin{remark}
    Regarding \cref{P1}, it seems likely that one can simply take 
    \begin{equation*}
        \mathcal{P}_{\sigma} \coloneqq \{p \in \mathcal{P} : p \cap (\{\sigma\} \times \R) \neq \emptyset\},
    \end{equation*}
    but the existence of the claimed subset is directly guaranteed by \cite[Claim A.2]{zbmath:7846313}.
    The proof of \cite[Claim A.2]{zbmath:7846313} shows that $\{p \in \mathcal{P} : p \cap (\{\sigma\} \times \R) \neq \emptyset\}$ contains an arithmetic progression with gap $\delta^{\tau - 1}$ and length $2\delta^{1 - \tau}$.
    
    Now $\mathcal{P}_{\sigma}$ can be defined to be this progression, and the ``density property'' \cref{form24} holds for $\delta > 0$ so small that $2\delta^{\tau - 1} \leq \Delta$.
    (A technical point here is that since \cref{form24} is supposed to hold for all intervals $I \subset [-2,2]$, in particular for $I = [-2,2]$, we need the progression to have diameter $\geq 2$.
    This is the reason why we initially defined $\mathcal{P}$ to be grid in $[-3,3]^{2}$ instead of just $[0,1]^{2}$ as opposed to \cite{zbmath:7846313}.
    With this modification, one can check from the proof of \cite[Claim A.2]{zbmath:7846313} that the progression $Y_{\mathcal{P}_{\sigma}}$ "covers" $[-2,2]$ in the sense \cref{form24}.)
\end{remark} 
\begin{remark}
    The set $\Sigma$ in \cref{P2} is defined to be a maximal $\delta^{s}$-separated subset of the family ``$\mathcal{E}$'' appearing in \cite[(A.3)]{zbmath:7846313}.
    Now \cite[Claim A.5]{zbmath:7846313} states that $|\Sigma \cap I| \sim \ell(I)\delta^{-s}$ for all intervals $I \subset [0,1]$ of length $\ell(I) \geq C\delta^{s/2}\log(1/\delta)$, where $C \geq 1$ is absolute.
    In particular, \cref{P2} holds provided that $\delta > 0$ is so small that $C\delta^{s/2}\log(1/\delta) \leq \Delta$.
    The upper bound $|\pi_{\zeta}(\mathcal{P})|_{\delta} \lesssim \delta^{-(s + \tau)/2}$, for $\zeta \in \Sigma$, is established in \cite[Claim A.4]{zbmath:7846313}.
\end{remark}

We are then ready to prove \cref{p:bbb}, which we restate here:
\begin{restatement}{p:bbb}
    For every $\epsilon,\Delta \in 2^{-\N}$, there exists $\delta_{0} = \delta_{0}(\epsilon,\Delta) \in (0,\tfrac{1}{2}\Delta]$ such that the following holds for all $\delta \in 2^{-\N} \cap (0,\delta_{0}]$.
    There exists a family $\mathfrak{P} \subset \mathcal{D}_{\delta}$ with the following property.
    Assume that $t \in [1 + \epsilon,\tfrac{3}{2}]$, and $\mathcal{Q} \subset \mathcal{D}_{\Delta}$.
    Then 
    \begin{equation}\label{form3app}
        \mathcal{H}_{\infty}^{t - \epsilon}(\mathcal{Q} \cap \mathfrak{P}) \geq \mathcal{H}^{t}_{\infty}(\mathcal{Q}).
    \end{equation}
    Moreover, for each $\sigma \in \delta \Z \cap [0,1)$ there exists a family $\mathcal{R}_{\sigma}$ of $(\delta \times \Delta)$-rectangles whose longer sides have slope $\sigma$, with the following properties:
    \begin{enumerate}
        \item Let $\mathcal{T}_{\sigma}$ be the $\delta$-tubes spanned by the rectangles $\mathcal{R}_{\sigma}$.
            Every tube in $\mathcal{T}_{\sigma}$ contains exactly one element of $\mathcal{R}_{\sigma}$, and the tubes in $\mathcal{T}_{\sigma}$ are $10\delta$-separated.
        \item $\cup \mathfrak{P} \cap (\cup \{10T : T \in \mathcal{T}_{\sigma}\}) = \emptyset$.
        \item If $t \in [1 + \epsilon,\tfrac{3}{2}]$, and $\mathcal{Q} \subset \mathcal{D}_{\Delta}$  then 
            \begin{equation*} \mathcal{H}^{t - \epsilon}_{\infty}(\mathcal{Q} \cap \mathcal{D}_{\delta}(\cup \mathcal{R}_{\sigma})) \geq \mathcal{H}^{t}_{\infty}(\mathcal{Q}).
            \end{equation*}
    \end{enumerate}
\end{restatement} 
\begin{proof}
    Fix $\epsilon,\Delta \in 2^{-\N}$ as in the statement (we may assume that $\epsilon \in (0,1)$).
    Fix also $\underline{\Delta} \in 2^{-\N}\cap(0,\tfrac{1}{100}\Delta]$ so small that
    \begin{equation}\label{form9a}
        \underline{\Delta}^{-\epsilon/2} \geq \max\{C,\Delta^{-1}\},
    \end{equation} 
    where $C > 0$ is an absolute constant to be determined later.
    Then, we apply \cref{lemma1a} with this $\underline{\Delta}$,
    \begin{equation*}
        \tau \coloneqq \tfrac{3}{2} \quad \text{and} \quad s \coloneqq (1 - \epsilon)/2 \in (0,2 - \tau],
    \end{equation*}
    and $\delta \in (0,\underline{\Delta}]$ so small that the conditions of \cref{lemma1a} are satisfied, and additionally
    \begin{equation}\label{form7a}
        C\delta^{\epsilon/2} \leq \underline{\Delta}^{4}/2.
    \end{equation}
    (The largest ``$\delta$'' so that the previous conditions are met is now defined to be the scale threshold ``$\delta_{0}$'' in \cref{p:bbb}.)
    Let $\mathcal{P} \subset \mathcal{D}_{\delta}([-3,3]^{2})$ be the resulting set.

    \begin{figure}[t]
        \begin{center}
            \input{figures/rectangle_choice.tex}
            \caption{The choice of the rectangles $R_{T}$, $T \in \mathcal{T}_{\sigma}$, shown in red.
            For artistic reasons the picture shows the rectangles $T_{\sigma,b_{j}} \cap \bar{\pi}_{\sigma}^{-1}(I_{i})$; the real ones defined in \cref{def:rectangles} have twice the length.}\label{fig3}
        \end{center}
    \end{figure}
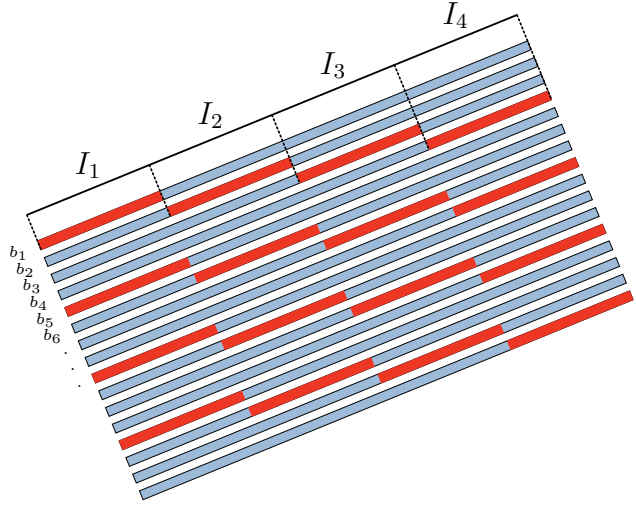

    We now define the families of rectangles and tubes appearing in the statement.
    Fix $\sigma \in \delta \Z \cap [0,1)$, and recall the (subset of the) vertical ``slice'' $\mathcal{P}_{\sigma}$ from \cref{lemma1a}~\cref{P1}.
    We define the $\delta$-tube family $\mathcal{T}_{\sigma}$ as the dual of $\mathcal{P}_{\sigma}$, under standard point-line duality.
    Abbreviate $Y_{\sigma} \coloneqq Y_{\mathcal{P}_{\sigma}}$ (recall the notation from \cref{lemma1a}), and define the family of (parallel slope-$\sigma$) lines
    \begin{equation*}
        \mathcal{L}_{\sigma} \coloneqq \{\ell_{\sigma,b} : b \in Y_{\sigma}\},
    \end{equation*}
    where $\ell_{\sigma,b} = \{(x,y) \in \R^{2} : y = \sigma x + b\}$.
    Finally, define the $\delta$-tubes 
    \begin{equation*}
        \mathcal{T}_{\sigma} \coloneqq \{T_{\sigma,b} : b \in Y_{\sigma}\} \coloneqq  \{[\ell_{\sigma,b}]_{\delta/2} : b \in Y_{\sigma}\}.
    \end{equation*}
    Note that the lines in $\mathcal{L}_{\sigma}$ have roughly the same separation as the squares in $\mathcal{P}_{\sigma}$, which is $\sim \delta^{\tau - 1} = \delta^{1/2}$ according to \cref{lemma1a}~\cref{P1}.
    In particular, the tubes in $\mathcal{T}_{\sigma}$ are $10\delta$-separated (assuming $\delta > 0$ small enough).
    This proves \cref{p:bbb}~\cref{1}.

    For each $T \in \mathcal{T}_{\sigma}$, we plan to pick a single $(\delta \times \Delta)$-rectangle $R_{T} \subset T$.
    Once this has been accomplished, we will define the family of $(\delta \times \Delta)$-rectangles appearing in the statement of \cref{p:bbb} as
    \begin{equation*} \mathcal{R}_{\sigma} \coloneqq \{R_{T} : T \in \mathcal{T}_{\sigma}\}.
    \end{equation*}
    Likely one (working) way of choosing the rectangles $R_{T}$ would be a uniformly random selection among all those $(\delta \times \Delta)$-subrectangles of $T$ which hit $[0,1]^{2}$ (or an arbitrary rectangle if $T \cap [0,1]^{2} = \emptyset$).
    For technical reasons, a deterministic choice is slightly preferable, where the rectangles are chosen as well-spread as possible.
    The idea can be deciphered from \cref{fig3}, but also need an analytic expression to work with.

    Let $\bar{\pi}_{\sigma}$ be the orthogonal projection to the subspace parallel to lines in $\mathcal{L}_{\sigma}$ (or the tubes $\mathcal{T}_{\sigma}$).
    Enumerate $\mathcal{D}_{\Delta/2}([-2,2]) \eqqcolon \{I_{1},\dots,I_{m}\}$, where $m \sim \Delta^{-1}$.
    Enumerate also $Y_{\sigma} \coloneqq \{b_{1},\ldots,b_{n}\}$ in increasing order, where $n = |Y_{\sigma}| \sim \delta^{-1/2}$.
    Note that $n \gg m$.
    Now, for each $b_{j} \in Y_{\sigma}$, define the $(\delta \times \Delta)$-rectangle
    \begin{equation}\label{def:rectangles} R_{\sigma,b_{j}} \coloneqq T_{\sigma,b_{j}} \cap \bar{\pi}_{\sigma}^{-1}(2I_{i}) \subset T_{\sigma,b_{j}}, \qquad j = i (\mathrm{mod\,} m).
    \end{equation}

    Now the tubes and rectangles in the statement of \cref{p:bbb} have been defined.
    Next, we define the family $\mathfrak{P} \subset \mathcal{D}_{\delta}$, and verify the (remaining) properties \cref{form3app} and (2)-(3).
    Let
    \begin{equation*}
        \mathcal{T} \coloneqq \bigcup_{\sigma \in \delta \Z \cap [0,1)} \mathcal{T}_{\sigma}.
    \end{equation*}
    Thus, $\mathcal{T}$ is roughly the ``tube dual'' of $\mathcal{P}$, or more precisely the dual of the subset obtained as the union of the sets $\mathcal{P}_{\sigma}$, $\sigma \in \delta \Z \cap [0,1)$.
    Recall from \cref{lemma1a}~\cref{P2} that $|\pi_{\zeta}(\mathcal{P})|_{\delta} \lesssim \delta^{-(s + \tau)/2} = \delta^{-1 + \epsilon/2}$ for all $\zeta \in \Sigma$.
    Write $10\mathcal{T} \coloneqq \{[\ell]_{5\delta} : [\ell]_{\delta/2} \in \mathcal{T}\}$.
    For $\zeta \in [0,1]$, write also $(10\mathcal{T})_{\zeta}$ for the vertical $\zeta$-``slice'' of the tube family $10\mathcal{T}$:
    \begin{equation*} (10\mathcal{T})_{\zeta} \coloneqq \{q = [\zeta,\zeta + \delta) \times [b,b + \delta) \in \mathcal{D}_{\delta} : q \cap 10T \neq \emptyset \text{ for some } T \in \mathcal{T}\}.
    \end{equation*}
    A standard (easy to check) fact about point-line duality is the following: the $\pi_{\zeta}$-projection of $P \subset \R^{2}$ equals the $y$-coordinates of the vertical $\zeta$-slice of the dual line family $\{\ell_{a,b} : (a,b) \in P\}$.
    In our situation, this principle implies that the $y$-coordinates of the squares in $(10\mathcal{T})_{\zeta}$ are contained in the $O(\delta)$-neighbourhood of $\pi_{\zeta}(\mathcal{P})$.
    In particular,
    \begin{equation}\label{form6}
        |(10\mathcal{T})_{\zeta}| \leq C|\pi_{\zeta}(\mathcal{P})|_{\delta} \leq C\delta^{-1 + \epsilon/2}, \qquad \zeta \in \Sigma,
    \end{equation}
    where $C > 0$ is absolute.
    We are now prepared to define the set $\mathfrak{P}$:
    \begin{equation}\label{form5a}
        \mathfrak{P} \coloneqq \bigcup_{\zeta \in \Sigma} \{q = [\zeta,\zeta + \delta) \times [b,b + \delta) \in \mathcal{D}_{\delta} : q \notin (10\mathcal{T})_{\zeta}\}.
    \end{equation} 
    In words, we include in $\mathfrak{P}$ all the dyadic $\delta$-squares around the vertical segments $\{\zeta\} \times [0,1]$, $\zeta \in \Sigma$, which are not contained in the $10\delta$-neighbourhood of any tube in $\mathcal{T}$.

    Now the set $\mathfrak{P}$ and the families $\mathcal{R}_{\sigma},\mathcal{T}_{\sigma}$ have been defined, so it remains to prove the claimed properties in \cref{p:bbb}, namely \cref{form3app} and (3).
    The property (2) is already clear from the construction of $\mathfrak{P}$.

    We start with \cref{form3app}.
    Let $\mathcal{Q} \subset \mathcal{D}_{\Delta}$ be arbitrary, and fix $t \in [1 + \epsilon,\tfrac{3}{2}]$.
    The claim is that $\mathcal{H}_{\infty}^{t - \epsilon}(\mathcal{Q} \cap \mathfrak{P}) \geq \mathcal{H}^{t}_{\infty}(\mathcal{Q})$.
    Since $\delta \leq \underline{\Delta} \leq \Delta$, this claim is formally equivalent to the claim $\mathcal{H}^{t - \epsilon}(\mathcal{D}_{\underline{\Delta}}(\mathcal{Q}) \cap \mathfrak{P}) \geq \mathcal{H}^{t}_{\infty}(\mathcal{D}_{\underline{\Delta}}(\mathcal{Q}))$.
    Replacing $\mathcal{Q}$ by $\mathcal{D}_{\underline{\Delta}}(\mathcal{Q})$ without altering notation, we will from now on assume that $\mathcal{Q} \subset \mathcal{D}_{\underline{\Delta}}$.

    Let $\mathcal{W} \subset \mathcal{D}$ be an arbitrary cover of $\mathcal{Q} \cap \mathfrak{P}$.
    We need to show that $\sum_{W \in \mathcal{W}} \ell(W)^{t - \epsilon} \geq \mathcal{H}^{t}(\mathcal{Q})$.
    Reducing $\mathcal{W}$ to its maximal elements, we may assume that $\mathcal{W}$ is disjoint.
    We may also assume that every $W \in \mathcal{W}$ intersects $Q \cap \mathfrak{P}$ for some $Q \in \mathcal{Q}$.
    Since $\mathfrak{P} \subset \mathcal{D}_{\delta}$, and $t < 2$, we may finally assume that $\ell(W) \geq \delta$ for all $W \in \mathcal{W}$ (as explained after \cref{def:content}).

    Let $\mathcal{Q}_{0} \subset \mathcal{Q} \subset \mathcal{D}_{\underline{\Delta}}$ be the squares in $\mathcal{Q}$ which are not strictly contained in any element of $\mathcal{W}$.
    We claim that
    \begin{equation}\label{form4a}
        \sum_{\substack{W \in \mathcal{W}\\W \subset Q}} \ell(W)^{t - \epsilon} \geq \ell(Q)^{t}, \qquad Q \in \mathcal{Q}_{0}, \, t \in [1 + \epsilon,\tfrac{3}{2}].
    \end{equation} 
    Let us see how this implies $\mathcal{H}^{t - \epsilon}_{\infty}(\mathcal{Q} \cap \mathfrak{P}) \geq \mathcal{H}^{t}_{\infty}(\mathcal{Q})$.
    Let 
    \begin{equation*}
        \mathcal{W}' \coloneqq \mathcal{Q}_{0} \cup \mathcal{W}'',
    \end{equation*}
    where $\mathcal{W}''$ consists of all the squares in $\mathcal{W}$ which are not contained in any element of $\mathcal{Q}$.
    We claim that $\mathcal{W}'$ is a cover of $\mathcal{Q}$.
    Indeed, if $Q \in \mathcal{Q}_{0}$, then $Q$ is evidently covered by $\mathcal{W}'$.
    On the other hand, if $Q \in \mathcal{Q} \, \setminus \, \mathcal{Q}_{0}$, then by definition $Q$ is strictly contained in some $W \in \mathcal{W}$.
    Then $\ell(W) > \ell(Q) = \underline{\Delta}$, so $W$ cannot be contained in any element of $\mathcal{Q}$, hence $W \in \mathcal{W}''$.
    Thus $Q \subset W \in \mathcal{W'}$.

    Now we know that $\mathcal{W}'$ is a cover of $\mathcal{Q}$.
    We next claim that that every element of $\mathcal{W}$ either lies in $\mathcal{W}''$, or then is contained in some element of $\mathcal{Q}_{0}$ (these two subsets of $\mathcal{W}$ are evidently disjoint).
    Indeed, if $W \in \mathcal{W}$, then $W \cap \mathcal{Q} \cap \mathfrak{P} \neq \emptyset$ for some $Q \in \mathcal{Q}$.
    Now, if $Q \subsetneq W$, it holds $W \in \mathcal{W}''$ ($\ell(W) > \ell(Q) = \underline{\Delta}$, so $W$ cannot be contained in any element of $\mathcal{Q}$).
    Alternatively, $W \subset Q$.
    But now $Q \in \mathcal{Q}_{0}$, because $Q$ cannot be strictly contained in any element of $\mathcal{W}$: if $Q$ was strictly contained in some $W' \in \mathcal{W}$, then $W \subsetneq W'$, violating the disjointness of $\mathcal{W}$.
    Thus $W \subset Q \in \mathcal{Q}_{0}$.

    Based on the previous claim, we can now decompose
    \begin{equation*}
        \sum_{W \in \mathcal{W}} \ell(W)^{t - \epsilon} = \sum_{W \in \mathcal{W}''} \ell(W)^{t - \epsilon} + \sum_{Q \in \mathcal{Q}_{0}} \sum_{\substack{W \in \mathcal{W}\\W \subset Q}} \ell(W)^{t - \epsilon} \stackrel{\cref{form4a}}{\geq} \sum_{W \in \mathcal{W}''} \ell(W)^{t} + \sum_{Q \in \mathcal{Q}_{0}} \ell(Q)^{t}.
    \end{equation*} 
    Since $\mathcal{Q}_{0} \cup \mathcal{W}'' = \mathcal{W}'$ is a cover of $\mathcal{Q}$, the right hand side is bounded from below by $\mathcal{H}^{t}_{\infty}(\mathcal{Q})$.
    This completes the proof of $\mathcal{H}^{t - \epsilon}_{\infty}(\mathcal{Q} \cap \mathfrak{P}) \geq \mathcal{H}^{t}_{\infty}(\mathcal{Q})$, assuming \cref{form4a}.

    We next prove \cref{form4a}.
    Fix $Q \coloneqq I \times J \in \mathcal{Q}_{0}$, where $I,J \in \mathcal{D}_{\underline{\Delta}}([0,1)$.
    Let us first record some properties of $Q \cap \mathfrak{P}$.
    By definition, recall \cref{form5a}, $\mathfrak{P}$ consists of certain squares of the form $[\zeta,\zeta + \delta) \times [b,b + \delta) \subset [0,1)^{2}$, where $\zeta \in \Sigma$ and $b \in \delta \Z$.
    By \cref{lemma1a}~\cref{P2}, the set $\Sigma$ is $\delta^{s}$-separated, and $|\Sigma \cap I| \sim \underline{\Delta} \delta^{-s}$.
    This implies that 
    \begin{equation}\label{form8a}
        |\mathfrak{P} \cap W| \lesssim
        \begin{cases}
            \ell(W)^{2} \delta^{-1 - s}, & \delta^{s} \leq \ell(W) \leq \underline{\Delta}, \\
            \ell(W)\delta^{-1}, & \delta \leq \ell(W) \leq \delta^{s},
        \end{cases}
        \qquad W \in \mathcal{D}.
    \end{equation} 
    We also need a lower bound for $|\mathfrak{P} \cap Q|$.
    Recall from \cref{form6} that $|(10\mathcal{T})_{\zeta}| \leq C\delta^{-1 + \epsilon/2}$ for all $\zeta \in \Sigma$.
    Therefore, 
    \begin{align}
        |\mathfrak{P} \cap Q| & \geq |\Sigma \cap I| \cdot ((\underline{\Delta}/\delta) - |(10\mathcal{T})_{\zeta}|) \notag\\
                              &\label{form12} \gtrsim \underline{\Delta} \delta^{-s} \cdot ((\underline{\Delta}/\delta) - C\delta^{-1 + \epsilon/2}) \stackrel{\cref{form7a}}{\gtrsim} \underline{\Delta}^{2}\delta^{-1 - s}.
    \end{align} 
    Now, to establish \cref{form4a}, note that $\mathfrak{P} \cap Q$ is covered by the squares $W \in \mathcal{W}$ with $W \subset Q$ (otherwise $Q$ would be strictly contained in some element of $\mathcal{W}$ contrary to the definition of $Q \in \mathcal{Q}_{0}$).
    The argument splits to two cases: (i) at least half of the squares in $\mathfrak{P} \cap Q$ are contained in squares $W \in \mathcal{W}$ with $\delta^{s} \leq \ell(W) \leq \underline{\Delta}$, or (ii) at least half of the squares in $\mathfrak{P} \cap Q$ are contained in squares $W \in \mathcal{W}$ with $\delta \leq \ell(W) \leq \delta^{s}$.

    Assume that case (i) occurs.
    Then, using $t - \epsilon - 2 \leq 0$, we estimate
    \begin{equation*}
        \sum_{\substack{W \in \mathcal{W}\\W \subset Q}} \ell(W)^{t - \epsilon} \geq \underline{\Delta}^{t - \epsilon - 2}\delta^{1 + s} \sum_{\substack{W \in \mathcal{W}\\W \subset Q, \delta^{s} \leq \ell(W) \leq \underline{\Delta}}} \ell(W)^{2}\delta^{-1 - s} \stackrel{\mathclap{\cref{form8a}-\cref{form12}}}{\gtrsim}\qquad\underline{\Delta}^{t - \epsilon} = \underline{\Delta}^{-\epsilon}\ell(Q)^{t}.
    \end{equation*}
    This gives \cref{form4a}, provided that the constant $C$ in \cref{form9a} was chosen sufficiently large.

    Assume next that case (ii) occurs.
    This time, using $1 + \epsilon \leq t \leq \tfrac{3}{2}$, we estimate
    \begin{align*}
        \sum_{\substack{W \in \mathcal{W}\\W \subset Q}} \ell(W)^{t - \epsilon}
    & \geq \delta^{t - \epsilon} \sum_{\substack{W \in \mathcal{W}\\W \subset Q, \delta \leq \ell(W) \leq \delta^{s}}} \ell(W)\delta^{-1}\\
    & \stackrel{\cref{form8a}-\cref{form12}}{\gtrsim} \delta^{t - \epsilon} \cdot \underline{\Delta}^{2}\delta^{- s - 1} = \underline{\Delta}^{2}\delta^{t - \epsilon/2 - 3/2} \geq \underline{\Delta}^{2}\delta^{-\epsilon/2}.
    \end{align*}
    Since $\underline{\Delta}^{2}\delta^{-\epsilon/2} \geq C\underline{\Delta}^{t}$ by \cref{form7a}, this yields \cref{form4a} in case (ii).
    The proof of \cref{form4a} (and \cref{form3app}) is therefore complete.

    We move on to the proof of \cref{p:bbb}~\cref{3}.
    Fix $\sigma \in \delta \Z \cap [0,1)$, $\mathcal{Q} \subset \mathcal{D}_{\Delta}$ and $t \in [1 + \epsilon,\tfrac{3}{2}]$.
    The claim is that $\mathcal{H}^{t - \epsilon}_{\infty}(\mathcal{Q} \cap \mathcal{D}_{\delta}(\cup \mathcal{R}_{\sigma}))) \geq \mathcal{H}^{t}_{\infty}(\mathcal{Q})$.
    The proof greatly resembles the proof of \cref{form3app}.
    Again, we may assume that $\mathcal{Q} \subset \mathcal{D}_{\underline{\Delta}}$.
    We abbreviate $\mathfrak{R}_{\sigma} \coloneqq \mathcal{D}_{\delta}(\cup \mathcal{R}_{\sigma})$.
    Let $\mathcal{W}$ be an arbitrary cover of $\mathcal{Q} \cap \mathfrak{R}_{\sigma}$ by dyadic cubes with side $\leq 1$.
    We may (again) assume that the elements in $\mathcal{W}$ are disjoint and have side-length $\geq \delta$.

    As in the proof of \cref{form3app}, let $\mathcal{Q}_{0} \subset \mathcal{Q}$ be the squares in $\mathcal{Q}$ which are not strictly contained in any element of $\mathcal{W}$.
    Our claim is formally the same as \cref{form4a}, namely
    \begin{equation}\label{form10a}
        \sum_{\substack{W \in \mathcal{W}\\W \subset Q}} \ell(W)^{t - \epsilon} \geq \ell(Q)^t, \qquad Q \in \mathcal{Q}_{0}, \, t \in [1 + \epsilon,\tfrac{3}{2}].
    \end{equation}
       This implies $\mathcal{H}^{t - \epsilon}_{\infty}(\mathcal{Q} \cap \mathfrak{R}_{\sigma}) \geq \mathcal{H}^{t}_{\infty}(\mathcal{Q})$ by repeating the argument below \cref{form4a}.

    Let us prove \cref{form10a}.
    Fix $Q \in \mathcal{Q}_{0} \subset \mathcal{D}_{\underline{\Delta}}$.
    We first record some properties of $Q \cap \mathfrak{R}_{\sigma}$.
    We already discussed earlier that the tubes $T \in \mathcal{T}_{\sigma}$ are $\sim \delta^{1/2}$-separated.
    This implies the following non-concentration condition for $\mathfrak{R}_{\sigma}$:
    \begin{equation}\label{form11a}
        |\mathfrak{R}_{\sigma} \cap W| \lesssim
        \begin{cases}
            \ell(W)^{2}\delta^{-3/2}, & \delta^{1/2} \leq \ell(W) \leq \underline{\Delta},
            \\ \ell(W)\delta^{-1}, & \delta \leq \ell(W) \leq \delta^{1/2},
        \end{cases}
        \qquad W \in \mathcal{D}, \, \sigma \in \delta \Z \cap [0,1].
\end{equation} 
    We also need a lower bound for $|\mathfrak{R}_{\sigma} \cap Q|$.
    For this purpose, we start by noting that if $\sigma \in \delta \Z \cap [0,1]$, then many elements of $\mathcal{T}_{\sigma}$ intersect $Q$ significantly.
    More precisely, there exists an interval $I_{Q} \subset [-2,2]$ of length $\ell(I_Q) \sim \underline{\Delta}$ such that 
       \begin{equation}\label{form16}
        b \in I_{Q} \quad \Longrightarrow \quad |T_{\sigma,b} \cap Q|_{\delta} \sim \underline{\Delta}/\delta,
    \end{equation} 
    where (as before) $T_{\sigma,b} = [\ell_{\sigma,b}]_{\delta/2}$, see \cref{fig4} for the idea.

    Now, we apply the ``density property'' of $\mathcal{P}_{\sigma}$ stated in \cref{lemma1a}~\cref{P1} to $I_{Q}$:
    \begin{equation*}
        |Y_{\sigma} \cap I_{Q}| \sim \underline{\Delta} \delta^{-1/2}, \qquad \sigma \in [0,1].
    \end{equation*}
    Now all the tubes $T = T_{\sigma, b} \in \mathcal{T}_{\sigma}$ with $b \in Y_{\sigma} \cap I_{Q}$ intersect $Q$ significantly.
    However, this is not enough to say anything about $Q \cap \mathfrak{R}_{\sigma}$, because $\mathfrak{R}_{\sigma} = \mathcal{D}_{\delta}(\cup \mathcal{R}_{\sigma})$ is defined as the union of the $(\Delta \times \delta)$-rectangles $R_{T} \subset T$, not the full tubes.
    To get around this, we have to recall from \cref{def:rectangles} how the rectangles $R_{T} \subset T$ were selected.
    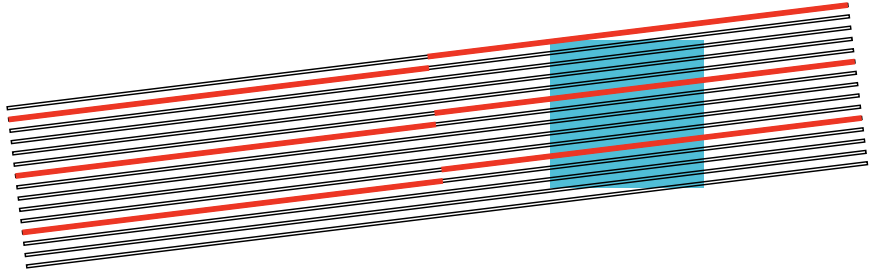
\begin{figure}[t]
        \begin{center}
            \input{figures/rect_int.tex}
            \caption{The square $Q \in \mathcal{D}_{\underline{\Delta}}$ (drawn in blue) intersects many rectangles $R_{T}$, $T \in \mathcal{T}_{\sigma}$ (drawn in red).}
            \label{fig4}
        \end{center}
    \end{figure}

    Recall the (increasing) enumeration $Y_{\sigma} = \{b_{1},\ldots,b_{n}\}$.
    Fix $b_{j} \in Y_{\sigma} \cap I_{Q}$ (so ``$T_{\sigma,b_{j}}$ intersects $Q$ significantly'' in the sense of \cref{form16}), and let $R_{\sigma,b_{j}} \subset T_{\sigma,b_{j}}$ be the (unique) element of $\mathcal{R}_{\sigma}$ contained in $T_{\sigma,b_{j}}$.
    Recall from \cref{def:rectangles} that 
    \begin{equation*}
        R_{\sigma,b_{j}} = T_{\sigma,b_{j}} \cap \bar{\pi}_{\sigma}^{-1}(2I_{i(j)})
    \end{equation*}
    for the interval $I_{i(j)} \in \{I_{1},\ldots,I_{m}\} \coloneqq \mathcal{D}_{\Delta/2}([-2,2])$ satisfying $j = i(j) (\mathrm{mod\, } m)$.

    Next note that the (hypothetical) inclusion $\bar{\pi}_{\sigma}(Q) \subset 2I_{i(j)}$ implies $R_{\sigma,b_{j}} \cap Q = T_{\sigma,b_{j}} \cap Q$, and therefore $R_{\sigma,b_{j}}$ intersects $Q$ significantly -- more precisely $|R_{\sigma,b_{j}} \cap Q|_{\delta} \sim \underline{\Delta}/\delta$.
    On the other hand, since $\ell(Q) = \underline{\Delta} \leq \tfrac{1}{100}\Delta$, there exists at least one index $i_{Q} \in \{1,\ldots,m\}$ such that $\bar{\pi}_{\sigma}(Q) \subset 2I_{i_{Q}}$.
    Summarising these two observations, we conclude the following: whenever $b_{j} \in Y_{\sigma} \cap I_{Q}$ is such that $i(j) = i_{Q}$, then $|R_{\sigma,b_{j}} \cap Q|_{\delta} \sim \underline{\Delta}/\delta$.

    Are there (m)any elements $b_{j} \in Y_{\sigma} \cap I_{Q}$ such that $i(j) = i_{Q}$? By the definition of ``$i(j)$'', this requirement is equivalent to $j = i_{Q} (\mathrm{mod\,} m)$.
    Finally, note that $Y_{\sigma} \cap I_{Q} = \{b_{k},\ldots,b_{l}\}$ for some $k,l \in \{1,\ldots,n\}$ with $l - k \sim |Y_{\sigma} \cap I_{Q}| \sim \underline{\Delta}\delta^{-1/2}$.
    Evidently,
    \begin{equation*}
        |\{b_{j} \in Y_{\sigma} \cap I_{Q} : j = i_{Q}(\mathrm{mod\,} m)\}| \sim (l - k)/m \sim \Delta\underline{\Delta} \delta^{-1/2} \stackrel{\cref{form9a}}{\geq} \underline{\Delta}^{1 + \epsilon/2}\delta^{-1/2}.
    \end{equation*}
    Therefore,
    \begin{equation*}
        |\{b_{j} \in Y_{\sigma} \cap I_{Q} : |R_{\sigma,b_{j}} \cap Q|_{\delta} \sim \underline{\Delta}/\delta\}| \gtrsim \underline{\Delta}^{1 + \epsilon/2}\delta^{-1/2},
    \end{equation*}
    and finally
    \begin{equation}\label{form25}
        |\mathfrak{R}_{\sigma} \cap Q| \gtrsim |\{b_{j} \in Y_{\sigma} \cap I_{Q} : |R_{\sigma,b_{j}} \cap Q|_{\delta} \sim \underline{\Delta}/\delta\}| \cdot (\underline{\Delta}/\delta) \gtrsim \underline{\Delta}^{2 + \epsilon/2}\delta^{-3/2}.
    \end{equation} 
    The upper bound \cref{form11a} and the lower bound above are the same as in \cref{form8a}-\cref{form12}, except that $s = (1 - \epsilon)/2$ has been replaced by $1/2$, and the lower bound has the additional (small) factor $\underline{\Delta}^{\epsilon/2}$.
    Now virtually the same argument as in the proof of \cref{form4a} yields \cref{form10a}.
    We repeat the details for completeness.
    First, assume that at least half of the squares in $\mathfrak{R}_{\sigma} \cap Q$ are contained in such squares $W \in \mathcal{W}$ with $\delta^{1/2} \leq \ell(W) \leq \underline{\Delta}$.
    Then,
    \begin{align*}
        \sum_{\substack{W \in \mathcal{W}\\W \subset Q}} \ell(W)^{t - \epsilon}
    & \geq \underline{\Delta}^{t - \epsilon - 2}\delta^{3/2} \sum_{\substack{W \in \mathcal{W}\\W \subset Q, \delta^{1/2} \leq \ell(W) \leq \underline{\Delta}}} \ell(W)^{2}\delta^{-3/2} \stackrel{\cref{form11a} + \cref{form25}}{\gtrsim} \underline{\Delta}^{t - \epsilon/2}.
    \end{align*} 
    This yields \cref{form10a}, using \cref{form9a}, and recalling $\ell(Q) = \underline{\Delta}$.

    Finally, assume that at least half of the squares in $\mathfrak{R}_{\sigma} \cap Q$ are contained in squares $W \in \mathcal{W}$ with $\delta \leq \ell(W) \leq \delta^{1/2}$.
    Then, recalling that $t \in [1 + \epsilon,\tfrac{3}{2}]$, 
    \begin{align*}
        \sum_{\substack{W \in \mathcal{W}\\W \subset Q}} \ell(W)^{t - \epsilon}
    & \geq \delta^{t - \epsilon} \sum_{\substack{W \in \mathcal{W}\\W \subset Q, \delta \leq \ell(W) \leq \delta^{1/2}}} \ell(W)\delta^{-1}\\
    & \stackrel{\mathclap{\cref{form11a} + \cref{form25}}}{\gtrsim}\qquad\underline{\Delta}^{2 + \epsilon/2}\delta^{t - \epsilon - 3/2} \geq \underline{\Delta}^{2 + \epsilon/2}\delta^{-\epsilon/2}.
    \end{align*}
    The right hand side is $\geq C\underline{\Delta}^{t}$ by \cref{form7a}.
    This completes the proof of \cref{form10a}, and the proof of \cref{p:bbb}.
\end{proof}
\end{document}

%% file: figures/grid.tex
\begin{tikzpicture}[x=1cm,y=1cm, scale=0.7]
    \foreach \i in {0,...,7} {
        \foreach \j in {0,...,7} {
            \fill
                ({0.5*\i-0.08},{0.5*\j-0.08})
                rectangle ++(0.16,0.16);
        }
    }

    \draw[densely dotted,->,>=stealth,line width=0.9pt]
        (3.88,1.80)
        .. controls (4.42,2.30) and (5.18,2.30) ..
        node[pos=0.5,above=1.5mm] {rotate}
        (5.72,1.78);

    \foreach \i in {0,...,7} {
        \foreach \j in {0,...,7} {
            \fill
                ({6.0+0.5*\i+\j/6-0.08},{0.5*\j-0.08})
                rectangle ++(0.16,0.16);
        }
    }

    \draw[red,densely dotted,line width=0.55pt]
        (8.1667,-0.13) -- (8.1667,3.63);
\end{tikzpicture}

%% file: figures/rectangle_choice.tex
\begin{tikzpicture}[x=1cm,y=1cm, scale=0.7]
    \definecolor{intervalblue}{RGB}{157,187,216}
    \definecolor{intervalred}{RGB}{237,54,36}

    \begin{scope}[rotate=22.2]
        \foreach \r in {1,...,16} {
            \pgfmathsetmacro{\rowy}{-0.34*(\r-1)}
            \pgfmathtruncatemacro{\redcolumn}{mod(\r-1,4)}

            \draw[fill=intervalblue,draw=black,line width=0.35pt]
                (0,{\rowy-0.09}) rectangle (10,{\rowy+0.09});
            \fill[intervalred]
                ({2.5*\redcolumn-0.005},{\rowy-0.095})
                rectangle
                ({2.5*(\redcolumn+1)+0.005},{\rowy+0.095});

            \ifnum\r<7
                \node[left, font=\tiny, yshift=-1mm]
                    at (0,\rowy) {$b_{\r}$};
            \fi

            \ifnum\r>6
                \ifnum\r<10
                    \node[left, font=\tiny, yshift=-1mm]
                        at (0,\rowy) {$\cdot$};
                \fi
            \fi
        }

        \draw[line width=0.7pt] (0,0.62)
            -- node[above]{$I_1$} (2.5,0.62)
            -- node[above]{$I_2$} (5,0.62)
            -- node[above]{$I_3$} (7.5,0.62)
            -- node[above]{$I_4$} (10,0.62);

        \draw[densely dotted,line cap=round,line width=0.65pt]
            (0,0.62) -- (0,-0.095);
        \draw[densely dotted,line cap=round,line width=0.65pt]
            (2.5,0.62) -- (2.5,-0.435);
        \draw[densely dotted,line cap=round,line width=0.65pt]
            (5,0.62) -- (5,-0.775);
        \draw[densely dotted,line cap=round,line width=0.65pt]
            (7.5,0.62) -- (7.5,-1.115);
        \draw[densely dotted,line cap=round,line width=0.65pt]
            (10,0.62) -- (10,-1.115);
    \end{scope}
\end{tikzpicture}

%% file: figures/rect_int.tex
\begin{tikzpicture}[x=1cm,y=1cm, scale=0.7]
    \definecolor{intervalred}{RGB}{237,54,36}
    \definecolor{squareblue}{RGB}{80,190,215}

    \fill[squareblue]
        (9.90,1.48) rectangle (12.80,4.26);

    \begin{scope}[rotate=7]
        \foreach \r in {0,...,14} {
            \pgfmathsetmacro{\y}{0.215*\r}
            \draw[black,line width=0.55pt]
                (0,{\y-0.0275}) rectangle
                (16,{\y+0.0275});
        }

        \foreach \r/\xa/\xb in {
            3/0/8,   4/8/16,
            8/0/8,   9/8/16,
            13/0/8, 14/8/16
        } {
            \pgfmathsetmacro{\y}{0.215*\r}
            \draw[
                intervalred,
                line width=2.15pt,
                line cap=butt
                ] (\xa,\y) -- (\xb,\y);
        }
    \end{scope}
\end{tikzpicture}